\documentclass[twoside,10pt,a4paper,leqno]{amsart}

\usepackage{amsmath,amsfonts,amsthm,enumitem}
\usepackage[colorlinks=true,linkcolor=blue,citecolor=blue,%
filecolor=blue,urlcolor=blue,pageanchor=true,plainpages=false,%
linktocpage]{hyperref}

\numberwithin{equation}{section}
\newtheorem{theorem}{Theorem}[section]
\newtheorem{proposition}[theorem]{Proposition}
\newtheorem{corollary}[theorem]{Corollary}
\newtheorem{lemma}[theorem]{Lemma}
\theoremstyle{definition}
\newtheorem{definition}[theorem]{Definition}
\newtheorem{remark}[theorem]{Remark}
\newtheorem{example}[theorem]{Example}

\newcommand{\N}{\mathbb{N}}
\newcommand{\Z}{\mathbb{Z}}
\newcommand{\R}{\mathbb{R}}
\newcommand{\Om}{\Omega}
\newcommand{\Rb}{\overline{\R}}

\newcommand{\dvg}{\mathrm{div}\,}

\newcommand{\essinf}{\operatornamewithlimits{\mathrm{ess\,inf}}}
\newcommand{\leb}{\mathcal{L}}
\newcommand{\bor}{\mathcal{B}}
\newcommand{\cmeas}{\mathcal{M}_0^p}
\newcommand{\argmin}{\operatornamewithlimits{\mathrm{arg\,min}}}
\newcommand{\idx}[1]{i\left(#1\right)}
\newcommand{\cp}{\mathrm{cap}}
\newcommand{\mup}{\nu_1}
\newcommand{\mum}{\nu_2}
\makeatletter
\newcommand{\mylabel}[2]{#2\def\@currentlabel{#2}\label{#1}}
\makeatother

\newcommand{\hi}{$(i)$}
\newcommand{\hii}{$(ii)$}

\newcommand{\his}{$(is)$}
\newcommand{\hiis}{$(iis)$}
\newcommand{\hiiis}{$(iiis)$}
\newcommand{\hivs}{$(ivs)$}
\newcommand{\hvs}{$(vs)$}

\newcommand{\hei}{$(ie)$}
\newcommand{\heii}{$(iie)$}
\newcommand{\ha}{$(a)$}
\newcommand{\hb}{$(b)$}
\newcommand{\hc}{$(c)$}
\newcommand{\hd}{$(d)$}
\newcommand{\he}{$(e)$}
\newcommand{\muW}{$({\mu}W)$}
\newcommand{\AW}{$(AW)$}

\renewcommand{\rho}{\varrho}
\renewcommand{\theta}{\vartheta}

\newcommand{\la}{\lambda}
\newcommand{\eps}{\varepsilon}
\begin{document}

%TOPMATTER

\title[Optimization results for the higher eigenvalues]{Optimization
results for the higher eigenvalues \\
of the $p$-Laplacian associated with \\
sign-changing capacitary measures}

\author{Marco Degiovanni}
\address{Dipartimento di Matematica e Fisica\\
         Universit\`a Cattolica del Sacro Cuore\\
         Via Trieste 17\\
         25121 Bre\-scia, Italy}
\email{marco.degiovanni@unicatt.it}
\author{Dario Mazzoleni}
\address{Dipartimento di Matematica F. Casorati\\
         Universit\`a di Pavia\\
         Via Ferrata 5\\
         27100 Pavia, Italy}
\email{dario.mazzoleni@unipv.it}

\thanks{The authors are members of the 
        Gruppo Nazionale per l'Analisi Matematica, la Probabilit\`a
				e le loro Applicazioni (GNAMPA) of the 
				Istituto Nazionale di Alta Matematica (INdAM).
				They have been partially supported by the 2019 INdAM-GNAMPA
				project ``Ottimizzazione spettrale non lineare''}



%END TOPMATTER

%--------------------------------------------------------------------

%
\begin{abstract}
In this paper we prove the existence of an optimal set for the 
minimization of the $k$-th variational eigenvalue of the 
$p$-Laplacian among $p$-quasi open sets of fixed measure included 
in a box of finite measure. 
An analogous existence result is obtained for eigenvalues of the 
$p$-Laplacian associated with Schr\"odinger potentials.
In order to deal with these nonlinear shape optimization problems, 
we develop a general approach which allows to treat the 
continuous dependence of the eigenvalues of the $p$-Laplacian 
associated with sign-changing capacitary measures under 
$\gamma$-convergence. 
\end{abstract}
\maketitle %AMSART
\tableofcontents

%--------------------------------------------------------------------
\section{Introduction}\label{sect:intro}
In the last few years, shape optimization problems for the 
eigenvalues of the Dirichlet-Laplacian have been a very studied 
topic in many fields of mathematics, see~\cite{henrotbook} for a 
general overview.
Recently, there has been an interest in extending these results 
also to the case of the eigenvalues of the $p$-Laplacian for 
$p\not=2$ (often called \emph{nonlinear} eigenvalues).
Given an open subset $\Om$ of $\R^N$ with finite measure and 
$1 < p < \infty$, we say that $\la>0$ is an eigenvalue of the 
$p$-Laplacian if there is a nonzero weak solution $u$, called 
eigenvector, of the problem
\[
\begin{cases}
- \dvg(|\nabla u|^{p-2}\nabla u)=\la|u|^{p-2}u
&\qquad \text{in $\Om$}\,,\\
\noalign{\medskip}
u=0
&\qquad \text{on $\partial \Om$}\,.
\end{cases}
\]
The eigenvalues can be characterized as the critical values of 
the functional 
\[
f : W^{1,p}_0(\Om)\rightarrow \R\,,\qquad 
f(u)=\int_{\Om}|\nabla u|^p\,d\leb^N\,,
\]
on the manifold 
$\displaystyle{M=\left\{u\in W^{1,p}_0(\Om):\,\, 
\int_{\Om}|u|^p\,d\leb^N = 1\right\}}$.
The first eigenvalue can be proved to be a minimum, while 
higher eigenvalues (if $p\not=2$) are less understood.
More precisely, one can obtain a nondecreasing sequence of 
eigenvalues by the minimax procedure 
\begin{equation}
\label{eq:vareig}
\la_m^p(\Om)=\inf_{K\in\mathcal{K}_m}\;\sup_{u\in K}\,f(u)
\qquad\text{for all integer $m\geq 1$}\,,
\end{equation}
where $\mathcal{K}_m$ denotes the collection of compact and 
symmetric subsets $K$ of $M$ such that $i(K)\geq m$ and~$i$ 
denotes a suitable index, e.g. Krasnosel’skii genus, 
(see~\cite{garcia_peral1987}).
Unfortunately, it is still a major open problem to understand 
if all the eigenvalues of the $p$-Laplacian are of this form. 
In the present paper we focus on the ``variational'' eigenvalues 
arising from the minimax procedure described above.
\par
A first shape optimization result for these eigenvalues 
was recently obtained by Fusco, Mukherjee and Zhang 
in~\cite[Theorem~1.2]{fuscop}.
\begin{theorem}[Fusco-Mukherjee-Zhang]
\label{thm:fuscoopt}
Let $1<p<\infty$, $\Om$ be a bounded and open subset of $\R^N$,
$c\in ]0,\leb^N(\Om)]$ and $F:\R^2\rightarrow \R$ be a 
function nondecreasing in each variable and lower semicontinuous.
\par
Then the problem
\begin{equation}
\min{\left\{F(\la_1^p(A),\la_2^p(A)):\,\, 
\text{$A$ is a $p$-quasi open subset of $\Om$ with 
$\leb^N(A)=c$}\right\}}
\end{equation}
admits a solution.
\end{theorem}
We note that, also when $A$ is only a $p$-quasi open set, it is 
possible to define the space $W^{1,p}_0(A)$ and then the
variational eigenvalues $\la_m^p(A)$ again by~\eqref{eq:vareig}
\par
The main aim of this paper is to extend this existence result 
also to higher nonlinear variational eigenvalues and to 
nonlinear eigenvalues associated with Schr\"odinger potentials.
Actually, we follow a unified approach based on the concept
of \emph{$p$-capacitary measure}.
The reason for which the above existence result was proved only 
for the first two eigenvalues is that a lower semicontinuity 
result for nonlinear eigenvalues with respect to an appropriate 
convergence was not known, as the typical arguments, relating
convergence of quasi open sets and linear operators 
(see e.g.~\cite[Chapter~6]{bucur_buttazzo2005}), cannot be 
adapted to the case $p\neq 2$.
Let us collect the key results (see 
\cite[Corollary~4.5 and Proposition~4.6]{fuscop}) involved in
the proof of Theorem~\ref{thm:fuscoopt}. 
\begin{theorem}[Fusco-Mukherjee-Zhang]
\label{thm:fuscolsc}
Let $\Om$ be a bounded and open subset of $\R^N$
and $(A_n)$ be a sequence of $p$-quasi open sets 
$\gamma$-converging to a $p$-quasi open set $A$ in $\Om$.
\par
Then we have
\[
\la_m^p(A) \leq \liminf_{n\to\infty} \,\la_m^p(A_n)
\qquad\text{for $m=1,2$}\,.
\]
\end{theorem}
In particular, the technique used in~\cite{fuscop} in order to
show the desired lower semicontinuity for $m=2$ is based on a 
mountain pass characterization of the second eigenvalue of the 
$p$-Laplacian on $p$-quasi open sets, and extending this approach  
to the case $m\geq 3$ seems out of reach.
Thus, a key issue is to understand the continuity properties of 
the higher nonlinear variational eigenvalues with respect to the 
$\gamma$-convergence of $p$-quasi open sets.
\par
In this paper we investigate in depth this question, developing 
the more general framework of $p$-capacitary measures, in which a 
lower semicontinuity result for nonlinear eigenvalues under 
$\gamma$-convergence is proved, see Proposition~\ref{prop:muW}.
Then, in Theorems~\ref{thm:existmuV} and~\ref{thm:existmuA},
we provide some related optimization results, still in the
setting of $p$-capacitary measures.
\par
In the special case of $p$-quasi open sets, we deduce that 
Theorem~\ref{thm:fuscolsc} holds for all $m\geq 1$, 
see Corollary~\ref{cor:AW}.
As a consequence of Theorem~\ref{thm:existmuA}, we can 
prove the following extension of Theorem~\ref{thm:fuscoopt} 
to higher nonlinear variational eigenvalues.
This result provides also a theoretical setting for performing 
numerical simulations, as recently done in~\cite{antunes}.
\begin{theorem}
\label{thm:main1}
Let $k\geq 1$, $1<p<\infty$, $\Om$ be an open subset of $\R^N$
with finite measure, $c\in ]0,\leb^N(\Om)]$ and 
$F:\R^k\rightarrow \R$ be a function nondecreasing in each 
variable and lower semicontinuous.
\par
Then the problem
\begin{equation}
\label{sobox}
\min{\left\{F(\la_1^p(A),\dots,\la_k^p(A)):\,\, 
\text{$A$ is a $p$-quasi open subset of $\Om$ with 
$\leb^N(A)=c$}\right\}}
\end{equation}
admits a solution.
\end{theorem}
We now briefly describe the motivations for working in a class 
wider than $p$-quasi open sets and the other new results that 
we obtain.  
First of all, the works from the 1980s and 1990s of Buttazzo, 
Dal Maso, Mosco, Murat~\cite{bdm,dalmaso1987,dmmo,dmmu} 
suggest that the natural setting for spectral problems in the 
line of~\eqref{sobox} is the space of $p$-capacitary measures, 
i.e. Borel measures in $\Om$ that vanish on sets of zero 
$p$-capacity.
One can consider $\la$ to be an eigenvalue associated with the 
$p$-capacitary measure $\mu$ if there is a nonzero solution~$u$
of the problem
\begin{equation}
\label{eq:eigintro}
\begin{cases}
-\Delta_pu+|u|^{p-2}u\,\mu=\la|u|^{p-2}u
&\qquad\text{in $\Om$}\,,\\
\noalign{\medskip}
u=0
&\qquad\text{on $\partial \Om$}\,,
\end{cases}
\end{equation}
where the formal writing above should be read through the 
variational formulation described in~\cite{dmmu}.
On the other hand, also on the right hand side of the eigenvalue 
equation~\eqref{eq:eigintro} things can become more complicated 
and more interesting. 
In particular, the study of eigenvalues with an $L^\infty$ 
sign-changing weight on the right hand side arises naturally in
many problems from population dynamics (see~\cite{cantrellcosner} 
for an overview) and the existence of eigenvalues of the 
$p$-Laplacian was studied in~\cite{sw} in the sign-changing case.
We generalize also this sign-changing weight on the right hand 
side to be the difference of two non-negative $p$-capacitary 
measures and we set the problem in the whole $\R^N$ (with some 
additional assumptions on the measures).
Summing all up, given three (non-negative) $p$-capacitary 
measures $\mu,\nu_1,\nu_2$, we study the variational 
eigenvalues $\la^p_m(\mu,\nu_1,\nu_2)$ of the problem
\[
\begin{cases}
-\Delta_pu+|u|^{p-2}u\,\mu = \la |u|^{p-2}u(\nu_1-\nu_2)
\qquad\text{in $\R^N$}\,,\\
\noalign{\medskip}
\displaystyle{
\int |u|^p\,d\nu_2 < \int |u|^p\,d\nu_1} \,,
\end{cases}
\]
with a homogeneous Dirichlet-type condition at infinity, noting 
that, in order to set the problem in a bounded and open subset 
$\Om$ of $\R^N$, it is enough to replace $\mu$ with
$\infty_{\R^N\setminus \Om} + \mu$.
The motivation for considering the case of $\R^N$ as ambient 
space is in view of a possible extension of the existence 
Theorem~\ref{thm:main1} for nonlinear spectral functionals to 
the case $\Om=\R^N$, which is a difficult open problem that we plan 
to investigate in the future and that has been only recently 
solved in the case $p=2$ (see~\cite{bucur,bm,mp}).
\par
Thanks to the general theory developed, we can also prove an
extension to nonlinear eigenvalues 
of~\cite[Theorem~4.1]{bugeruve} which deals, in the case $p=2$,
with the optimization of Schr\"odinger potentials, that is, of 
capacitary measures absolutely continuous with respect to the 
Lebesgue measure in $\R^N$.
\begin{theorem}
\label{thm:main2}
Let $k\geq 1$, $1<p<\infty$, $\Om$ be an open subset of $\R^N$ 
with finite measure, $\nu$ be a $p$-capacitary measure in $\Om$, 
$\Psi: [0,+\infty]\rightarrow[0,+\infty]$ be a strictly 
decreasing and continuous function such that there exists 
$\alpha>1$ with $s\mapsto \Psi^{-1}(s^\alpha)$ convex on 
$\{s\geq 0 :\,\, s^\alpha\in \Psi([0,+\infty])\}$,
\[
0 < c \leq \Psi(0)\,\leb^N(\Omega)
\]
and $F: \R^k\rightarrow \R$ be a function nondecreasing in 
each variable and lower semicontinuous.
Denote by $\mathcal{V}$ the set of $\leb^N$-measurable 
functions $V:\Omega\rightarrow[0,+\infty]$ such that 
\[
\int_\Om \Psi(V)\,d\leb^N \leq c
\]
and such that there exists $u\in W^{1,p}_0(\Om)$ satisfying
\[
\int |u|^p\,V\,d\leb^N < +\infty\,,\qquad
\int |u|^p\,d\nu < \int |u|^p\,d\leb^N \,.
\]
If $\mathcal{V}\neq\emptyset$, then there exists a minimizer 
$V$ for the problem
\begin{equation}\label{eq:pbpot}
\min\left\{F(\la_1^p(V),\dots,\la_k^p(V)):\,\, 
V\in \mathcal{V}\right\}
\end{equation}
satisfying
\[
\int_\Om \Psi(V)\,d\leb^N = c\,,
\]
where $\la_m^p(V)$ is associated with
\[
\begin{cases}
-\Delta_pu + V|u|^{p-2}u = 
\la|u|^{p-2}u - \la|u|^{p-2}u \nu
&\qquad\text{in $\Om$}\,,\\
\noalign{\medskip}
u=0
&\qquad\text{on $\partial \Om$}\,,\\
\noalign{\medskip}
\displaystyle{
\int_\Omega |u|^p\,d\nu < \int_\Omega |u|^p\,d\leb^N} \,,
\end{cases}
\]
according to Section~\ref{sect:eigmeas}.
\end{theorem}
The most interesting examples of the function $\Psi$ for which the 
assumptions of the above theorem hold are $\Psi(s)=e^{-\beta s}$ 
for all $\beta>0$ and $\Psi(s)=s^{-\beta}$ for all $\beta>0$.
\par
The abstract theory developed in this paper 
allows us also to  prove an upper semicontinuity result for 
nonlinear eigenvalues of $p$-capacitary measures, under very 
mild assumptions, Theorem~\ref{thmusc}. 
Though this is not needed for the shape optimization problem that 
was our motivation, we believe it is a very important property 
and it involves an interesting reduction to finite dimensional 
spaces in the inf-sup procedure.
Moreover, it could be useful when dealing with spectral problems 
with non-monotone functional.
%}
\par\medskip
The paper is organized as follows.
After the Introduction and Section~\ref{sect:prelim}, where we 
recall the main notions of $p$-capacity, $p$-quasi open set,
$p$-fine topology and $\Gamma$-convergence, the paper is 
divided into an abstract and an applied part.
\par
The abstract part is developed in Sections~\ref{sect:convminmax} 
and~\ref{sect:nep}, where first we study the behavior of sup 
functionals and of inf-sup values in a topological vector space 
and we prove suitable lower and upper semicontinuity results 
under $\Gamma$-convergence, then we study nonlinear eigenvalue 
problems involving a sign-changing weight in a reflexive
Banach space. 
\par
The applied part of the paper is organized as follows. 
Section~\ref{sect:convergence} is devoted to the study first of
convergence properties of $p$-capacitary measures and then of 
convergence of related functionals in $L^p_{loc}(\R^N)$, in the 
line of~\cite{dmmo}.
In Section~\ref{sect:preeig} we define, in the $L^p_{loc}(\R^N)$ 
setting, the variational eigenvalues involving sign-changing 
$p$-capacitary measures, we provide general conditions for 
existence of a sequence of (finite) variational eigenvalues and 
we provide an inf-sup characterization by means of suitable 
finite dimensional spaces.
Section~\ref{sect:semicontinfsup} is devoted to the study,
still in the $L^p_{loc}(\R^N)$ setting, of lower 
and upper semicontinuity properties of the variational 
eigenvalues defined in Section~\ref{sect:preeig}.
So far, the variational eigenvalues are just inf-sup values.
In Section~\ref{sect:eigmeas} we prove that they can be also 
defined with respect to a suitable reflexive Banach space where 
the results of Section~\ref{sect:nep} apply.
In particular, each inf-sup value is an eigenvalue with
a corresponding eigenvector.
Finally, in Section~\ref{sect:potentials}, we apply the theory 
developed in the previous sections to the case of $p$-quasi open 
sets and of Schr\"{o}dinger potentials, thus proving the main 
results of the paper, Theorems~\ref{thm:main1} and~\ref{thm:main2}.

%--------------------------------------------------------------------

\section{Notations and preliminaries}
\label{sect:prelim}
Throughout the paper, we fix an integer $N\geq 1$ and $1<p<\infty$.
We denote by $\leb^N$ the $N-$dimensional Lebesgue measure and,
if $p<N$, by $p^*=\frac{Np}{N-p}$ the critical Sobolev exponent.
We will usually write $\int$ instead of $\int_{\R^N}$.
For every real number~$s$, we denote by $s^\pm:=\max\{\pm s,0\}$ 
its positive and negative parts.
If $(X,d)$ is a metric space, we set
$B_r(x):=\{y\in X:\,\, d(y,x)<r\}$ and we denote by $\bor(X)$
the family of Borel subsets of $X$.
\paragraph{\bf{Capacity, quasi open sets and fine topology}}
We need to introduce the notion of $p$-capacity; we refer 
to~\cite[Chapter~2]{hekima_book}, to~\cite{hekima} and to~\cite{km} for more 
details. 
\begin{definition}
For every subset $E$ of $\R^N$, the $p$-capacity of $E$ in 
$\R^N$ is defined as 
\begin{multline*}
\cp_p(E):=\inf\,\biggl\{
\int \left(|\nabla u|^p+|u|^p\right)\,d\leb^N:\,\,
u\in W^{1,p}(\R^N)\,, 
\biggr. \\ \biggl.
\text{$0\leq u\leq 1$ $\leb^N$-a.e. on $\R^N$,
$u = 1$ $\leb^N$-a.e. on an open set containing $E$}\biggr\}\,,
\end{multline*}
where we agree that $\inf\emptyset=+\infty$.
If $E\subseteq \R^N$, we say that a property $\mathcal P(x)$ 
holds \emph{$\cp_p$-quasi everywhere in $E$}, if it holds for 
all $x\in E$ except at most a set of zero $p$-capacity.
We will write q.e. in $E$ instead of $\cp_p$-quasi everywhere
in $E$, for the sake of simplicity.
\end{definition}
\begin{definition}\label{def:quasiopen}
A subset $A$ of $\R^N$ is said to be \emph{$p$-quasi open} if, 
for every $\eps>0$, there exists an open subset $\omega_\eps$ 
of $\R^N$ such that $\cp_p(\omega_\eps) <\eps$ and 
$A\cup \omega_\eps$ is open in $\R^N$.
\end{definition}
\begin{remark}\label{rem:quasiopen}
First of all, we note that the open sets $\omega_\eps$ in the 
above definition can be chosen to be nondecreasing, i.e. if 
$\eps_1\leq\eps_2$, then 
$\omega_{\eps_1}\subseteq\omega_{\eps_2}$.
Then, for every $p$-quasi open subset $A$ of $\R^N$, it is 
possible to check from the definition that there exist two 
Borel and $p$-quasi open sets $G_1,G_2$ and two sets of zero 
$p$-capacity $E_1,E_2$ such that $A=G_1\cup E_1=G_2\setminus E_2$.
For example, with $\omega_\eps$ as in 
Definition~\ref{def:quasiopen}, one can take 
$G_1 = A\setminus (\bigcap_{j\in\N}\omega_{1/j})
= \bigcup_{j\in\N}((A\cup\omega_{1/j})\setminus \omega_{1/j})$, 
$E_1=A\cap (\bigcap_{j\in\N}\omega_{1/j})$, 
$G_2=\cap_{j\in\N}(A\cup\omega_{1/j})$ 
and $E_2=G_2\setminus A$.
\end{remark}
\begin{definition}
A function $u: \R^N\rightarrow \overline \R$ is said to be 
\emph{$p$-quasi continuous} (\emph{$p$-quasi lower semicontinuous,
$p$-quasi upper semicontinuous}, resp.) if for every $\eps>0$ 
there exists an open subset $\omega_\eps$ of $\R^N$ with 
$\cp_p(\omega_\eps)<\eps$ such that 
$u\bigl|_{\R^N\setminus \omega_\eps}$ is continuous 
(lower semicontinuous, upper semicontinuous, 
resp.).
\end{definition}
\begin{remark}\label{rem:quasiusc}
It can be proved (see~\cite[Proposition~IV.2 (d)]{brelot} 
for the case $p=2$) that  a function  
$u: \R^N\rightarrow \overline \R$ 
is $p$-quasi lower (resp. $p$-quasi upper) semicontinuous 
if and only if the sets $\{x\in\R^N : u(x)>t\}$ 
(resp. $\{x\in\R^N : u(x)<t\}$) are $p$-quasi open for 
every $t\in\R$.
\end{remark}
For every $u\in W^{1,p}_{loc}(\R^N)$, there exists a 
Borel and $p$-quasi continuous representative 
$\tilde{u}:\R^N\rightarrow\R$ of~$u$ and, if $\tilde{u}$ and 
$\hat{u}$ are two $p$-quasi continuous representatives of 
the same $u$, then we have $\tilde{u}=\hat{u}$ q.e. in~$\R^N$.
In the following, for every $u\in W^{1,p}_{loc}(\R^N)$,
we will consider only its Borel and $p$-quasi continuous 
representatives.
\begin{definition}
If $A$ is a $p$-quasi open subset of $\R^N$, we set
\[
W^{1,p}_0(A):=\left\{u\in W^{1,p}(\R^N) :\,\,
\text{$u=0$ q.e. in $\R^N\setminus A$}\right\}\,.
\]
\end{definition}
It turns out that the above definition is naturally equivalent 
to the usual one, if $A$ is an open subset of $\R^N$.
In the following, we also denote by $W^{1,p}_c(\R^N)$ the
set of $u$'s in $W^{1,p}(\R^N)$ vanishing q.e. outside some
compact subset of $\R^N$.
\par
From now on in this paragraph, we restrict ourselves to the 
case $p\leq N$, since if $p>N$ every point $x\in\R^N$ has 
positive $p$-capacity, thus $p$-quasi open sets coincide with 
Euclidean open sets.
\par
Although $p$-quasi open subsets do not form a topology on $\R^N$ 
(because an uncountable union of $p$-quasi open subsets is not 
always $p$-quasi open), it is possible to define the $p$-fine 
topology, which turns out to be a useful tool from nonlinear 
potential theory. 
In the present work we recall only the basic notions and 
properties that we need, and refer to~\cite{hekima,hekima_book} 
and the references therein for more details.
\begin{definition}\label{def:finelyopen}
A subset $W$ of $\R^N$ is said to be \emph{$p$-finely open} 
if for every $x\in W$ we have 
\[
\int_{0}^1\left(\frac{\cp_p(B_r(x)\setminus W)}{r^{N-p}} 
\right)^{\frac{1}{p-1}}\frac{dr}{r} < +\infty\,.
\]
The $p$-finely open subsets form a topology called the
\emph{$p$-fine topology}, which can be equivalently defined as the 
coarsest topology making all $p$-superharmonic functions continuous.
\end{definition}
We recall now the properties of the $p$-fine topology we need.
In particular, we refer to~\cite[Theorem~2.3]{hekima}
for the quasi-Lindel\"of property.
\begin{proposition}
\label{prop:proppfine}
The $p$-fine topology has the \emph{quasi-Lindel\"of} property, 
that is: for each family $\mathcal W$ of $p$-finely open sets, 
there is a countable subfamily $\mathcal W'$ such that  
\begin{equation}\label{eq:quasilindelof}
\cp_p\left(\bigcup_{W\in\mathcal W}W\setminus 
\bigcup_{W\in\mathcal W'}W\right)=0\,.
\end{equation}
Moreover, for every subset $A$ of $\R^N$, the following are 
equivalent:
\begin{enumerate}[label={\upshape\alph*)}, align=parleft, 
widest=iii, leftmargin=*]
\item[\mylabel{haproppfine}{\ha}]
$A$ is Borel and $p$-quasi open;
\item[\mylabel{hbproppfine}{\hb}]
$A=W\cup E$, with $W, E$ Borel, $W$ $p$-finely open and 
$\cp_p(E)=0$;
\item[\mylabel{hcproppfine}{\hc}] 
there exists $u\in W^{1,p}_{loc}(\R^N)$ such that $A=\{u>0\}$.
\end{enumerate}
\end{proposition}
\begin{remark}
\label{rem:quasiopenfine}
From Remark~\ref{rem:quasiopen} and 
Proposition~\ref{prop:proppfine} we infer that, for every 
$p$-quasi open subset $A$ of~$\R^N$, there exist a Borel 
and $p$-finely open set $W$ and a set of zero $p$-capacity 
$E$ such that $A=W\cup E$.
\end{remark}
\paragraph{{\bf Basic definitions about $\Gamma-$convergence}}
Before stating the definition of $\Gamma-$convergence, we recall 
that, given a topological space $\mathcal{X}$ and a function 
$f:\mathcal{X}\rightarrow \overline \R$, the \emph{lower 
semicontinuous envelope} of $f$ is defined as 
\[
\text{sc}^-f := \sup{\Big\{g :\,\, 
\text{$g:\mathcal{X}\rightarrow \overline \R$ is lower 
semicontinuous and $g\leq f$}\Big\}}\,.
\]
We start with the topological definition of $\Gamma-$convergence
(see~\cite{dmgammaconv}).
\begin{definition}
Let $\mathcal{X}$ be a topological space and $\mathcal N(u)$ the 
family of all open neighborhoods of a point $u\in \mathcal{X}$. 
Given a sequence of functions 
$f_n: \mathcal{X}\rightarrow \overline \R$ with $n\in\N$, we define
\begin{alignat*}{3}
&(\Gamma-\liminf_{n\to\infty}\, f_n)(u)
&&:=
\sup_{U\in \mathcal N(u)}\liminf_{n\to \infty}
\inf_{v\in U}f_n(v)
&&\qquad\text{for all $u\in \mathcal{X}$}\,,\\
&(\Gamma-\limsup_{n\to\infty}\, f_n)(u)
&&:=
\sup_{U\in \mathcal N(u)}\limsup_{n\to \infty}
\inf_{v\in U}f_n(v)
&&\qquad\text{for all $u\in \mathcal{X}$}\,.
\end{alignat*}
At each $u\in\mathcal{X}$ satisfying
\[
(\Gamma-\liminf_{n\to\infty} f_n)(u) 
= (\Gamma-\limsup_{n\to\infty} f_n)(u) \,,
\]
we denote by $(\Gamma-\lim\limits_{n\to\infty} f_n)(u)$ the
common value of $(\Gamma-\liminf\limits_{n\to\infty} f_n)(u)$ 
and $(\Gamma-\limsup\limits_{n\to\infty} f_n)(u)$.
\par
Given $f:\mathcal{X}\rightarrow \overline \R$, we say that 
$(f_n)$ is \emph{$\Gamma-$convergent} to $f$ in $\mathcal{X}$, if
\[
f(u) = (\Gamma-\lim_{n\to\infty} f_n)(u) 
\qquad\text{for all $u\in \mathcal{X}$}\,.
\]
\end{definition}
If $\mathcal{X}$ is metrizable, then the following properties
hold:
\begin{itemize}
\item
for every $u_n \to u$, we have
\[
(\Gamma-\liminf_{n\to\infty} f_n)(u)\leq 
\liminf_{n\to\infty} f_n(u_n)\,;
\]
\item
for every $u\in \mathcal{X}$ there exists a recovery 
sequence $u_n \to u$ such that 
\[
(\Gamma-\liminf_{n\to\infty} f_n)(u) = 
\liminf_{n\to\infty} f_n(u_n)\,;
\]
\item
for every $u_n \to u$, we have
\[
(\Gamma-\limsup_{n\to\infty} f_n)(u)\leq 
\limsup_{n\to\infty} f_n(u_n)\,;
\]
\item
for every $u\in \mathcal{X}$ there exists a recovery 
sequence $u_n \to u$ such that 
\[
(\Gamma-\limsup_{n\to\infty} f_n)(u) = 
\limsup_{n\to\infty} f_n(u_n)\,.
\]
\end{itemize}
%

%--------------------------------------------------------------------

\section{Convergence of functionals and of inf-sup values}
\label{sect:convminmax}
In this section we develop some results of~\cite{dema}.
We consider an index $i$ with the following properties:
\emph{
\begin{enumerate}[align=parleft]
\item[\mylabel{ha}{\ha}]
$\idx{K}$ is an integer greater or equal than $1$ and is defined 
whenever $K$ is a nonempty, compact and symmetric subset of a 
metrizable topological vector space $\mathcal{X}$ such that 
$0\not\in K$;
\item[\mylabel{hb}{\hb}]
if $K\subseteq \mathcal{X}\setminus\{0\}$ is nonempty, compact 
and symmetric, then there exists an open subset $U$ of 
$\mathcal{X}\setminus\{0\}$ such that $K\subseteq U$ and
\[
~\qquad\quad
\idx{\widehat{K}} \leq \idx{K}
\quad\text{for all nonempty, compact and symmetric 
$\widehat{K}\subseteq U$}\,;
\]
\item[\mylabel{hc}{\hc}]
if $K_1, K_2\subseteq \mathcal{X}\setminus\{0\}$ are nonempty, 
compact and symmetric, then
\[
\idx{K_1\cup K_2} \leq \idx{K_1} + \idx{K_2}\,;
\]
\item[\mylabel{hd}{\hd}]
if $\mathcal{Y}$ also is a metrizable topological vector space, 
$K\subseteq \mathcal{X}\setminus\{0\}$ is nonempty, compact and
symmetric and $\pi:K\rightarrow \mathcal{Y}\setminus\{0\}$ is 
continuous and odd, then we have $\idx{\pi(K)} \geq \idx{K}$;
\item[\mylabel{he}{\he}]
if $\mathcal{X}$ is a real normed space with
$1\leq \dim \mathcal{X} < \infty$, then we have
\[
\idx{\left\{u\in \mathcal{X}:\,\,\|u\|=1\right\}} = 
\dim\mathcal{X}\,.
\]
\end{enumerate}
}
Well known examples are the Krasnosel'skii genus
(see e.g.~\cite{krasnoselskii1964, rabinowitz1986}) 
and the $\Z_2$-cohomo\-logical index 
(see~\cite{fadell_rabinowitz1977, fadell_rabinowitz1978}).
More general examples are contained in~\cite{bartsch1993}.
\par
Throughout this section, $\mathcal{X}$ will denote a 
metrizable and locally convex topological vector space.
We also denote by $\mathcal{K}$ the family of nonempty and 
compact subsets of $\mathcal{X}\setminus\{0\}$ endowed with 
the metrizable topology of the Hausdorff convergence (see 
e.g.~\cite[Definition~4.4.9]{ambrosio_tilli2004}).
Finally, for every integer $m\geq 1$, we denote by 
$\mathcal K_{m}$ the family of nonempty, compact and 
symmetric subsets $K$ of $\mathcal{X}\setminus\{0\}$ 
such that $\idx{K}\geq m$. 
\par
Assume we also have the even functionals
\[
R^{(n)},R:\mathcal X\setminus\{0\}\rightarrow [0,+\infty]\,,
\qquad\text{where $n\in\N$}\,,
\]
and define 
$\mathcal{R}_m^{(n)}:\mathcal{K}\rightarrow[0,+\infty]$ by
\[
\mathcal{R}_m^{(n)}(K) = 
\begin{cases}
\sup\limits_{u\in K} R^{(n)}(u)
&\qquad\text{if $K\in \mathcal{K}_{m}$}\,,\\
+\infty
&\qquad\text{otherwise}\,,
\end{cases}
\]
and $\mathcal{R}_m:\mathcal{K}\rightarrow[0,+\infty]$
in the analogous way with $R$ instead of $R^{(n)}$.
\par
The next result is a simple adaptation 
of~\cite[Theorem~4.2 and Corollary~4.3]{dema}.
We provide the proof for reader's convenience.
Let us point out that~\cite{dema}
extended by a general abstract approach 
previous results of~\cite{champion_depascale2007}.
However, the main result of this last paper, 
namely~\cite[Theorem~3.3]{champion_depascale2007}, requires
an upper estimate (assumption~(A2)) which is not compatible
with the case of moving (quasi-)open subsets
of a given open set $\Omega$.
\begin{theorem}
\label{liminfinfsup}
If we have 
\[
R(u)\leq \left(\Gamma-\liminf_{n\to\infty} R^{(n)}\right)(u)
\qquad\text{for all $u\in \mathcal X\setminus\{0\}$} \,,
\]
then it is
\[
\mathcal{R}_m(K) \leq 
\left(\Gamma-\liminf_{n\to\infty} \mathcal{R}_m^{(n)}\right)(K)
\qquad\text{for all $m\geq 1$ and $K\in\mathcal{K}$}\,.
\]
If we also have that:
\begin{itemize}
\item
for every strictly increasing sequence $(n_k)$ in $\N$ and every 
sequence $(u^{(k)})$ in $\mathcal{X}\setminus\{0\}$ with
\[
\sup_k R^{(n_k)}(u^{(k)})<+\infty\,,
\]
there exists a subsequence $(u^{(k_j)})$ converging to some 
$u\neq 0$, 
\end{itemize}
then it is also
\begin{alignat*}{3}
&\inf_{K\in\mathcal{K}} \mathcal{R}_m(K) &&\leq 
\liminf_{n\to\infty} \left(
\inf_{K\in \mathcal{K}} \mathcal{R}_m^{(n)}(K)\right)\,,\\
&\inf_{K\in\mathcal{K}_m} \sup_{u\in K} R(u) &&\leq
\liminf_{n\to\infty} \left(
\inf_{K\in\mathcal{K}_m} \sup_{u\in K} R^{(n)}(u)\right)\,,
\end{alignat*}
for all $m\geq 1$, where we agree that $\inf\emptyset=+\infty$.
\end{theorem}
\begin{proof}
Let $m\geq 1$, let $K\in\mathcal{K}$ and let $(K^{(n)})$ be 
a sequence Hausdorff converging to $K$ such that
\[
\liminf_{n\to\infty} \mathcal{R}_m^{(n)}(K^{(n)}) =
\left(\Gamma-\liminf_{n\to\infty} \mathcal{R}_m^{(n)}\right)(K) 
\,.
\]
Without loss of generality, we may assume that this value
is not $+\infty$.
Let $\lambda\in\R$ with
\[
\lambda > \liminf_{n\to\infty} \mathcal{R}_m^{(n)}(K^{(n)})\,.
\]
Then there exists a subsequence $(K^{(n_k)})$ such that
\[
\sup_{k\in\N} \,\sup_{u\in K^{(n_k)}} \,R^{(n_k)}(u) = 
\sup_{k\in \N}\,\mathcal{R}_m^{(n_k)}(K^{(n_k)}) < \lambda \,.
\]
In particular, $K^{(n_k)}\in\mathcal{K}_m^{(n_k)}$ so that
$K$ also is symmetric.
\par
On the other hand, for every $u\in K$, there exists 
$u^{(n)}\in K^{(n)}$ with $u^{(n)}\to u$, whence
\[
R(u) \leq \liminf_{n\to\infty} R^{(n)}(u^{(n)})
\leq \liminf_{k\to\infty} R^{(n_k)}(u^{(n_k)}) \leq \lambda\,,
\]
which implies that
\[
\sup_{u\in K} R(u) \leq \lambda\,.
\]
Let $U$ be an open subset of $\mathcal{X}\setminus\{0\}$ such 
that $K\subseteq U$ and
\[
\idx{\widehat{K}} \leq \idx{K}
\]
for all nonempty, compact and symmetric subset $\widehat{K}$ 
of $U$.
Since $K^{(n_k)}\subseteq U$ eventually as $k\to\infty$,
we have $\idx{K^{(n_k)}}\leq \idx{K}$ eventually as $k\to\infty$,
whence $\idx{K}\geq m$.
Therefore, it is $K\in\mathcal{K}_m$ and
\[
\mathcal{R}_m(K) = \sup_{u\in K} R(u) \leq \lambda\,.
\]
By the arbitrariness of $\lambda$, it follows that
\[
\mathcal{R}_m(K) \leq 
\left(\Gamma-\liminf_{n\to\infty} 
\mathcal{R}_m^{(n)}\right)(K)\,.
\]
Assume now that, for every strictly increasing sequence 
$(n_k)$ in $\N$ and every sequence $(u^{(k)})$ in 
$\mathcal{X}\setminus\{0\}$ with
\[
\sup_k R^{(n_k)}(u^{(k)})<+\infty\,,
\]
there exists a subsequence $(u^{(k_j)})$ converging to some 
$u\neq 0$.
\par
Then the sequence $(R^{(n)})$ is asymptotically 
equicoercive in the sense of~\cite[Definition~2.3]{dema}.
From~\cite[Proposition~2.5]{dema} we infer that the
sequence $(\mathcal{R}_m^{(n)})$ is asymptotically
equicoercive with respect to the Hausdorff convergence.
From~\cite[Proposition~2.4]{dema} we conclude that
\[
\inf_{K\in\mathcal{K}} \mathcal{R}_m(K) \leq 
\liminf_{n\to\infty} \left(
\inf_{K\in \mathcal{K}} \mathcal{R}_m^{(n)}(K)\right)\,,
\]
namely
\[
\inf_{K\in\mathcal{K}_m} \sup_{u\in K} R(u) \leq
\liminf_{n\to\infty} \left(
\inf_{K\in\mathcal{K}_m} \sup_{u\in K} R^{(n)}(u)\right)
\]
and the proof is complete.
\end{proof}
Now we consider the particular case in which
\[
R^{(n)}(u) = 
\begin{cases}
f^{(n)}(u)
&\qquad\text{if $1+g_2^{(n)}(u)\leq 
g_1^{(n)}(u) < +\infty$}\,,\\
\noalign{\medskip}
+\infty
&\qquad\text{otherwise}\,,
\end{cases}
\]
where 
\[
f^{(n)}, g_1^{(n)}, g_2^{(n)}:\mathcal X\rightarrow [0,+\infty]
\]
are even functionals, and $R$ is defined in the analogous way
with respect to the even functionals
\[
f, g_1, g_2:\mathcal X\rightarrow [0,+\infty] \,.
\]
For every $E\subseteq \mathcal{X}$, define also
$I_E:\mathcal{X}\rightarrow[0,+\infty]$ by
\[
I_E(u) =
\begin{cases}
0 &\qquad\text{if $u\in E$}\,,\\
\noalign{\medskip}
+\infty &\qquad\text{otherwise}\,.
\end{cases}
\]
\begin{corollary}\label{cor4.3dema}
Assume that:
\begin{enumerate}[label={\upshape\alph*)}, align=parleft, 
widest=iii, leftmargin=*]
\item[\mylabel{ha43}{\ha}]
the functionals $f^{(n)}, f$, $g_1^{(n)}, g_1$ and 
$g_2^{(n)}, g_2$ are positively homogeneous of the 
same degree $\alpha > 0$;
\item[\mylabel{hb43}{\hb}]
we have
\begin{multline*}
~\qquad
(f+\lambda g_2+I_{\{g_1<+\infty\}})(u)\leq 
\left(\Gamma-\liminf_{n\to\infty} 
\left(f^{(n)}+\lambda g_2^{(n)}+I_{\{g_1^{(n)}<+\infty\}}
\right)\right)(u) \\
\qquad\text{for all $\lambda >0$ and 
$u\in \mathcal X\setminus\{0\}$}\,;
\end{multline*}
\item[\mylabel{hc43}{\hc}]
for every strictly increasing sequence $(n_k)$ in $\N$ and 
every sequence $(u^{(k)})$ in $\mathcal{X}\setminus\{0\}$ with
\[
\sup_k f^{(n_k)}(u^{(k)})<+\infty\,,\quad
\sup_k g_1^{(n_k)}(u^{(k)})<+\infty\,,\quad
g_2^{(n_k)}(u^{(k)})<g_1^{(n_k)}(u^{(k)})
\quad\text{for all $k\in\N$}\,,
\]
there exists a subsequence $(u^{(k_j)})$ converging in 
$\mathcal{X}$ to some $u$ satisfying
\[
g_1(u) \geq
\limsup_{j\to\infty}\,g_1^{(n_{k_j})}(u^{(k_j)}) \,;
\]
\item[\mylabel{hd43}{\hd}]
we have $g_1(0)=0$ and $f(u)>0$ for all $u\neq 0$ with 
$g_2(u) \leq g_1(u) < +\infty$.
\end{enumerate}
\par\indent
Then, for every $m\geq 1$, we have
\begin{alignat*}{3}
&\inf_{K\in\mathcal{K}} \mathcal{R}_m(K) &&\leq 
\liminf_{n\to\infty} \left(
\inf_{K\in \mathcal{K}} \mathcal{R}_m^{(n)}(K)\right)\,,\\
&\inf_{K\in\mathcal{K}_m} \sup_{u\in K} R(u) &&\leq
\liminf_{n\to\infty} \left(
\inf_{K\in\mathcal{K}_m} \sup_{u\in K} R^{(n)}(u)\right)\,.
\end{alignat*}
\end{corollary}
\begin{proof}
We aim to apply Theorem~\ref{liminfinfsup}.
\par\noindent
I)~First of all we claim that, if $(u^{(n)})$ is a sequence 
converging to $u$ in $\mathcal{X}\setminus\{0\}$ with
\[
\sup_n f^{(n)}(u^{(n)})<+\infty\,,\,\,
\sup_n g_2^{(n)}(u^{(n)})<+\infty\,,\,\,
g_1^{(n)}(u^{(n)})<+\infty\,\,\text{for all $n\in\N$}\,,
\]
then we have
\[
f(u) \leq \liminf_{n\to\infty}\,f^{(n)}(u^{(n)}) \,,\qquad
g_2(u) \leq \liminf_{n\to\infty}\,g_2^{(n)}(u^{(n)})\,,\qquad
g_1(u) < +\infty \,.
\]
Actually, by assumption~\ref{hb43} we have $g_1(u)<+\infty$ 
and, for every $\lambda>0$, 
\begin{alignat*}{3}
&f(u) \leq f(u) + \lambda g_2(u) &&\leq \liminf_{n\to\infty}\,
\left(f^{(n)}(u^{(n)})+\lambda g_2^{(n)}(u^{(n)})\right) 
\\ &&&~\qquad 
\leq \liminf_{n\to\infty}\,f^{(n)}(u^{(n)}) 
+ \lambda \limsup_{n\to\infty}\,g_2^{(n)}(u^{(n)})\,,\\
&g_2(u) \leq \frac{1}{\lambda}\,
\left(f(u) + \lambda g_2(u)\right) &&\leq 
\frac{1}{\lambda}\,\liminf_{n\to\infty}\,
\left(f^{(n)}(u^{(n)})+\lambda g_2^{(n)}(u^{(n)})\right) 
\\ &&&~\qquad 
\leq \frac{1}{\lambda}\,\limsup_{n\to\infty}\,f^{(n)}(u^{(n)})
+ \liminf_{n\to\infty}\,g_2^{(n)}(u^{(n)})\,.
\end{alignat*}
By the arbitrariness of $\lambda$ the claim follows.
\par\noindent
II)~Assume now that $(n_k)$ is a strictly increasing sequence 
in $\N$ and $(u^{(k)})$ a sequence in $\mathcal{X}\setminus\{0\}$ 
such that
\[
\sup_k R^{(n_k)}(u^{(k)})<+\infty\,.
\]
We aim to show that there exists a subsequence $(u^{(k_j)})$ 
converging to some $u$ in $\mathcal{X}\setminus\{0\}$.
\par
Actually, we have
$1+g_2^{(n_k)}(u^{(k)}) \leq g_1^{(n_k)}(u^{(k)}) <+\infty$ and
\[
\sup_k\,f^{(n_k)}(u^{(k)}) < +\infty\,.
\]
First we show that $(g_1^{(n_k)}(u^{(k)}))$ is bounded.
Assume for the sake of contradiction that, up to a subsequence,
\[
\lim_{k\to\infty}\,g_1^{(n_k)}(u^{(k)}) = +\infty\,,
\]
so that a suitably rescaled sequence $(v^{(k)})$ satisfies
\[
\lim_{k\to\infty}\,f^{(n_k)}(v^{(k)}) = 0\,,\qquad
g_2^{(n_k)}(v^{(k)}) < g_1^{(n_k)}(v^{(k)}) = 1
\qquad\text{for all $k\in\N$}\,.
\]
By assumption~\ref{hc43}, up to a further subsequence $(v^{(k)})$ 
is convergent in $\mathcal{X}$ to some $v$ satisfying
$g_1(v)\geq 1$, whence $v\neq 0$ by assumption~\ref{hd43}.
Then by step I we have
\[
f(v) = 0\,,\qquad g_2(v) \leq 1 \leq g_1(v) < +\infty
\]
and a contradiction follows again by assumption~\ref{hd43}.
Therefore $(g_1^{(n_k)}(u^{(k)}))$ is bounded.
\par
Again by assumption~\ref{hc43} we infer that there exists
a subsequence $(u^{(k_j)})$ converging in~$\mathcal{X}$
to some $u$ satisfying
\[
1 \leq \limsup_{j\to\infty}\,g_1^{(n_{k_j})}(u^{(k_j)})
\leq g_1(u)\,,
\]
whence $u\neq 0$.
\par\noindent
III)~Finally, let $u$ in $\mathcal{X}\setminus\{0\}$
and let $(u^{(n)})$ be a sequence converging to $u$ such that
\[
\liminf_{n\to\infty} R^{(n)}(u^{(n)}) =
\left(\Gamma-\liminf_{n\to\infty} R^{(n)}\right)(u)  \,.
\]
If
\[
\liminf_{n\to\infty}\,R^{(n)}(u^{(n)}) < b <+\infty\,,
\]
namely
$1+g_2^{(n)}(u^{(n)}) \leq g_1^{(n)}(u^{(n)}) <+\infty$ and
\[
\liminf_{n\to\infty}\,f^{(n)}(u^{(n)}) < b\,,
\]
up to a subsequence, we have
\[
\sup_n\,f^{(n)}(u^{(n)}) < b\,.
\]
Then, as in step II, we infer that
$(g_1^{(n)}(u^{(n)}))$ is bounded.
From step I and assumption~\ref{hc43} it follows
that
\[
f(u) \leq b\,,\qquad
g_2(u) \leq \liminf_{n\to\infty}\,g_2^{(n)}(u^{(n)})\,,\qquad
\limsup_{n\to\infty}\,g_1^{(n)}(u^{(n)})
\leq g_1(u) < +\infty \,.
\]
Therefore $1+g_2(u) \leq g_1(u) < +\infty$ and
\[
R(u) = f(u) \leq b\,.
\]
From the arbitrariness of $b$ we infer that
\[
R(u)\leq \left(\Gamma-\liminf_{n\to\infty} R^{(n)}\right)(u)
\qquad\text{for all $u\in \mathcal X\setminus\{0\}$} \,.
\]
Then the assertion follows by Theorem~\ref{liminfinfsup}.
\end{proof}
The next results are a variant of~\cite[Theorem~4.1]{dema}.
However, because of the presence of $g_2^{(n)}, g_2$, 
a more involved argument is required.
\par
Let us introduce the subfamily $\mathcal{K}_m^{fin}$ 
of $K$'s in $\mathcal K_m$ such that $K$ is included 
in some finite dimensional subspace of $\mathcal{X}$.
\begin{lemma}
\label{pi}
There exists a compatible distance $d$ on $\mathcal{X}$
such that $d(-u,-v)=d(u,v)$ and such that
$B_r(u)$ is convex for all $u,v\in \mathcal{X}$ and $r>0$.
\par
Moreover, for every nonempty, compact and symmetric
$K\subseteq\mathcal{X}\setminus\{0\}$ and every $r>0$, there exist
a finite and symmetric subset $F$ of $K$ and a continuous map
\[
\begin{array}{ccc}
F\times \mathcal{X} & \longrightarrow & [0,1] \\
\noalign{\medskip}
(v,u) & \mapsto & \vartheta_v(u)
\end{array}
\]
such that
\begin{gather*}
\text{$\vartheta_v(u) = 0$ whenever $d(u,v)\geq r$}\,,\\
\text{$\sum_{v\in F}\vartheta_v(u) = 1$
for all $u\in K$}\,,\\
\text{$\sum_{v\in F}\vartheta_v(u) \leq 1$
for all $u\in \mathcal{X}$}\,,\\
\text{$\vartheta_{-v}(u) = \vartheta_{v}(-u)$
for all $v\in F$ and $u\in \mathcal{X}$}\,.
\end{gather*}
\end{lemma}
\begin{proof}
It is the first part of the proof
of~\cite[Proposition~3.1]{dema}.
\end{proof}
\begin{theorem}
\label{thm4.1dema}
Assume that:
\begin{enumerate}[label={\upshape\alph*)}, align=parleft, 
widest=iii, leftmargin=*]
\item[\mylabel{ha41}{\ha}]
the functionals $f^{(n)}, f$, $g_1^{(n)}, g_1$ and 
$g_2^{(n)}, g_2$ are convex and positively homogeneous of the 
same degree $\alpha\geq 1$;
\item[\mylabel{hb41}{\hb}]
we have
\begin{multline*}
~\qquad
(f+\lambda g_2+I_{\{g_1<+\infty\}})(u)\geq 
\left(\Gamma-\limsup_{n\to\infty} 
\left(f^{(n)}+\lambda g_2^{(n)}+I_{\{g_1^{(n)}<+\infty\}}
\right)\right)(u) \\
\qquad\text{for all $\lambda >0$ and 
$u\in \mathcal X\setminus\{0\}$}\,;
\end{multline*}
\item[\mylabel{hc41}{\hc}]
for every strictly increasing sequence $(n_k)$ in $\N$ and 
every sequence $(u^{(k)})$ converging to $u$ in 
$\mathcal{X}\setminus\{0\}$ with
\[
\sup_k f^{(n_k)}(u^{(k)})<+\infty\,,\qquad
\sup_k g_2^{(n_k)}(u^{(k)})<+\infty\,,
\]
we have
\[
g_1(u) \leq \liminf_{k\to\infty}\,g_1^{(n_k)}(u^{(k)})\,.
\]
\end{enumerate}
\par\indent
Then, for every $m\geq 1$, we have
\[
\inf_{K\in\mathcal{K}_m^{fin}} \sup_{u\in K} \, R(u) 
\geq \limsup_{n\to\infty} \left(
\inf_{K\in\mathcal{K}_m^{fin}} 
\sup_{u\in K} \, R^{(n)}(u)\right)\,.
\]
\end{theorem}
\begin{proof}
Let $d$ be a distance as in Lemma~\ref{pi},
let $K\in\mathcal{K}_m^{fin}$ and $\lambda$ with
\[
\sup_{u\in K}\,R(u) < \lambda < +\infty\,,
\]
whence
\[
\text{$f(u) < \lambda$ and
$1+g_2(u)\leq g_1(u)<+\infty$ for all $u\in K$}\,.
\]
It follows
\[
f(u) + \lambda g_2(u) < \lambda g_1(u)
\qquad\text{for all $u\in K$}\,.
\]
On the other hand, if we denote by $Y$ the vector subspace 
spanned by $K$, we have that $f$, $g_1$ and $g_2$
are finite, hence continuous, if restricted to $Y$
(see e.g.~\cite[Corollary~2.3]{ekeland_temam1999}).
Therefore, there exists $r>0$ such that 
$K\cap B_{r}(0)=\emptyset$ and
\[
f(v) + \lambda g_2(v) < \lambda g_1(w)
\qquad\text{for all $v\in K$ and $w\in Y$ with $d(w,v)<3r$}\,.
\]
Let $F$ and $\vartheta$ be as in Lemma~\ref{pi} and define
an odd and continuous map 
$\pi:K\rightarrow Y$ by
\[
\pi(u) = \sum_{v\in F} \vartheta_v(u)v\,.
\]
Then 
\[
\text{$\pi(u) \in\mathrm{conv} 
\left\{v\in F:\,\,d(v,u)<r\right\}$,\,\,
$d(\pi(u),u)<r$ and $\pi(u)\neq 0$ for all $u\in K$}\,,
\]
whence
\begin{equation}
\label{eq:fg1g2}
f(v) + \lambda\,g_2(v) < \lambda\,g_1(\pi(u)) 
\,\,\text{for all $u,v\in K$ with $d(u,v)<2r$}\,.
\end{equation}
Since $F$ is a finite set, by assumption~\ref{hb41} there exists,
for every $n\in\N$, an odd map
$\psi^{(n)}:F\rightarrow \mathcal{X}$ such  that
\begin{alignat*}{3}
&\lim_{n\to\infty} \psi^{(n)}(v) = v
&&\qquad\text{for all $v\in F$}\,,\\
&\text{$g_1^{(n)}(\psi^{(n)}(v))<+\infty$
eventually as $n\to\infty$}
&&\qquad\text{for all $v\in F$}\,,\\
&f(v) + \lambda\,g_2(v) \geq \limsup_{n\to\infty} 
\left(f^{(n)}(\psi^{(n)}(v)) 
+ \lambda g_2^{(n)}(\psi^{(n)}(v))\right)
&&\qquad\text{for all $v\in F$}\,.
\end{alignat*}
If we define an odd and continuous map 
$\pi^{(n)}:K\rightarrow \mathcal{X}$ by
\[
\pi^{(n)}(u) = \sum_{v\in F} \vartheta_v(u)\psi^{(n)}(v)\,,
\]
we have by the convexity of $f^{(n)}$, $g_1^{(n)}$ and $g_2^{(n)}$
\[
g_1^{(n)}(\pi^{(n)}(u)) < +\infty
\qquad\text{for all $u\in K$, eventually as $n\to\infty$}\,,
\]
\begin{multline*}
\lim_{n\to\infty}\,\pi^{(n)}(u^{(n)}) = \pi(u)\,,
\qquad
\limsup_{n\to\infty}\,f^{(n)}(\pi^{(n)}(u^{(n)})) < +\infty\,,\\
\limsup_{n\to\infty}\,g_2^{(n)}(\pi^{(n)}(u^{(n)})) < +\infty\,,
\qquad\text{whenever $u^{(n)} \to u$ in $K$}\,.
\end{multline*}
Therefore, by assumption~\ref{hc41} and~\eqref{eq:fg1g2}, there 
exists $\overline{n}\in\N$ such that
\begin{multline*}
\pi^{(n)}(u)\neq 0\,,\qquad
f^{(n)}(\psi^{(n)}(v)) 
+ \lambda\,g_2^{(n)}(\psi^{(n)}(v)) < 
\lambda\,g_1^{(n)}(\pi^{(n)}(u)) < +\infty \\ 
\qquad\text{for all $n\geq\overline{n}$, 
$u\in K$ and $v\in F$ with $d(u,v)<r$}\,.
\end{multline*}
By the convexity of $f^{(n)}+\lambda g_2^{(n)}$, we infer that
\[
f^{(n)}(\pi^{(n)}(u)) 
+ \lambda\,g_2^{(n)}(\pi^{(n)}(u)) < 
\lambda\,g_1^{(n)}(\pi^{(n)}(u)) < +\infty
\qquad\text{for all $n\geq\overline{n}$ and $u\in K$}\,,
\]
whence
\[
g_2^{(n)}(\pi^{(n)}(u)) < g_1^{(n)}(\pi^{(n)}(u)) 
\qquad\text{for all $n\geq\overline{n}$ and $u\in K$}\,.
\]
If we denote by $Y^{(n)}$ the vector subspace spanned by
$\psi^{(n)}(F)$, we have again that $g_1^{(n)}$ and $g_2^{(n)}$
are finite, hence continuous, if restricted to $Y^{(n)}$.
If we set
\[
K^{(n)} = \left\{\frac{\pi^{(n)}(u)}{
(g_1^{(n)}(\pi^{(n)}(u)) - g_2^{(n)}(\pi^{(n)}(u)))^{1/\alpha}}:
\,\,u\in K\right\}\,,
\]
it follows that $K^{(n)}$ is included in $Y^{(n)}$ and
\[
\idx{K^{(n)}} \geq \idx{K}\geq m \,,
\]
whence
\[
\text{$K^{(n)}\in \mathcal{K}_m^{fin}$,
$f^{(n)}(u)  < \lambda$ and 
$1+g_2^{(n)}(u) = g_1^{(n)}(u)<+\infty$
for all $n\geq\overline{n}$ and $u\in K^{(n)}$}\,.
\]
Then
\[
\limsup_{n\to\infty} \left(\sup_{u\in K^{(n)}} R^{(n)}(u)\right)
\leq \lambda
\]
and the assertion follows by the arbitrariness of $\lambda$.
\end{proof}
\begin{theorem}
\label{minimaxfin}
Assume that $f$, $g_1$ and $g_2$ are convex and positively
homogeneous of the same degree $\alpha\geq 1$.
Suppose also that:
\begin{enumerate}[label={\upshape\alph*)}, align=parleft, 
widest=iii, leftmargin=*]
\item[\mylabel{ha43a}{\ha}]
for every $b, a > 0$ and sequences $(v_k)$ converging to $v$ in
\[
\left\{u\in \mathcal{X}\setminus\{0\}:\,\,
f(u)\leq b\,,\,\,a + g_2(u) \leq g_1(u) < +\infty\right\}
\]
and $(w_k)$ in
\[
\left\{u\in \mathcal{X}\setminus\{0\}:\,\,
f(u)\leq b\,,\,\,g_2(u) \leq b\right\}
\]
also converging to $v$, we have
\begin{alignat*}{3}
&\limsup_{k\to\infty}\,g_2(v_k) < +\infty\,,\\
&\liminf_{k\to\infty}\,\left(g_1(w_k) - g_2(v_k)\right) \geq a\,.
\end{alignat*}
\end{enumerate}
\par\indent
Then, for every integer $m\geq 1$, we have
\[
\inf_{K\in\mathcal{K}_m}\,\sup_{u\in K}\,R(u) =
\inf_{K\in\mathcal{K}_m^{fin}}\,\sup_{u\in K}\,R(u) \,.
\]
\end{theorem}
\begin{proof}
Let $d$ be again a distance as in Lemma~\ref{pi}.
Of course, we have
\[
\inf_{K\in\mathcal{K}_m}\,\sup_{u\in K}\,R(u) \leq
\inf_{K\in\mathcal{K}_m^{fin}}\,\sup_{u\in K}\,R(u) \,.
\]
To prove the opposite inequality, let $K\in\mathcal{K}_m$ 
and $\lambda$ with
\[
\sup_{u\in K}\,R(u) < \lambda < +\infty
\]
and let $\varepsilon>0$ be such that
\[
\sup_{u\in K}\,R(u) \leq (1-\varepsilon)\lambda \,,
\]
namely
\[
\text{$f(u) \leq (1-\varepsilon)\lambda$ and
$1 +  g_2(u) \leq g_1(u) < +\infty$ for all $u\in K$}\,.
\]
Taking into account assumption~\ref{ha43a}, 
there exists firstly $b > 0$ such that
\[
\text{$f(u) \leq b$ and $g_2(u)\leq b$ for all $u\in K$}
\]
and then $r>0$ such that $K\cap B_{r}(0)=\emptyset$ and
\begin{multline*}
 g_1(w) - g_2(v) > 1 - \varepsilon
\qquad\text{for all $v\in K$ and $w\in \mathcal{X}\setminus\{0\}$}\\
\text{with $f(w) \leq b$, $g_2(w) \leq b$ and $d(w,v)<2r$}\,.
\end{multline*}
It follows
\begin{multline*}
f(v) + \lambda g_2(v) < \lambda g_1(w) \qquad
\text{for all $v\in K$ and $w\in \mathcal{X}\setminus\{0\}$}\\ 
\text{with $f(w) \leq b$, $g_2(w) \leq b$ and $d(w,v)<2r$}\,.
\end{multline*}
Let $F$ and $\vartheta$ be as in Lemma~\ref{pi} and define
an odd and continuous map 
$\pi:K\rightarrow \mathcal{X}$ by
\[
\pi(u) = \sum_{v\in F} \vartheta_v(u)v\,.
\]
Then we have again
\[
\text{$\pi(u) \in\mathrm{conv} 
\left\{v\in F:\,\,d(v,u)<r\right\}$,\,\,
$d(\pi(u),u)<r$ and $\pi(u)\neq 0$ for all $u\in K$}\,.
\]
In particular, by the convexity of $f$, $g_1$  and $g_2$ 
it follows first that
\[
f(v) + \lambda g_2(v) < \lambda g_1(\pi(u)) < +\infty
\qquad\text{for all $v\in F$ and $u\in K$ with $d(u,v)<r$}
\]
and then that
\[
f(\pi(u)) + \lambda g_2(\pi(u)) < \lambda g_1(\pi(u))
< +\infty \qquad\text{for all $u\in K$}\,.
\]
As before, $g_1$ and $g_2$ are continuous when restricted to the 
vector subspace spanned by~$F$.
If we set
\[
\widehat{K} = \left\{\frac{\pi(u)}{
(g_1(\pi(u)) - g_2(\pi(u)))^{1/\alpha}}:
\,\,u\in K\right\}\,,
\]
it follows that
\[
\idx{\widehat{K}} \geq \idx{K}\geq m \,,
\]
whence
\[
\text{$\widehat{K}\in \mathcal{K}_m^{fin}$,
$f(u) < \lambda$ and 
$1+g_2(u) = g_1(u)<+\infty$
for all $u\in \widehat{K}$} \,.
\]
Therefore
\[
\sup_{u\in \widehat{K}}\,R(u) \leq \lambda
\]
and the assertion follows by the arbitrariness of $\lambda$.
\end{proof}
\begin{remark}
\label{rem:minimaxfin}
Suppose that $f$, $g_1$ and $g_2$ are convex and positively
homogeneous of the same degree $\alpha\geq 1$.
\par
Then assumption~\ref{ha43a} of Theorem~\ref{minimaxfin} is 
satisfied in each of the following cases:
\begin{enumerate}[label={\upshape\alph*)}, align=parleft, 
widest=iii, leftmargin=*]
\item[\mylabel{hbsuff}{\hb}]
for every $b>0$, the restriction of $g_1$ to
\[
\left\{u\in \mathcal{X}:\,\,
f(u)\leq b\,,\,\,g_1(u) \leq b\,,\,\,g_2(u) \leq b\right\}
\]
is continuous;
\item[\mylabel{hcsuff}{\hc}]
for every $b>0$, the restriction of $g_1$ to
\[
\left\{u\in \mathcal{X}\setminus\{0\}:\,\,
f(u)\leq b\,,\,\,g_1(u) \leq b\right\}
\]
is lower semicontinuous and $g_2=0$.
\end{enumerate}
\end{remark}
\begin{proof}
Let $(v_k)$ and $(w_k)$ be two sequences as in
assumption~\ref{ha43a} of Theorem~\ref{minimaxfin}.
If~\ref{hbsuff} holds, we first claim that $(g_1(v_k))$
is bounded.
Otherwise, up to a subsequence, a rescaled sequence $(u_k)$
is convergent to $0$ and satisfies $f(u_k) \to 0$ and 
$g_2(u_k) < g_1(u_k)=1$.
On the other hand $f(0)=g_1(0)=g_2(0)=0$ by convexity and 
homogeneity, whence a contradiction.
Since
\[
g_1(w_k) - g_2(v_k) \geq g_1(w_k) - g_1(v_k) + a\,,
\]
the assertion follows.
\par
If~\ref{hcsuff} holds, we have
\[
g_1(w_k) - g_2(v_k) = g_1(w_k)
\]
with $a\leq g_1(v)<+\infty$ and the assertion immediately follows.
\end{proof}
%

%--------------------------------------------------------------------

\section{Nonlinear eigenvalue problems}
\label{sect:nep}
This section is devoted to some basic facts concerning nonlinear 
eigenvalues problems.
Up to some adaptation, our approach is inspired by~\cite{sw}.
\par
Throughout this section, $X$ will denote a reflexive Banach 
space and
\[
\varphi, \psi_1, \psi_2:X\rightarrow \R
\]
three even functionals of class $C^1$ which are assumed to
be positively homogeneous of the same degree $\alpha>1$.
We aim to study the nonlinear eigenvalue problem
\begin{equation}
\label{eq:eig}
\varphi'(u) = \lambda(\psi_1'(u)-\psi_2'(u))\,.
\end{equation}
\begin{definition}
We say that $u\in X$ is an \emph{eigenvector} 
of~\eqref{eq:eig} if $\psi_1(u) - \psi_2(u) \neq 0$ and there 
exists $\lambda\in\R$ such that $(u,\lambda)$ 
satisfies~\eqref{eq:eig}.
It is easily seen that
\[
\lambda = \frac{\varphi(u)}{\psi_1(u) - \psi_2(u)}
\]
and $\lambda$ is said to be the \emph{eigenvalue} associated 
with $u$.
\end{definition}
In the following of this section, we consider only the
eigenvectors $u$ with $\psi_1(u) - \psi_2(u) > 0$ and the
associated eigenvalues $\lambda$.
If we set
\[
\widehat{M} =
\left\{u\in X:\,\,\psi_1(u) - \psi_2(u) = 1\right\}\,,
\]
it is easily seen that $\widehat{M}$ is a symmetric hypersurface 
in $X\setminus\{0\}$ of class $C^1$ and that $\lambda$ is an
eigenvalue if and only if $\lambda$ is a critical value of
$\varphi\bigl|_{\widehat{M}}$. 
\par
For the next concepts, we refer the reader 
to~\cite{browder1983, deimling1985}.
\par
\begin{definition}
Let $D \subseteq X$.
A map $F:D \rightarrow X'$ is said to be \emph{of class~$(S)_+$}
if, for every sequence $(u_n)$ in $D$ weakly convergent to $u$
in $X$ with
\[
\limsup_{n\to\infty} \,\langle F(u_n),u_n-u\rangle \leq 0\,,
\]
we have $\|u_n-u\|\to 0$.
\par
If $Y$ is a topological space, a map $F:D \rightarrow Y$ is 
said to be \emph{completely continuous} if it is continuous and, 
for every bounded sequence $(u_n)$ in $D$, the sequence 
$(F(u_n))$ admits a convergent subsequence in $Y$.
\end{definition}
Throughout this section, we assume that:
\emph{
\begin{enumerate}[align=parleft]
\item[\mylabel{hae}{\hei}]
for every $\lambda>0$, we have that 
\[
\varphi' + \lambda\psi_2' :X\rightarrow X'
\]
is of class $(S)_+$, while
\[
\psi_1':X\rightarrow X'
\]
is completely continuous with respect 
to the strong topology of $X'$;
\item[\mylabel{hbe}{\heii}]
we have $\varphi(u)>0$ for all $u\neq 0$ with 
$\psi_1(u) - \psi_2(u) \geq 0$.
\end{enumerate}
}
\begin{lemma}
\label{lemeig}
For every $b\in\R$, the set
\[
\left\{u\in X:\,\,
\varphi(u)\leq b\,,\,\, \psi_1(u) - \psi_2(u) \geq 0\right\}
\]
is bounded and we have
\[
\inf_{u\in\widehat{M}} \varphi(u) > 0 \,.
\]
\end{lemma}
\begin{proof}
Let us recall that, because of assumption~\ref{hae}, the 
functional $\varphi+\lambda\psi_2$ is sequentially lower 
semicontinuous with respect to the weak topology for all 
$\lambda>0$ (see 
also~\cite[Proposition~3.5]{cingolani_degiovanni_vannella2018}), 
while $\psi_1$ is sequentially continuous with respect to the 
weak topology.
\par
Let $b\in\R$, let $(u_n)$ be a sequence in $X$ with
$\varphi(u_n)\leq b$ and $\psi_1(u_n) - \psi_2(u_n) \geq 0$
and assume, for the sake of contradiction, that 
\[
\lim_{n\to\infty}\, \|u_n\| = +\infty\,.
\]
Then a suitably rescaled sequence $(v_n)$ satisfies
\[
\lim_{n\to\infty}\,\varphi(v_n) = 0\,,\qquad
\psi_2(v_n) \leq \psi_1(v_n)\,,\,\,\|v_n\|=1
\qquad\text{for all $n\in\N$}\,.
\]
Up to a subsequence, we may also assume that $(v_n)$ is weakly 
convergent to some $v$.
For every $\lambda>0$, it follows that
\[
\begin{split}
\lambda\varphi(v) + \psi_2(v) 
&= \lambda\left(\varphi(v) + \lambda^{-1}\,\psi_2(v)\right) 
\leq \lambda \liminf_{n\to\infty}\left(\varphi(v_n) 
+ \lambda^{-1}\,\psi_2(v_n)\right) \\
&=
\lambda \lim_{n\to\infty}\,\varphi(v_n) 
+ \liminf_{n\to\infty}\,\psi_2(v_n) 
=
\liminf_{n\to\infty}\,\psi_2(v_n) \\
&\leq
\lim_{n\to\infty}\,\psi_1(v_n) = \psi_1(v) \,.
\end{split}
\]
From the arbitrariness of $\lambda$ we infer that
$\psi_2(v)\leq\psi_1(v)$ and that $\varphi(v) = 0$,
whence $v=0$ by assumption~\ref{hbe}.
On the other hand, we have
\[
\begin{split}
\limsup_{n\to\infty}\,
\langle \varphi'(v_n)+\psi_2'(v_n),v_n\rangle 
&=
\alpha\,\limsup_{n\to\infty}\,
\left(\varphi(v_n)+\psi_2(v_n)\right) \\
&=
\alpha\,\lim_{n\to\infty}\,\varphi(v_n) 
+ \alpha\,\limsup_{n\to\infty}\,\psi_2(v_n) \\
&\leq
\alpha\,\lim_{n\to\infty}\,\psi_1(v_n) = \alpha \,\psi_1(v) = 0\,,
\end{split}
\]
whence $\|v_n\|\to 0$ by assumption~\ref{hae} and a
contradiction follows.
\par
Now let $(u_n)$ in $\widehat{M}$ be such that
\[
\lim_{n\to\infty}\,\varphi(u_n) = 
\inf_{u\in\widehat{M}} \varphi(u)\,.
\]
By the previous step $(u_n)$ is weakly convergent, up to a 
subsequence, to some $u$.
If $\inf\limits_{u\in\widehat{M}} \varphi(u) = 0$, arguing 
as before we find
\[
\lambda\varphi(u) + \psi_2(u) \leq 
\lim_{n\to\infty}\, \psi_2(u_n) = 
\lim_{n\to\infty}\,(\psi_1(u_n)-1) = \psi_1(u)-1
\]
for all $\lambda>0$, whence a contradiction.
Therefore, it is $\inf\limits_{u\in\widehat{M}} \varphi(u) > 0$.
\end{proof}
\begin{theorem}
\label{thmeig}
The functional $\varphi\bigl|_{\widehat{M}}$ is bounded from 
below and satisfies $(PS)_c$ for all $c\in\R$, namely every 
sequence $(u_n)$ in $\widehat{M}$ satisfying
\[
\lim_{n\to\infty}\, \varphi(u_n)=c\,,\qquad
\lim_{n\to\infty} \,
\left\|\left(\varphi\bigl|_{\widehat{M}}\right)'(u_n)\right\| = 0
\]
admits a converging subsequence.
\end{theorem}
\begin{proof}
Of course, $\varphi\bigl|_{\widehat{M}}$ is bounded from below
by assumption~\ref{hbe}.
To prove $(PS)_c$, let us recall that
\[
\left\|\left(\varphi\bigl|_{\widehat{M}}\right)'(u)\right\| =
\min\left\{\|\varphi'(u) - \lambda(\psi_1'(u)-\psi_2'(u))\|:\,\,
\lambda\in\R\right\} \qquad
\text{for all $u\in \widehat{M}$}\,.
\]
Let $(u_n)$ be a sequence in $\widehat{M}$ and $(\lambda_n)$ 
a sequence in $\R$ such that
\[
\lim_{n\to\infty} \,\varphi(u_n)=c\,,\qquad
\lim_{n\to\infty}\,
\|\varphi'(u_n) - \lambda_n(\psi_1'(u_n)-\psi_2'(u_n))\| = 0\,.
\]
By Lemma~\ref{lemeig} we have $c>0$ and $(u_n)$ is bounded 
hence weakly convergent, up to a subsequence, to some $u$.
If we set
\[
z_n = \varphi'(u_n) - \lambda_n(\psi_1'(u_n)-\psi_2'(u_n))\,,
\]
it follows
\[
\alpha\varphi(u_n) = \langle\varphi'(u_n),u_n\rangle 
=\lambda_n\langle\psi_1'(u_n)-\psi_2'(u_n),u_n\rangle
+ \langle z_n,u_n\rangle 
= \lambda_n \alpha + \langle z_n,u_n\rangle \,,
\]
whence
\[
\lim_{n\to\infty} \lambda_n = c > 0\,.
\]
Up to a subsequence, $(\psi_1'(u_n))$ is strongly convergent
in $X'$ and there exists $\lambda>0$ such that
\[
\limsup_{n\to\infty}\,
(\lambda-\lambda_n)\langle \psi_2'(u_n),u_n-u\rangle \leq 0\,.
\]
Then we have
\begin{multline*}
\limsup_{n\to\infty}\,
\langle \varphi'(u_n) + \lambda\, \psi_2'(u_n),u_n-u\rangle \\
= \limsup_{n\to\infty}\,\left[
\langle \lambda_n \psi_1'(u_n)+z_n,u_n-u\rangle
+ (\lambda-\lambda_n)
\langle \psi_2'(u_n),u_n-u\rangle\right] \leq 0\,.
\end{multline*}
From assumption~\ref{hae} we infer that $\|u_n-u\| \to 0$ 
and $(PS)_c$ follows.
\end{proof}
Now let $i$ be an index as in Section~\ref{sect:convminmax} 
and define, for every $m\geq 1$,
\[
\hat{\lambda}_m =
\inf\left\{\max_{u\in K} \varphi(u):\,\,
\text{$K$ is a nonempty, compact and symmetric subset of 
$\widehat{M}$ with $\idx{K}\geq m$} \right\}\,,
\]
where we agree that $\hat{\lambda}_m=+\infty$ if there is no $K$
with $\idx{K}\geq m$.
It is easily seen that $\hat{\lambda}_m \leq \hat{\lambda}_{m+1}$.
\begin{theorem}
\label{generaleig}
The following facts hold:
\begin{enumerate}[label={\upshape\alph*)}, align=parleft, 
widest=iii, leftmargin=*]
\item[\mylabel{hage}{\ha}]
if $\widehat{M}\neq\emptyset$, which is equivalent to
\[
\left\{u\in X:\,\,\psi_1(u) - \psi_2(u) > 0\right\}\neq\emptyset\,,
\]
then $\inf\limits_{u\in\widehat{M}} \varphi(u)$ is achieved and
\[
\hat{\lambda}_1 = \min_{u\in\widehat{M}} \varphi(u)\,;
\]
\item[\mylabel{hbge}{\hb}]
if there exists an odd and continuous map
\[
\left\{\xi\in \R^m:\,|\xi|=1\right\} \rightarrow
\left\{u\in X:\,\,\psi_1(u) - \psi_2(u) > 0\right\}\,,
\]
then $\hat{\lambda}_m<+\infty$;
\item[\mylabel{hcge}{\hc}]
if $\hat{\lambda}_m < +\infty$, then
$\hat{\lambda}_m$ is an eigenvalue;
\item[\mylabel{hdge}{\hd}]
if 
\[
\hat{\lambda}_m = \cdots = \hat{\lambda}_{m+j-1} < +\infty\,,
\]
then
\[
\idx{\left\{u\in \widehat{M}:\,\,
\text{$u$ is an eigenvector with eigenvalue $\hat{\lambda}_m$}
\right\}} \geq j\,;
\]
\item[\mylabel{hege}{\he}]
we have
\[
\lim_{m\to\infty} \hat{\lambda}_m = +\infty\,.
\]
\end{enumerate}
\end{theorem}
\begin{proof}
When $\widehat{M}$ is of class $C^2$, the assertions are well 
known consequences of Theorem~\ref{thmeig} (see 
e.g.~\cite{rabinowitz1986}).
The result in the case of manifolds of class~$C^1$ follows 
from~\cite{cdm, szulkin1988}.
\end{proof}
\begin{example}
Let $\varphi, \psi_1, \psi_2:\R^2\rightarrow\R$ be defined by
\[
\varphi(u) = \frac{1}{2}\,(u_1+u_2)^2\,,\qquad
\psi_1(u) = \frac{1}{2}\,u_1^2\,,\qquad 
\psi_2(u) = \frac{1}{2}\,u_2^2\,.
\]
Then the problem
\[
\begin{cases}
u_1 + u_2 = \lambda u_1\\
\noalign{\medskip}
u_1 + u_2 = - \lambda u_2
\end{cases}
\]
has no solution with $\psi_1(u) - \psi_2(u)\neq 0$ and we have
\[
\inf \left\{\varphi(u):\,\,\psi_1(u)-\psi_2(u)=1\right\}=0\,,
\qquad
\text{$\varphi(u)>0$ for all $u$ 
with $\psi_1(u)-\psi_2(u)\neq 0$}\,.
\]
On the other hand, assumption~\ref{hbe} is not satisfied.
\end{example}
%

%--------------------------------------------------------------------

\section{Convergence of measures and of functionals}
\label{sect:convergence}
In this section we introduce the notion of local 
$\gamma$-convergence of measures in $\R^N$ and study its properties 
in relation to the $\Gamma$-convergence of suitable functionals.
\subsection{Convergence of capacitary measures}
In the first part of this subsection we take advantage of the
results of~\cite{dalmaso1987}, where the case $p=2$ was considered.
On the other hand, taking into account 
Proposition~\ref{prop:proppfine}, only minor changes are 
required in the general case.
\begin{definition}
Let $\Om$ be an open subset of $\R^N$. 
We say that a non-negative Borel measure $\mu$ in $\Om$ is 
\emph{$p$-capacitary} if, for every $B\in\bor(\Om)$ 
with $\cp_p(B)=0$, we have $\mu(B)=0$.
\par
A $p$-capacitary measure $\mu$ in $\Omega$ is said to
be \emph{outer regular}, if
\begin{equation*}
%\label{eq:propMp0}
\mu(B)=\inf \left\{\mu(A) :\,\, 
\text{$A\in\bor(\Om)$, $A\supseteq B$ and $A$ is
$p$-quasi open}\right\} \qquad\text{for all $B\in\bor(\Om)$}\,.
\end{equation*}
\end{definition}
\begin{definition}
Two $p$-capacitary measures $\mu_1, \mu_2$ in $\Omega$ are
said to be \emph{equivalent}, if
\[
\mu_1(A) = \mu_2(A)
\qquad\text{for all $A\in\bor(\Omega)$ with $A$ $p$-quasi open}\,.
\]
We denote by $\cmeas(\Omega)$ the quotient of the set of all
$p$-capacitary measures in $\Omega$ with respect to such an
equivalence relation.
\end{definition}
\begin{proposition}
\label{prop:equivout}
For every $p$-capacitary measure $\mu$ in $\Omega$, if we set
\[
\tilde{\mu}(B)=\inf \left\{\mu(A) :\,\, 
\text{$A\in\bor(\Om)$, $A\supseteq B$ and $A$ is
$p$-quasi open}\right\} \qquad\text{for all $B\in\bor(\Om)$}\,,
\]
then $\tilde{\mu}$ is an outer regular $p$-capacitary
measure in $\Omega$ equivalent to $\mu$.
\par
Moreover, if $\mu_1, \mu_2$ are two equivalent
outer regular $p$-capacitary measures in $\Omega$,
then $\mu_1=\mu_2$.
\end{proposition}
\begin{proof}
In the case $p=2$, 
see~\cite[Theorems~2.6, 3.9 and~3.10 and Remark~3.4]{dalmaso1987}.
\end{proof}
\begin{definition}
If $\mu, \nu\in\cmeas(\Omega)$, we write $\mu\leq\nu$ if
\[
\mu(A) \leq \nu(A)
\qquad\text{for all $A\in\bor(\Omega)$ with $A$ $p$-quasi open}\,.
\]
\end{definition}
It is easily seen that this is an order relation in 
$\cmeas(\Omega)$.
\begin{example}
\label{ex:cmeas}
Let us provide the two most important examples of $p$-capacitary 
measures.
The first one is given by the measure $\infty_E$ corresponding 
to a subset $E$ of $\Omega$, defined as
\[
\infty_{E}(B):=
\begin{cases}
0 
&\qquad\text{if $\cp_p(B\cap E)=0$}\,,\\
\noalign{\medskip}
+\infty
&\qquad\text{if $\cp_p(B\cap E)>0$}\,,
\end{cases}
\qquad\text{for all $B\in\bor(\Omega)$}\,.
\]
The other one consists in a measure absolutely continuous with 
respect to $\leb^N$, that is, for a $\leb^N$-measurable 
function $V:\Omega\rightarrow [0,+\infty]$, the 
measure $V\,\leb^N$ defined as
\[
(V\,\leb^N)(B) = \int_B V\,d\leb^N
\qquad\text{for all $B\in\bor(\Omega)$}\,.
\]
\end{example}
On the other hand, let us see that each $p$-capacitary measure
admits a decomposition incorporating contributions of this 
particular form.
\begin{definition}
For every $\mu\in\cmeas(\Omega)$, we denote by $A_\mu$ the union 
of all Borel and $p$-finely open subsets~$W$ of $\Omega$ such 
that $\mu(W) <+\infty$. 
This is called the \emph{set of $\sigma$-finiteness} of $\mu$.
\end{definition}
Since each $p$-finely open set is $p$-quasi open, the set 
$A_\mu$ is well defined and in fact $p$-finely open.
\begin{proposition}
\label{prop:Amu}
Let $\mu\in\cmeas(\Omega)$ and let $\mu_1, \mu_2$ be two 
representatives of $\mu$.
Then the following facts hold:
\begin{enumerate}[label={\upshape\alph*)}, align=parleft, 
widest=iii, leftmargin=*]
\item[\mylabel{haAmu}{\ha}]
we have
\begin{alignat*}{3}
&\mu(A) = +\infty
&&\qquad\text{for all $A\in\bor(\Omega)$ with
$\cp_p(A\setminus A_\mu)>0$ and $A$ $p$-quasi open}\,,\\
&\mu_1(B) = \mu_2(B)
&&\qquad\text{for all $B\in\bor(\Omega)$ with
$\cp_p(B\setminus A_\mu)=0$}\,;
\end{alignat*}
\item[\mylabel{hbAmu}{\hb}]
we have that
\[
\tilde\mu(B) = 
\begin{cases}
\mu(B)
&\qquad\text{if $B\in \bor(\Omega)$ and 
$\cp_p(B\setminus A_\mu)=0$}\,,\\
\noalign{\medskip}
+\infty
&\qquad\text{if $B\in \bor(\Omega)$ and 
$\cp_p(B\setminus A_\mu)>0$}\,,
\end{cases}
\]
is the outer regular representative of $\mu$;
\item[\mylabel{hcAmu}{\hc}]
there exists a Borel and $p$-finely open subset $W$
of $A_\mu$ such that $\cp_p(A_\mu \setminus W)=0$ and,
if we set
\[
\mu^{A_\mu}(B) = \mu(B\cap W)
\qquad\text{for all $B\in\bor(\Omega)$}\,,
\]
then $\mu^{A_\mu}$ is a $\sigma$-finite $p$-capacitary measure 
in $\Omega$ independent of the choice of $W$ and of the 
representative of $\mu$.
\end{enumerate}
\end{proposition}
\begin{proof}
In the case $p=2$, 
see~\cite[Theorem~2.6, Proposition~3.16, Remark~3.13
and Theorem~3.17]{dalmaso1987}.
\end{proof}
\begin{definition}
\label{def:Vmu}
For every $\mu\in\cmeas(\Omega)$, we define a $\leb^N$-measurable
function $V_\mu:\Omega\rightarrow [0,+\infty]$ by
\[
V_\mu = 
\begin{cases}
\frac{d\mu^{A_\mu}}{d\leb^N}
&\qquad\text{on $A_\mu$}\,,\\
\noalign{\medskip}
+\infty
&\qquad\text{on $\Omega\setminus A_\mu$}\,,
\end{cases}
\]
and we denote by $\mu_s$ the singular part of 
$\mu^{A_\mu}$ with respect to $\leb^N$.
\end{definition}
\begin{proposition}
\label{prop:AVmu}
The following facts hold:
\begin{enumerate}[label={\upshape\alph*)}, align=parleft, 
widest=iii, leftmargin=*]
\item[\mylabel{haAVmu}{\ha}]
for every $\mu\in\cmeas(\Om)$, we have
\[
\mu(B) =
\infty_{\Om\setminus A_\mu}(B)
+ \int_B V_\mu\,d\leb^N + \mu_s(B)
\]
for all $B\in\bor(\Om)$ with either
$\cp_p(B\setminus A_\mu)=0$ or $B$ $p$-quasi open;
moreover,
\[
\infty_{\Om\setminus A_\mu}
+ V_\mu\,\leb^N + \mu_s
\]
is the outer regular representative of $\mu$;
\item[\mylabel{hbAVmu}{\hb}]
for every $\mu, \nu\in\cmeas(\Om)$ with $\mu\leq\nu$, we have
\[
A_\mu \supseteq A_\nu\,,\qquad
V_\mu \leq V_\nu\quad\text{$\leb^N$-a.e. in $\Omega$;}
\]
\item[\mylabel{hcAVmu}{\hc}]
if $A$ is a $p$-quasi open subset of $\Om$ and
$\mu=\infty_{\Om\setminus A}$, then
\[
\cp_p\left[\left(A \setminus A_\mu\right)
\cup\left(A_\mu\setminus A\right)\right] = 0\,,
\]
whence
\[
\infty_{\Om\setminus A_\mu}(B) = 
\infty_{\Om\setminus A}(B)
\qquad\text{for all $B\in\bor(\Om)$}\,;
\]
\item[\mylabel{hdAVmu}{\hd}]
if $V:\Om\rightarrow [0,+\infty]$ is $\leb^N$-measurable 
and $\mu=V\,\leb^N$, then
\begin{alignat*}{3}
&V_\mu \geq V 
&&\qquad\text{$\leb^N$-a.e. in $\Om$}\,,\\
&V_\mu = V 
&&\qquad\text{$\leb^N$-a.e. in $A_\mu$}\,;
\end{alignat*}
\item[\mylabel{heAVmu}{\he}]
if $V:\Om\rightarrow [0,+\infty]$ is $p$-quasi upper
semicontinuous and $\mu=V\,\leb^N$, then
\[
V_\mu = V 
\qquad\text{$\leb^N$-a.e. in $\Omega$}\,.
\]
\end{enumerate}
\end{proposition}
\begin{proof}
\par\noindent
\par\noindent
\ref{haAVmu}~If $\cp_p(B\setminus A_\mu)=0$, 
it follows from Proposition~\ref{prop:Amu} and the Radon-Nikodym 
Theorem that
\[
\mu(B) =
\infty_{\Om\setminus A_\mu}(B)
+ \int_B V_\mu\,d\leb^N + \mu_s(B)\,,
\]
while, if $\cp_p(B\setminus A_\mu)>0$ and $B$ is $p$-quasi open, 
we have $\infty_{\Om\setminus A_\mu}(B)=\mu(B)=+\infty$ by 
assertion~\ref{haAmu} of Proposition~\ref{prop:Amu}.
In particular, 
\[
\infty_{\Om\setminus A_\mu} + V_\mu\,\leb^N + \mu_s
\]
is the outer regular representative of $\mu$ by~\ref{hbAmu}
of Proposition~\ref{prop:Amu}.
\par\noindent
\ref{hbAVmu}~The fact is obvious.
\par\noindent
\ref{hcAVmu}~Since~$\mu^{A_\mu}$ is $\sigma$-finite, we have
$\mu^{A_\mu}=0$.
Then the assertion follows from~\ref{haAVmu}, as $A$ and
$A_\mu$ can be supposed to be also Borel, up to a set of null 
$p$-capacity.
\par\noindent
\ref{hdAVmu}~By the Radon-Nikodym Theorem, we have
\[
V_\mu = V 
\qquad\text{$\leb^N$-a.e. in $A_\mu$}\,,
\]
while it is obvious that
\[
V_\mu \geq V  
\qquad\text{$\leb^N$-a.e. in $\Om\setminus A_\mu$}\,.
\]
\par\noindent
\ref{heAVmu}~For every $n\in\N$, the set
\[
\left\{x\in\Omega:\,\,V(x) < n\right\}
\]
is $p$-quasi open (see Remark~\ref{rem:quasiusc}).
Therefore, by Remark~\ref{rem:quasiopenfine}, 
there exist a Borel and $p$-finely open set
$W_n$ and $E_n$ with $\cp_p(E_n)=0$ such that
\[
\left\{x\in\Omega:\,\,V(x) < n\right\} = W_n \cup E_n\,.
\]
Then we have
\[
W_n \cap B_n(0) \subseteq A_\mu\,,
\]
whence
\[
\left\{x\in\Omega:\,\,V(x) < +\infty\right\}
\subseteq A_\mu \cup \left(\bigcup_{n\in\N} E_n\right)
\]
and the assertion follows from~\ref{hdAVmu}.
\end{proof}
\begin{example}
Let $N=1$ and $C$ be a closed subset of $\R$ with
empty interior and $\leb^1(C)>0$.
If we set
\[
V = 
\begin{cases}
0
&\qquad\text{on $C$}\,,\\
\noalign{\medskip}
+\infty
&\qquad\text{on $\Omega\setminus C$}\,,
\end{cases}
\]
and consider $\mu = V\,\leb^1$, then we have 
$A_\mu=\emptyset$, whence $V_\mu=+\infty$ 
$\leb^1$-a.e. in $\R$.
\end{example}
\begin{remark}
\label{rem:intequiv}
If $\mu\in\cmeas(\Omega)$ and $u\in W^{1,p}_{loc}(\Omega)$,
then the integral
\[
\int_\Omega |u|^p\,d\mu
\]
is well defined, as
\[
\int_\Omega |u|^p\,d\mu =
\int_0^{+\infty} \mu\left(\left\{|u| > y^{1/p}\right\}
\right)\,d\leb^1(y)
\]
and the sets $\left\{|u| > y^{1/p}\right\}$ are Borel and
$p$-quasi open.
\par
Then the space 
\[
W^{1,p}_{loc}(\Omega)\cap L^p(\Omega,\mu) :=
\left\{u\in W^{1,p}_{loc}(\Omega):\,\,
\int_\Omega |u|^p\,d\mu < +\infty\right\}
\]
is well defined and, for every 
$u \in W^{1,p}_{loc}(\Omega)\cap L^p(\Omega,\mu)$, we have 
\[
\cp_p\left(\{|u|>0\} \setminus A_\mu\right)=0
\]
by~\ref{haAmu} of Proposition~\ref{prop:Amu}.
\par
Again from~\ref{haAmu} of Proposition~\ref{prop:Amu} we 
infer that the integral
\[
\int_\Omega |u|^{p-2}u\,v\,d\mu 
\]
is well defined for all
$u,v \in W^{1,p}_{loc}(\Omega)\cap L^p(\Omega,\mu)$.
\par
Moreover, if $(u_n)$ is a sequence in 
$W^{1,p}_{loc}(\Omega)\cap L^p(\Omega,\mu)$ and
$u\in W^{1,p}_{loc}(\Omega)\cap L^p(\Omega,\mu)$,
then the assertion \emph{the sequence $(u_n)$ is weakly
convergent to $u$ in $L^p(\Omega,\mu)$} is independent
of the choice of the representative of $\mu$.
\end{remark}
Assume now that $\Omega$ is a bounded and open 
subset of $\R^N$.
Here we take advantage of the results of~\cite{dmmu}.
For every $\mu\in\cmeas(\Omega)$, we denote by $w_\mu(\Omega)$ 
the \emph{torsion function} in $\Om$ associated with $\mu$,
defined as the (unique) minimizer of the functional
\[
W^{1,p}_0(\Om)\ni v\mapsto 
\frac{1}{p}\int_{\Om}|\nabla v|^p\,d\leb^N
+\frac{1}{p}\int_{\Om}|v|^p\,d\mu -\int_{\Om} v\,d\leb^N\,.
\]
\begin{remark}
\label{rem:Amuw}
The sets $A_\mu$ and $\{w_\mu(\Om)>0\}$ coincide up to sets
of null $p$-capacity.
\end{remark}
\begin{proof}
\par\noindent
From Remark~\ref{rem:intequiv} we infer that
$\cp_p\left(\{w_\mu(\Om)>0\} \setminus A_\mu\right)=0$.
On the other hand, by the quasi-Lindel\"of property (see
Proposition~\ref{prop:proppfine}), there exists a sequence
$(W_n)$ of Borel and $p$-finely open subsets of $\Omega$ with
$\mu(W_n)<+\infty$ and 
$\cp_p\left(A_\mu \setminus \bigcup\limits_{n\in\N} W_n\right)=0$.
Then we have
\[
\cp_p\left(W_n \setminus \{w_\mu(\Om)>0\}\right)=0
\]
by Proposition~\ref{prop:equivout} and~\cite[Theorem~5.1]{dmmu}, 
whence $\cp_p\left(A_\mu \setminus \{w_\mu(\Om)>0\}\right)=0$.
\end{proof}
If we set
\[
\mathcal{K}^p(\Om) = \left\{v\in W^{1,p}_0(\Om):\,\,
\text{$v\geq 0$ a.e. in $\Omega$ and
$-\Delta_p v \leq 1$ in $W^{-1,p'}(\Om)$}\right\}\,,
\]
it follows from~\cite[Theorem~5.1]{dmmu} that
$\mathcal{K}^p(\Om)$, endowed with the weak topology of
$W^{1,p}_0(\Om)$, is compact and metrizable.
Moreover, again from~\cite[Theorem~5.1]{dmmu} and from 
Proposition~\ref{prop:equivout}, it follows that the map
\[
\begin{array}{ccc}
\cmeas(\Om) & \rightarrow & \mathcal{K}^p(\Om) \\
\noalign{\medskip}
\mu & \mapsto & w_\mu(\Omega)
\end{array}
\]
is bijective.
Then $\cmeas(\Om)$ is endowed with the topology that
makes such a map a homeomorphism.
Therefore, $\cmeas(\Om)$ is a compact and metrizable
topological space.
\begin{definition}
If $\Om$ is a bounded and open subset of $\R^N$,
a sequence $(\mu^{(n)})$ in $\cmeas(\Om)$ is said to be 
\emph{$\gamma^{-\Delta_p}$-convergent} to $\mu$ if it is 
convergent to $\mu$ with respect to the topology we have 
just defined.
This means that $(w_{\mu^{(n)}}(\Om))$ is weakly convergent 
to $w_\mu(\Om)$ in $W^{1,p}_0(\Om)$.
\end{definition}
In the following, we will simply write
\emph{$\gamma$-convergent} instead of
\emph{$\gamma^{-\Delta_p}$-convergent}.
Being a countable product of compact and metrizable
topological spaces, also
\[
\prod_{k\in\N} \cmeas(B_k(0))
\]
endowed with the product topology is compact and
metrizable.
\begin{proposition}
The map
\[
\begin{array}{ccc}
\cmeas(\R^N) & \rightarrow & 
\displaystyle{\prod_{k\in\N} \cmeas(B_k(0))} \\
\noalign{\medskip}
\mu & \mapsto & \left(\mu\bigl|_{\bor(B_k(0))}\right)
\end{array}
\]
is injective with closed image.
\end{proposition}
\begin{proof}
For every $\mu\in\cmeas(\R^N)$ and $A\in \bor(\R^N)$ with
$A$ $p$-quasi open, we have
\[
\mu(A) = \sup_k \mu(A\cap B_k(0))\,.
\]
Therefore the map is injective.
\par
If $(\mu^{(n)})$ is a sequence in $\cmeas(\R^N)$ such that
$(\mu^{(n)}\bigl|_{\bor(B_k(0))})$ is $\gamma$-convergent to 
$\nu_k$ in $\cmeas(B_k(0))$ for all $k\in\N$, it follows 
from Proposition~\ref{prop:equivout} and~\cite[Theorem~6.12]{dmmu}
that $\nu_{k+1}\bigl|_{\bor(B_k(0))}=\nu_k$.
If we set
\begin{equation*}
\mu(A) = \sup_k \nu_k(A\cap B_k(0))
\qquad\text{for all $A\in\bor(\R^N)$ with
$A$ $p$-quasi open}
\end{equation*}
and we denote by $\mu$ the equivalence class of 
\[
\tilde \mu(B) =\inf \left\{\mu(A):\,\, 
\text{$A\in\bor(\Om)$, $A\supseteq B$ and $A$ is
$p$-quasi open}\right\} 
\qquad\text{for all $B\in\bor(\R^N)$}\,,
\]
it is easily seen that $\mu\in\cmeas(\R^N)$ and
$\mu\bigl|_{\bor(B_k(0))}=\nu_k$ for all $k\in\N$.
Therefore the map has closed image.
\end{proof}
Then $\cmeas(\R^N)$ is endowed with the topology that
makes such a map a homeomorphism between $\cmeas(\R^N)$ 
and its image.
Therefore, $\cmeas(\R^N)$ also is a compact and metrizable
topological space.
\begin{definition}
\label{defn:locconv}
A sequence $(\mu^{(n)})$ in $\cmeas(\R^N)$ is said to be 
\emph{locally $\gamma$-convergent} to $\mu$ if it is 
convergent to $\mu$ with respect to the topology we have 
just defined.
Taking into account 
Proposition~\ref{prop:equivout} and~\cite[Theorem~6.12]{dmmu}, 
this means that $(\mu^{(n)}\bigl|_{\bor(\Om)})$ is 
$\gamma$-convergent to $\mu\bigl|_{\bor(\Om)}$ in $\cmeas(\Om)$ 
for all bounded and open subset $\Om$ of $\R^N$.
\par
In particular, if $A^{(n)}, A$ are $p$-quasi open subsets of $\R^N$,
the sequence $(A^{(n)})$ is said to be 
\emph{locally $\gamma$-convergent} to $A$ if
$\mu^{(n)}=\infty_{\R^N\setminus A^{(n)}}$ is 
locally $\gamma$-convergent to $\mu=\infty_{\R^N\setminus A}$.
If $A^{(n)}, A \subseteq \Omega$ for some bounded and open set 
$\Omega$,  then this is equivalent to the classical notion of 
$\gamma$-convergence of sets.
%}
\end{definition}

\subsection{Lower estimate and asymptotic equicoercivity
for a sequence of functionals}
For every $\mu\in \cmeas(\R^N)$, we define a first lower 
semicontinuous and convex functional
\[
f_\mu:L^p_{loc}(\R^N)\rightarrow[0,+\infty]
\]
by
\[
f_\mu(u) =
\begin{cases}
\displaystyle{
\frac{1}{p}\,\int |\nabla u|^p\,d\leb^N +
\frac{1}{p}\,\int |u|^p\,d\mu} 
&\qquad\text{if $u\in W^{1,p}_{loc}(\R^N)$}\,,\\
\noalign{\medskip}
+\infty
&\qquad\text{otherwise}\,.
\end{cases}
\]
\begin{proposition}
\label{mu-lsc}
If $\mu\in\cmeas(\R^N)$ and $(u^{(n)})$ is a sequence in 
$W^{1,p}_{loc}(\R^N)\cap L^p(\R^N,\mu)$ satisfying
\[
\sup_n \left(\int |\nabla u^{(n)}|^p\,d\leb^N 
+ \int |u^{(n)}|^p\,d\mu\right) < +\infty
\]
and converging to some $u$ in $L^p_{loc}(\R^N)$, then 
$u\in W^{1,p}_{loc}(\R^N)\cap L^p(\R^N,\mu)$ and  $(u^{(n)})$ 
is weakly convergent to $u$ in $L^p(\R^N,\mu)$.
\end{proposition}
\begin{proof}
The sequence $(\nabla u^{(n)})$ is weakly convergent to 
$\nabla u$ in $L^p(\R^N;\R^N)$ and, up to a subsequence, 
$(u^{(n)})$ is weakly convergent to some $v$ in $L^p(\R^N,\mu)$.
If we consider
\[
C = \left\{(w\bigl|_{B_1(0)},\nabla w,w):\,\,
\text{$w\in W^{1,p}_{loc}(\R^N)\cap L^p(\R^N,\mu)$
and $\nabla w \in L^p(\R^N;\R^N)$}\right\}
\]
as a convex subset of 
$L^p(B_1(0))\times L^p(\R^N;\R^N)\times L^p(\R^N,\mu)$, we have 
that $(u\bigl|_{B_1(0)},\nabla u,v)$ belongs to the weak closure 
of~$C$, as $(u^{(n)}\bigl|_{B_1(0)},\nabla u^{(n)},u^{(n)})\in C$.
Then there exists a sequence 
$(w^{(n)}\bigl|_{B_1(0)},\nabla w^{(n)},w^{(n)})$ in $C$
strongly converging to $(u\bigl|_{B_1(0)},\nabla u,v)$.
Up to a subsequence, $w^{(n)}\to u$ q.e. in $\R^N$,
hence $\mu$-a.e. in~$\R^N$.
Then $v=u$ $\mu$-a.e. in $\R^N$.
\end{proof}
\begin{theorem}
\label{Gammalimthm}
If $(\mu^{(n)})$ is locally $\gamma$-convergent to $\mu$ in 
$\cmeas(\R^N)$, then 
\[
f_\mu(u)\leq\left(\Gamma-\liminf_{n\to\infty} 
f_{\mu^{(n)}}\right)(u)
\qquad \text{for all $u\in L^p_{loc}(\R^N)$}\,.
\]
\end{theorem}
\begin{proof}
By Proposition~\ref{prop:equivout} we may assume, without
loss of generality, that we have chosen for each $\mu^{(n)}$ 
and for $\mu$ the outer regular representative.
\par
Let $(u^{(n)})$ be a sequence converging to $u$ in 
$L^{p}_{loc}(\R^N)$ with
\[
\liminf_{n\to\infty} f_{\mu^{(n)}}(u^{(n)}) =
\left(\Gamma-\liminf_{n\to\infty} 
f_{\mu^{(n)}}\right)(u) \,.
\]
Without loss of generality, we may assume that this value
is not $+\infty$.
Up to a subsequence, it follows that 
$u^{(n)}, u\in W^{1,p}_{loc}(\R^N)$,
\[
\sup_n \left(\int |\nabla u^{(n)}|^p\,d\leb^N +
\int |u^{(n)}|^p\,d\mu^{(n)}\right) < +\infty
\]
and $(\nabla u^{(n)})$ is weakly convergent to $\nabla u$
in $L^p(\R^N;\R^N)$.
\par
If we define $b:\R^N\rightarrow ]0,+\infty[$ by
\[
b(x) = 
2^{-j}\,
\left(1+
\sup_n \int_{\{j\leq |x| < j+1\}} |u^{(n)}|^p\,d\leb^N\right)^{-1}
\quad\text{if $j\geq 0$ and $j\leq |x| < j+1$}\,,
\]
then $b\in L^\infty(\R^N)$, with $\essinf\limits_K b>0$ for 
all compact subsets $K$ of $\R^N$, and 
$|u^{(n)}|^p\, b, |u|^p\, b \in L^1(\R^N)$ with
\[
\lim_{n\to\infty} \int |u^{(n)}-u|^p\, b\,d\leb^N = 0\,.
\]
Now fix $k\in\N$ and define
\[
u^{(n)}_k=\argmin_{v\in W^{1,p}_{loc}(\R^N)}\left\{
\frac{k}{p}\,\int |u^{(n)}-v|^p\,b\,d\leb^N 
+ f_{\mu^{(n)}}(v)\right\}\,,
\]
as the above minimization problem admits one 
and only one minimizer.
Then $u^{(n)}_k\in L^p(\R^N,\mu^{(n)})$, 
\[
\frac{k}{p}\,\int |u^{(n)}-u^{(n)}_k|^p\,b\,d\leb^N
+ f_{\mu^{(n)}}(u^{(n)}_k) \leq f_{\mu^{(n)}}(u^{(n)})\,,
\]
up to a subsequence $(u^{(n)}_k)$ is convergent in 
$L^p_{loc}(\R^N)$ to some $v_k\in W^{1,p}_{loc}(\R^N)$ and
$(\nabla u^{(n)}_k)$ is weakly convergent to $\nabla v_k$ in
$L^p(\R^N;\R^N)$.
Moreover
\begin{multline*}
\int |\nabla u^{(n)}_k|^{p-2}\nabla u^{(n)}_k\cdot\nabla v\,d\leb^N
+ \int |u^{(n)}_k|^{p-2}u^{(n)}_k v\,d\mu^{(n)} \\
= k \int |u^{(n)}-u^{(n)}_k|^{p-2}(u^{(n)}-u^{(n)}_k)v\,b\,d\leb^N 
\qquad\text{for all $v\in W^{1,p}_{c}(\R^N)\cap 
L^p(\R^N,\mu^{(n)})$}\,.
\end{multline*}
Since $(|u^{(n)}-u^{(n)}_k|^{p-2}(u^{(n)}-u^{(n)}_k)b)$ is 
strongly convergent to $|u-v_{k}|^{p-2}(u-v_{k})b$ in 
$L^{p'}_{loc}(\R^N)$, from~\cite[Theorems~6.3 and~6.11]{dmmu} 
we infer that $v_k\in L^p_{loc}(\R^N,\mu)$ and
\[
\lim_{n\to\infty} \left[
\int |\nabla u^{(n)}_k|^p\varphi\,d\leb^n 
+ \int |u^{(n)}_k|^p \varphi\,d\mu^{(n)}\right] =
\int |\nabla v_{k}|^p\varphi\,d\leb^n 
+ \int |v_{k}|^p \varphi\,d\mu
\]
for all $k\in\N$ and $\varphi\in C_c(\R^N)$.
\par
In particular, if $\varphi\in C_c(\R^N)$ with 
$0\leq \varphi\leq 1$, we have
\[
\begin{split}
\liminf_{n\to\infty} f_{\mu^{(n)}}(u^{(n)}) &\geq
\liminf_{n\to\infty} \biggl[\frac{k}{p}\,
\int |u^{(n)}-u^{(n)}_k|^p\,b\,d\leb^N
+ f_{\mu^{(n)}}(u^{(n)}_k)\biggr]\\
&\geq
\liminf_{n\to\infty} \biggl[\frac{k}{p}\,
\int |u^{(n)}-u^{(n)}_k|^p\,b\,d\leb^N
+ \frac{1}{p}\,\int |\nabla u^{(n)}_k|^p\varphi\,d\leb^N \\
&\qquad\qquad\qquad\qquad\qquad\qquad\qquad\qquad
+ \frac{1}{p}\,\int |u^{(n)}_k|^p\varphi\,d\mu^{(n)}\biggr] \\
&\geq
\frac{k}{p}\,\int |u-v_{k}|^p\,b\,d\leb^N
+ \frac{1}{p}\,\int |\nabla v_{k}|^p\varphi\,d\leb^N +
\frac{1}{p}\,\int |v_{k}|^p\varphi\,d\mu\,.
\end{split}
\]
By the arbitrariness of $\varphi$, we infer that
$v_k \in L^p(\R^N,\mu)$ and
\[
\liminf_{n\to\infty} f_{\mu^{(n)}}(u^{(n)}) \geq
\frac{k}{p}\,\int |u-v_{k}|^p\,b\,d\leb^N + f_\mu(v_k)
\]
for all $k\in\N$, whence
\[
\lim_{k\to\infty} \int |u-v_{k}|^p\,b\,d\leb^N = 0\,.
\]
In particular, $(v_k)$ is convergent to $u$
in $L^p_{loc}(\R^N)$.
By the lower semicontinuity of $f_\mu$ we conclude that
\[
\liminf_{n\to\infty} f_{\mu^{(n)}}(u^{(n)}) \geq
\liminf_{k\to\infty} f_\mu(v_k) \geq f_\mu(u)
\]
and the proof is complete.
\end{proof}
\begin{example}
Let $p<N$ and let $\mu^{(n)}=\infty_{\R^N\setminus B_n(0)}$.
Then $(\mu^{(n)})$ is locally $\gamma$-convergent to $\mu=0$,
but it is false that
\[
f_{\mu}(u) = \left(\Gamma-\lim_{n\to\infty} 
f_{\mu^{(n)}}\right)(u)
\qquad \text{for all $u\in L^p_{loc}(\R^N)$}\,.
\]
Actually, if we take $u=1$, we have $f_\mu(u)=0$ but it is
impossible to find a sequence $(u^{(n)})$ converging to~$1$
in~$L^p_{loc}(\R^N)$ with $f_{\mu^{(n)}}(u^{(n)})\to 0$,
because each $u^{(n)}$ has compact support, which implies that
$(u^{(n)})$ is convergent to $0$ in~$L^{p^*}(\R^N)$.
\end{example}
If $p\geq N$, we will see by Proposition~\ref{carattfmu0}
and Theorem~\ref{GammalimthmW1pc} that the assertion is true.
\begin{proposition}
\label{coerc}
Let $(\mu^{(n)})$ be a sequence in $\cmeas(\R^N)$ and
$\mu \in \cmeas(\R^N)$ with $\mu(\R^N) > 0$ and
\[
f_\mu(u)\leq\left(\Gamma-\liminf_{n\to\infty} 
f_{\mu^{(n)}}\right)(u)
\qquad \text{for all $u\in L^p_{loc}(\R^N)$}\,.
\]
\indent
Then, for every sequence $(u^{(n)})$ in $W^{1,p}_{loc}(\R^N)$ 
such that
\[
\sup_n f_{\mu^{(n)}}(u^{(n)}) < +\infty \,,
\]
there exist $u\in W^{1,p}_{loc}(\R^N)$ and a subsequence 
$(u^{(n_k)})$ converging to $u$ in $L^p_{loc}(\R^N)$.
\end{proposition}
\begin{proof}
It is enough to prove that
\[
\sup_n \int_{B_1(0)} |u^{(n)}|^p\,d\leb^N < +\infty\,.
\]
Assume, for the sake of contradiction, that, up to a subsequence,
we have
\[
\lim_{n\to\infty} \int_{B_1(0)} |u^{(n)}|^p\,d\leb^N = +\infty\,.
\]
Then a suitably rescaled sequence $(v^{(n)})$ satisfies
\[
\int_{B_1(0)} |v^{(n)}|^p\,d\leb^N = 1\,,\qquad
\lim_{n\to\infty} \left(\int |\nabla v^{(n)}|^p\,d\leb^N +
\int |v^{(n)}|^p\,d\mu^{(n)} \right) = 0\,.
\]
It follows that, up to a subsequence, $(v^{(n)})$ is convergent 
to some $v$ in $W^{1,p}_{loc}(\R^N)$, whence
\[
\int |\nabla v|^p\,d\leb^N + \int |v|^p\,d\mu = 0 \,,
\]
so that $v$ is a constant with $v=0$, as $\mu(\R^N)>0$.
On the other hand
\[
\int_{B_1(0)} |v|^p\,d\leb^N = 1
\]
and a contradiction follows.
\end{proof}

\subsection{Convergence of functionals}
In order to relate the local $\gamma$-convergence of 
measures in $\cmeas(\R^N)$ with the $\Gamma$-convergence
of functionals on $\R^N$, we need to introduce, roughly 
speaking, a homogeneous Dirichlet-type condition at infinity.
\par
For every $\mu\in \cmeas(\R^N)$, we first define the convex 
functional
\[
\widetilde f_{\mu,0}:L^p_{loc}(\R^N)\rightarrow[0,+\infty]
\]
by
\[
\widetilde f_{\mu,0}(u) =
\begin{cases}
\displaystyle{
\frac{1}{p}\,\int |\nabla u|^p\,d\leb^N +
\frac{1}{p}\,\int |u|^p\,d\mu} 
&\qquad\text{if } u\in W^{1,p}_{c}(\R^N)\,,\\
\noalign{\medskip}
+\infty
&\qquad\text{otherwise}\,,
\end{cases}
\]
then we denote by $f_{\mu,0}$ its lower semicontinuous
envelope.
\begin{lemma}
\label{pgeqN}
If $p\geq N$, there exists a sequence $(\vartheta_n)$
in $C^\infty_c(\R^N)$ such that $0\leq \vartheta_n 	\leq 1$,
\begin{alignat*}{3}
&\lim_{n\to\infty} \vartheta_n = 1
&&\qquad\text{uniformly on compact subsets of $\R^N$}\,,\\
&\lim_{n\to\infty} \nabla \vartheta_n = 0
&&\qquad\text{strongly in $L^p(\R^N;\R^N)$}\,.
\end{alignat*}
\end{lemma}
\begin{proof}
Consider the space
\[
\mathcal X =\left\{u\in W^{1,p+1}_{loc}(\R^N):\,\,
\nabla u\in L^p(\R^N;\R^N)\cap L^{p+1}(\R^N;\R^N)\,,\,\,
\int_{B_1(0)}u\,d\leb^N=0\right\}\,.
\]
Then $\mathcal X$ is a reflexive Banach space, when 
endowed with the norm
\[
\|u\| = \|\nabla u\|_p + \|\nabla u\|_{p+1}\,.
\]
Let $\hat\varphi\in C^{\infty}(\R^N)$ be such that
$0\leq \hat\varphi\leq 1$, $\hat\varphi (x)=0$ for 
$|x|\leq 1$ and $\hat\varphi(x)=-1$ for $|x|\geq 2$.
Then define $\hat\varphi_n(x)=\hat\varphi(x/n)$ for all 
$n\geq 1$.
Of course $\hat\varphi_n\in \mathcal{X}$ and it is easily seen 
that $(\hat\varphi_n)$ is weakly convergent to~$0$ in 
$\mathcal X$ (by the way, strongly if $p>N$).
Therefore $0$ belongs to the weak closure of the convex set
\[
\mathrm{conv} \{\hat\varphi_n:\,\, n\in\N\}\,.
\]
Then there exists a sequence $(\varphi_n)$ in such a convex 
set strongly converging  to $0$ in $\mathcal X$.
In particular, each $\varphi_n$ satisfies $\varphi_n=-1$ 
outside some compact subset of $\R^N$, $(\nabla \varphi_n)$ is 
strongly convergent to $0$ in $L^p(\R^N;\R^N)$ and $(\varphi_n)$ 
is convergent to $0$ uniformly on compact subsets of 
$\R^N$, as $p+1>N$.
\par
It follows that $\vartheta_n=1+\varphi_n$ has the
required properties.
\end{proof}
\begin{proposition}
\label{carattfmu0}
If $p < N$, we have 
\[
f_{\mu,0}(u) =
\begin{cases}
f_\mu(u) 
&\qquad\text{if } u\in W^{1,p}_{loc}(\R^N)\cap L^{p^*}(\R^N)\,,\\
\noalign{\medskip}
+\infty
&\qquad\text{otherwise}\,.
\end{cases}
\]
If $p\geq N$, we have 
\[
f_{\mu,0}(u) = f_\mu(u)
\qquad\text{for all $u\in L^p_{loc}(\R^N)$}\,.
\]
\end{proposition}
\begin{proof}
Since $f_\mu$ is lower semicontinuous, we clearly have
\[
f_\mu(u) \leq f_{\mu,0}(u) \leq
\widetilde f_{\mu,0}(u)
\qquad\text{for all $u\in L^p_{loc}(\R^N)$}\,.
\]
Assume first that $p<N$.
Let $u\in W^{1,p}_{loc}(\R^N)\cap L^{p^*}(\R^N)$
and let $\vartheta\in C^\infty(\R)$ be such that
$0\leq \vartheta\leq 1$, $\vartheta'\leq 0$, $\vartheta (s)=1$ 
for $s\leq 1$ and $\vartheta(s)=0$ for $s\geq 2$.
If we set $\vartheta_n(x) = \vartheta(|x|/n)$, we have
$0\leq\vartheta_n \leq \vartheta_{n+1}\leq 1$,
$\vartheta_n u\in W^{1,p}_c(\R^N)$ and $(\vartheta_n u)$
is convergent to $u$ in $L^p_{loc}(\R^N)$.
We also have
\[
f_{\mu,0}(\vartheta_n u) = 
\frac{1}{p}\,\int 
|\vartheta_n \nabla u + u \nabla \vartheta_n|^p\,d\leb^N +
\frac{1}{p}\,\int |\vartheta_n u|^p\,d\mu\,.
\]
It is easily seen that $(\nabla \vartheta_n)$ is bounded
in $L^N(\R^N;\R^N)$ and convergent to $0$ a.e. in $\R^N$.
Moreover, for every $\varepsilon>0$ there exists
$C_\varepsilon>0$ such that
\[
|u(x) \xi| \leq 
C_\varepsilon |u(x)|^{p^*/p} + \varepsilon |\xi|^{N/p} 
\qquad\text{for a.a. $x\in\R^n$ and all $\xi\in\R^N$}\,.
\]
It follows (see 
e.g.~\cite[Lemma~4.2]{degiovanni_lancelotti2007})
that $(u\nabla \vartheta_n)$ is strongly convergent
to $0$ in $L^p(\R^N;\R^N)$.
By the lower semicontinuity of $f_{\mu,0}$ we infer that
\[
f_{\mu,0}(u) \leq \lim_{n\to\infty} f_{\mu,0}(\vartheta_n u) =
\frac{1}{p}\,\int |\nabla u|^p\,d\leb^N +
\frac{1}{p}\,\int |u|^p\,d\mu = f_\mu(u) \,.
\]
Now it remains only to show that $f_{\mu,0}(u)=+\infty$
whenever $u\in W^{1,p}_{loc}(\R^N)\setminus L^{p^*}(\R^N)$.
Assume, for the sake of contradiction, that $f_{\mu,0}(u)<+\infty$
and let $(u_n)$ be a sequence converging to $u$ in 
$L^p_{loc}(\R^N)$ with 
$\widetilde f_{\mu,0}(u)(u_n)\to f_{\mu,0}(u)$, whence
$u_n\in W^{1,p}_c(\R^N)$ eventually as $n\to\infty$.
Since $(\nabla u_n)$ is bounded in $L^p(\R^N;\R^N)$,
we have that $(u_n)$ is bounded in $L^{p^*}(\R^N)$.
Therefore, $u\in L^{p^*}(\R^N)$ and a contradiction follows.
\par
If $p\geq N$, let 
$u\in W^{1,p}_{loc}(\R^N)\cap L^p(\R^N,\mu)$ with
$\nabla u\in L^p(\R^N;\R^N)$, let
$(\vartheta_n)$ be a sequence as in Lemma~\ref{pgeqN} and 
let $(c_n)$ be sequence of positive numbers increasing to
$+\infty$ such that $\|c_n\nabla\vartheta_n\|_p\to 0$.
If we define
\[
u_n = \min\{\max\{u, - c_n\vartheta_n\},c_n\vartheta_n\}\,,
\]
we have
\[
|\nabla u_n|^p \leq |\nabla u|^p + |c_n \nabla\vartheta_n|^p
\qquad\text{a.e. in $\R^N$}\,.
\]
Then we have that $(u_n)$ is convergent to $u$ in
$L^p_{loc}(\R^N)$ with $u_n \in W^{1,p}_c(\R^N)$ and
\[
f_{\mu,0}(u) \leq \lim_{n\to\infty} f_{\mu,0}(u_n) =
\frac{1}{p}\,\int |\nabla u|^p\,d\leb^N +
\frac{1}{p}\,\int |u|^p\,d\mu = f_\mu(u) \,,
\]
whence the assertion.
\end{proof}
Before dealing with the main result of this subsection, 
we need the following.
\begin{proposition}
\label{uniqueness}
Let $\mu, \nu \in \cmeas(\R^N)$ be such that
$f_{\mu,0}\leq f_{\nu,0}$.
Then $\mu\leq \nu$.
\end{proposition}
\begin{proof}
We have
\[
\int |u|^p\,d\mu\leq \int |u|^p\,d\nu
\qquad\text{for all $u\in W^{1,p}_c(\R^N)$}\,,
\]
whence
\[
\int |u|^p\,d\mu\leq \int |u|^p\,d\nu
\qquad\text{for all $u\in W^{1,p}_{loc}(\R^N)$}\,.
\]
For every Borel and $p$-quasi open subset $A$ of $\Omega$, there 
exists $u\in W^{1,p}_{loc}(\R^N)$ such that $A=\{u>0\}$ by 
Proposition~\ref{prop:proppfine}.
It follows
\[
\int_\Omega (\min\{k\,u^+,1\})^p\,d\mu \leq
\int_\Omega (\min\{k\,u^+,1\})^p\,d\nu
\qquad\text{for all $k\in\N$}\,,
\]
whence $\mu(A)\leq \nu(A)$ going to the limit as $k\to\infty$.
\end{proof}
\begin{corollary}
Let $\mu, \nu \in \cmeas(\R^N)$ be such that
$f_{\mu}\leq f_{\nu}$.
Then $\mu\leq \nu$.
\end{corollary}
\begin{proof}
It follows from Propositions~\ref{carattfmu0}
and~\ref{uniqueness}.
\end{proof}
The main purpose of this subsection is to show that a 
sequence of measures $(\mu^{(n)})$ is convergent to $\mu$ 
in $\cmeas(\R^N)$ if and only if $f_{\mu^{(n)},0}$ is 
$\Gamma-$convergent to $f_{\mu,0}$ in $L^p_{loc}(\R^N)$ .
In the case $p=2$ a similar result was obtained by Bucur 
in~\cite[Appendix]{bucc}; our more general case requires a 
more involved proof.
\begin{theorem}
\label{GammalimthmW1pc}
A sequence $(\mu^{(n)})$ is locally $\gamma$-convergent to 
$\mu$ in $\cmeas(\R^N)$ if and only if 
\[
f_{\mu,0}(u)=
\left(\Gamma-\lim_{n\to\infty} f_{\mu^{(n)},0}\right)(u)
\qquad\text{for all $u\in L^p_{loc}(\R^N)$}\,.
\]
\end{theorem}
\begin{proof}
Again, by Proposition~\ref{prop:equivout} we may assume, without
loss of generality, that we have chosen for each $\mu^{(n)}$ 
and for $\mu$ the outer regular representative.
\par
Assume first that $(\mu^{(n)})$ is locally $\gamma$-convergent 
to $\mu$.
\par\noindent
{\it Step 1. $\Gamma-$liminf inequality.}
By Proposition~\ref{carattfmu0} and Theorem~\ref{Gammalimthm},
we have to treat only the case $p<N$.
We take a sequence $(u^{(n)})$ converging to $u$ in 
$L^{p}_{loc}(\R^N)$ with
\[
\liminf_{n\to\infty} f_{\mu^{(n)}}(u^{(n)}) =
\left(\Gamma-\liminf_{n\to\infty} 
f_{\mu^{(n)},0}\right)(u)
\]
and, without loss of generality, we may 
assume that $u^{(n)}\in W^{1,p}_{loc}(\R^N)\cap L^{p^*}(\R^N)$ 
with
\[
\sup_n \left(\int |\nabla u^{(n)}|^p\,d\leb^N +
\int |u^{(n)}|^p\,d\mu^{(n)}\right) < +\infty \,.
\]
Since $(\nabla u^{(n)})$ is bounded in $L^p(\R^N;\R^N)$,
we have that $(u^{(n)})$ is bounded in $L^{p^*}(\R^N)$.
Therefore $u\in L^{p^*}(\R^N)$ and the assertion follows again 
from Proposition~\ref{carattfmu0} and Theorem~\ref{Gammalimthm}.
\par\noindent
{\it Step 2. $\Gamma-$limsup inequality.}
Let $u\in L^p_{loc}(\R^N)$ with $f_{\mu,0}(u)<+\infty$, let
$\beta > f_{\mu,0}(u)$ and let $U$ be an open neighborhood of 
$u$ in $L^p_{loc}(\R^N)$.
Let $z\in U\cap W^{1,p}_c(\R^N)$ with $f_{\mu,0}(z)\leq \beta$
and let $R>0$ be such that $z=0$ a.e. 
in~$\R^N\setminus B_R(0)$.
For every $k\in\N$, define
\[
u_k=\argmin_{v\in W^{1,p}_{0}(B_R(0))}\left\{
\frac{k}{p}\int_{B_R(0)} |z-v|^p\,d\leb^N
+f_{\mu,0}(v)\right\}\,,
\]
as the above minimization problem admits one 
and only one minimizer.
Then, testing with $v=z$ we obtain the upper bound
\[
\frac{k}{p}\int_{B_R(0)} |z-u_k|^p\,d\leb^N + f_{\mu,0}(u_k)
\leq f_{\mu,0}(z) \leq \beta\,,
\]
for all $k\in\N$. 
Thus $(u_k)$ is convergent to $z$ in~$L^{p}(B_R(0))$.
Let us fix $k$ large enough to have $u_k\in U$.
\par
Then $u_k\in W^{1,p}_{0}(B_R(0))\cap L^p(B_R(0),\mu)$ and the  
Euler-Lagrange equation for the minimization problem 
defining $u_k$ yields 
\begin{multline*}
\int_{B_R(0)} |\nabla u_{k}|^{p-2}\nabla u_{k}
\cdot\nabla v\,d\leb^N
+ \int_{B_R(0)} |u_{k}|^{p-2}u_{k}v\,d\mu \\
= k \int_{B_R(0)} |z-u_{k}|^{p-2}(z-u_{k})v\,d\leb^N 
\qquad
\text{for all $v\in W^{1,p}_{0}(B_R(0))\cap L^p(B_R(0),\mu)$}\,.
\end{multline*}
Now, for every $n\in\N$, let 
\[
u^{(n)}_k=\argmin_{v\in W^{1,p}_{0}(B_R(0))}\left\{
f_{\mu^{(n)},0}(v)
- k \int_{B_R(0)}|z-u_k|^{p-2}(z-u_k)v\,d\leb^N
\right\}\,,
\]
as again this problem has one and only one minimizer.
Then we have 
$u^{(n)}_k\in W^{1,p}_{0}(B_R(0))\cap L^p(B_R(0),\mu^{(n)})$
and
\begin{multline*}
\int_{B_R(0)} |\nabla u^{(n)}_k|^{p-2}\nabla u^{(n)}_k
\cdot\nabla v\,d\leb^N
+ \int_{B_R(0)} |u^{(n)}_k|^{p-2}u^{(n)}_k v\,d\mu^{(n)} \\
= k \int_{B_R(0)} |z-u_{k}|^{p-2}(z-u_{k})v\,d\leb^N 
\qquad\text{for all 
$v\in W^{1,p}_{0}(B_R(0))\cap L^p(B_R(0),\mu^{(n)})$}\,.
\end{multline*}
From~\cite[Theorem~6.3]{dmmu} we infer that $(u^{(n)}_k)$
is weakly convergent to $u_k$ in $W^{1,p}_{0}(B_R(0))$.
In particular, we have $u^{(n)}_k\in U$ eventually as
$n\to\infty$.
Moreover, it is
\begin{alignat*}{3}
&f_{\mu^{(n)},0}(u^{(n)}_k) &&=\frac{1}{p}\,
\int_{B_R(0)} |\nabla u^{(n)}_k|^p\,d\leb^N
+ \frac{1}{p}\,\int_{B_R(0)} |u^{(n)}_k|^p\,d\mu^{(n)} \\
&&&= \frac{k}{p}\, 
\int_{B_R(0)} |z-u_{k}|^{p-2}(z-u_{k})u^{(n)}_k\,d\leb^N\,, \\
&f_{\mu,0}(u_{k}) &&=
\frac{1}{p}\,\int_{B_R(0)} |\nabla u_{k}|^p\,d\leb^N
+ \frac{1}{p}\,\int_{B_R(0)} |u_{k}|^p\,d\mu \\
&&&= \frac{k}{p}\, 
\int_{B_R(0)} |z-u_{k}|^{p-2}(z-u_{k})u_{k}\,d\leb^N \,.
\end{alignat*}
Therefore, having in mind the topological definition of 
$\Gamma-$limsup, we obtain
\[
\limsup_{n\to\infty}\left(\inf_{v\in U}f_{\mu^{(n)},0}(v)\right)
\leq \lim_{n\to\infty} f_{\mu^{(n)},0}(u^{(n)}_k)
= f_{\mu,0}(u_{k})\leq \beta
\]
and the assertion follows from the arbitrariness of 
$\beta$ and $U$.
\par
Assume now that
\[
f_{\mu,0}(u)=
\left(\Gamma-\lim_{n\to\infty} f_{\mu^{(n)},0}\right)(u)
\qquad\text{for all $u\in L^p_{loc}(\R^N)$}\,.
\]
Up to a subsequence, $(\mu^{(n)})$ is locally $\gamma$-convergent
to some $\nu$ in $\cmeas(\R^N)$.
By the previous step, we infer that
\[
f_{\nu,0}(u)=
\left(\Gamma-\lim_{n\to\infty} f_{\mu^{(n)},0}\right)(u)
\qquad\text{for all $u\in L^p_{loc}(\R^N)$}\,,
\]
whence $f_{\nu,0}=f_{\mu,0}$.
By Proposition~\ref{uniqueness} we have $\nu=\mu$
and the assertion follows.
\end{proof}
We conclude the section by highlighting some further 
consequences of the local $\gamma$-convergence.
\begin{corollary}
\label{cor:leq}
Let $(\mu^{(n)})$ be locally $\gamma$-convergent to $\mu$
and $(\nu^{(n)})$ be locally $\gamma$-convergent to $\nu$
in $\cmeas(\R^N)$ with $\mu^{(n)} \leq \nu^{(n)}$ for
all $n\in\N$.
\par
Then $\mu\leq\nu$.
\end{corollary}
\begin{proof}
It follows from Theorem~\ref{GammalimthmW1pc} and 
Proposition~\ref{uniqueness}.
\end{proof}
\begin{corollary}
\label{cor:lsc}
If $(\mu^{(n)})$ is locally $\gamma$-convergent to $\mu$ in 
$\cmeas(\R^N)$, then 
\begin{gather*}
\leb^N(A_\mu) \leq \liminf_{n\to\infty}\,
\leb^N(A_{\mu^{(n)}})\,,\\
\int \Psi(V_\mu)\,d\leb^N \leq \liminf_{n\to\infty}\,
\int\Psi(V_{\mu^{(n)}})\,d\leb^N\,,
\end{gather*}
whenever $\Psi:[0,+\infty]\rightarrow[0,+\infty]$ is a strictly
decreasing and continuous function such that there exists 
$\alpha>1$ with $\left\{s\mapsto \Psi^{-1}(s^\alpha)\right\}$ 
convex on 
$\left\{s\geq 0:\,\,s^\alpha \in\Psi([0,+\infty])\right\}$.
\end{corollary}
\begin{proof}
If $\Omega$ is a bounded and open subset of $\R^N$ and
$(\mu^{(n)})$ is $\gamma$-convergent to $\mu$ in 
$\cmeas(\Omega)$, then
\[
\leb^N(A_\mu) \leq \liminf_{n\to\infty}\,
\leb^N(A_{\mu^{(n)}})\,,
\]
as
\[
\leb^N(A_\nu) = \sup_{k\in\N}\,
\int_\Omega \,\min\left\{k\,w_\nu(\Om),1\right\}\,d\leb^N
\]
for all $\nu\in\cmeas(\Om)$ by Remark~\ref{rem:Amuw}.
\par
On the other hand, for every $\nu\in\cmeas(\R^N)$ we have
\[
\leb^N(A_\nu) = \sup_{k\in\N}\,\leb^N(A_\nu\cap B_k(0))
= \sup_{k\in\N}\,\leb^N\left(A_{\nu\bigl|_{\bor(B_k(0))}}\right)
\]
and the first assertion follows.
\par
When dealing with the second assertion, we follow an argument 
inspired by~\cite[Theorem~4.1]{bugeruve}.
Without loss of generality, we assume that 
\[
\sup_n\int\Psi(V_{\mu^{(n)}})\,d\leb^N<+\infty
\] 
and we set $v_n=(\Psi(V_{\mu^{(n)}}))^{1/\alpha}$, so that
$(v_n)$ is a bounded sequence in $L^{\alpha}(\R^N)$, thus 
(up to subsequences) weakly convergent to some $v$ in 
$L^\alpha(\R^N)$. 
On the other hand, by Theorem~\ref{GammalimthmW1pc}, for every 
$u\in W^{1,p}_{c}(\R^N)$ there exists a sequence $(u^{(n)})$ in 
$W^{1,p}_{loc}(\R^N)$  converging to $u$ in $L^p_{loc}(\R^N)$ 
such that 
\[
\int |\nabla u|^p\,d\leb^N + \int |u|^p\,d\mu
=
\limsup_{n\rightarrow \infty}\,\left(
\int |\nabla u^{(n)}|^p\,d\leb^N 
+ \int |u^{(n)}|^p \,d\mu^{(n)}\right) \,.
\]
Combining assertion~\ref{haAVmu} of Proposition~\ref{prop:AVmu} 
with the strong-weak lower semicontinuity theorem 
of~\cite{ioffelsc}, we infer that
\[
\begin{split}
\int |\nabla u|^p\,d\leb^N + \int |u|^p\,d\mu
&\geq
\limsup_{n\rightarrow \infty}\,\left(
\int |\nabla u^{(n)}|^p\,d\leb^N 
+ \int |u^{(n)}|^p V_{\mu^{(n)}}\,d\leb^N\right) \\
&=
\limsup_{n\rightarrow \infty}\,\left(
\int |\nabla u^{(n)}|^p\,d\leb^N 
+ \int|u^{(n)}|^p\Psi^{-1}(v_n^\alpha)\,d\leb^N\right) \\
&\geq 
\int |\nabla u|^p\,d\leb^N
+ \int|u|^p\Psi^{-1}(v^\alpha)\,d\leb^N\,,
\end{split}
\]   
as the function $\left\{s\mapsto\Psi^{-1}(s^\alpha)\right\}$ 
is convex.
\par
By Proposition~\ref{uniqueness}, we infer
\[
\mu(A) \geq 
\int_A \Psi^{-1}(v^\alpha)\,d\leb^N
\qquad\text{for all $A\in \bor(\R^N)$ with $A$ $p$-quasi open}\,,
\]
whence $V_\mu\geq \Psi^{-1}(v^\alpha)$ $\leb^N$-a.e. in~$\R^N$
by~\ref{hbAVmu} and~\ref{hdAVmu} of Proposition~\ref{prop:AVmu}.
\par
Since $\Psi$ is strictly decreasing, we infer that
$\Psi(V_\mu)\leq v^\alpha$ $\leb^N$-a.e. in $\R^N$, whence
\[
\int \Psi(V_\mu)\,d\leb^N \leq \int v^\alpha\,d\leb^N 
\leq \liminf_{n\to \infty}\,\int v_n^\alpha\,d\leb^N
= \liminf_{n\to\infty}\,\int\Psi(V_{\mu^{(n)}})\,d\leb^N
\]
and the second assertion also follows.
\end{proof}
%

%--------------------------------------------------------------------

\section{Towards variational eigenvalues for sign-changing 
capacitary measures}
\label{sect:preeig}
Let $\mu, \mup, \mum\in\cmeas(\R^N)$.
In this section we introduce the candidate 
``variational eigenvalues'' for the problem
\[
\begin{cases}
-\Delta_pu+|u|^{p-2}u\,\mu = \la |u|^{p-2}u(\nu_1-\nu_2)
\qquad\text{in $\R^N$}\,,\\
\noalign{\medskip}
\displaystyle{
\int |u|^p\,d\nu_2 < \int |u|^p\,d\nu_1} \,,
\end{cases}
\]
and prove some basic properties.
\par
Consider an index $i$ as in Section~\ref{sect:convminmax} and 
the related families $\mathcal{K}_m$  and $\mathcal{K}_m^{fin}$ 
with respect to the metrizable and locally convex topological 
vector space $\mathcal{X}=L^p_{loc}(\R^N)$.
Let $f_{\mu,0}:L^p_{loc}(\R^N)\rightarrow[0,+\infty]$ be the 
functional introduced in Section~\ref{sect:convergence}
and define $g_1,g_2,R:L^p_{loc}(\R^N)\rightarrow[0,+\infty]$ by
\[
g_j(u) =
\begin{cases}
\displaystyle{\frac{1}{p}\,\int |u|^p\,d\nu_j}
&\qquad\text{if $u\in W^{1,p}_{loc}(\R^N)$}\,,\\
\noalign{\medskip}
+\infty
&\qquad\text{otherwise}\,,
\end{cases}
\]
\[
R(u) =
\begin{cases}
f_{\mu,0}(u)
&\qquad\text{if $1+g_2(u)\leq g_1(u)<+\infty$}\,,\\
\noalign{\medskip}
+\infty
&\qquad\text{otherwise}\,.
\end{cases}
\]
Then, for every integer $m\geq 1$, set
\[
\lambda_m^p(\mu,\mup,\mum) =
\inf_{K\in\mathcal{K}_m} 
\sup_{u\in K}\, R(u) \,.
\]
\begin{remark}\label{rem:monot}
It is immediate from the definition to note that, if 
$\mu_1,\mu_2\in\cmeas(\R^N)$ with $\mu_1\leq \mu_2$, then 
\[
\la_m^p(\mu_1,\nu_1,\nu_2)\leq \la_m^p(\mu_2,\nu_1,\nu_2)
\qquad\text{for all $\nu_1,\nu_2\in\cmeas(\R^N)$ and $m\geq 1$}\,.
\] 
\end{remark}
\begin{proposition}
\label{strongappr}
For every $u\in W^{1,p}_{loc}(\R^N)$ with
\[
f_{\mu,0}(u) < +\infty\,,\qquad
\int |u|^p\,d\mup < +\infty\,,\qquad
\int |u|^p\,d\mum < +\infty\,,
\]
there exists a sequence $(u_n)$ in 
$W^{1,p}_c(\R^N)\cap L^p(\R^N,\mu)\cap L^p(\R^N,\mup)
\cap L^p(\R^N,\mum)$ converging to $u$ in $L^p_{loc}(\R^N)$
such that
\[
\lim_{n\to\infty} \biggl(
\int |\nabla u_n-\nabla u|^p\,d\leb^N 
+ \int |u_n-u|^p\,d\mu
+ \int |u_n-u|^p\,d\mup
+ \int |u_n-u|^p\,d\mum\biggr) = 0\,.
\]
\end{proposition}
\begin{proof}
By Proposition~\ref{carattfmu0} we also have
$f_{\mu+\mup+\mum,0}(u)<+\infty$.
Therefore, there exists a sequence $(u_n)$ in 
$W^{1,p}_c(\R^N)\cap L^p(\R^N,\mu)\cap L^p(\R^N,\mup)
\cap L^p(\R^N,\mum)$ converging to $u$ in $L^p_{loc}(\R^N)$
such that
\begin{multline*}
\lim_{n\to\infty} \left(
\int |\nabla u_n|^p\,d\leb^N 
+ \int |u_n|^p\,d\mu
+ \int |u_n|^p\,d\mup
+ \int |u_n|^p\,d\mum\right) \\
= \int |\nabla u|^p\,d\leb^N 
+ \int |u|^p\,d\mu
+ \int |u|^p\,d\mup
+ \int |u|^p\,d\mum\,.
\end{multline*}
Taking into account Proposition~\ref{mu-lsc}, we have that 
$(\nabla u_n)$ is weakly convergent to $\nabla u$ 
in $L^p(\R^N;\R^N)$ and $(u_n)$ is weakly convergent to $u$
in $L^p(\R^N,\mu)$, $L^p(\R^N,\mup)$ and $L^p(\R^N,\mum)$ with
\begin{multline*}
\lim_{n\to\infty}\, \int |\nabla u_n|^p\,d\leb^N 
= \int |\nabla u|^p\,d\leb^N\,,\qquad
\lim_{n\to\infty}\,\int |u_n|^p\,d\mu
= \int |u|^p\,d\mu\,,\\
\lim_{n\to\infty}\, \int |u_n|^p\,d\mup
= \int |u|^p\,d\mup\,,\qquad
\lim_{n\to\infty}\, \int |u_n|^p\,d\mum
= \int |u|^p\,d\mum\,.
\end{multline*}
Then the assertion follows.
\end{proof}
\begin{proposition}
\label{prop:lambdafin}
The following facts hold:
\begin{enumerate}[label={\upshape\alph*)}, align=parleft, 
widest=iii, leftmargin=*]
\item[\mylabel{halambdafin}{\ha}]
if
\[
\left\{u\in W^{1,p}_c(\R^N):\,\,
\int |u|^p\,d\mu<+\infty\,,\,\,
\int |u|^p\,d\mum < \int |u|^p\,d\mup < +\infty\right\} 
= \emptyset\,,
\]
then we have
$\lambda_m^p(\mu,\mup,\mum)=+\infty$ for all $m\geq 1$;
\item[\mylabel{hblambdafin}{\hb}]
if 
\[
\left\{u\in W^{1,p}_c(\R^N):\,\,
\int |u|^p\,d\mu<+\infty\,,\,\,
\int |u|^p\,d\mum < \int |u|^p\,d\mup < +\infty\right\} 
\neq \emptyset \,,
\]
then we have $\lambda_1^p(\mu,\mup,\mum)<+\infty$;
\item[\mylabel{hclambdafin}{\hc}]
if 
\begin{multline*}
~\qquad
\biggl\{u\in W^{1,p}_c(\R^N):\,\,
\int |u|^p\,d\mu<+\infty\,,\,\,
\int |u|^p\,d\mum < \int |u|^p\,d\mup < +\infty\,,\,\,
\biggr. \\ \biggl.
\lim_{r\to 0}\,\int_{B_r(x)} |u|^p\,d\mup=0
\quad\text{for all $x\in\R^N$}
\biggr\} 
\neq \emptyset\,,
\end{multline*}
then we have
$\lambda_m^p(\mu,\mup,\mum)<+\infty$ for all $m\geq 1$.
\end{enumerate}
\end{proposition}
\begin{proof}
\par\noindent
\par\noindent
$(a)$~From Proposition~\ref{strongappr} it follows that
$R(u)=+\infty$ for all $u\in L^p_{loc}(\R^N)$, 
whence the assertion.
\par\noindent
$(b)$~If $u\in W^{1,p}_c(\R^N)$ satisfies
\[
\int |u|^p\,d\mu<+\infty\,,\,\,
\int |u|^p\,d\mum < \int |u|^p\,d\mup < +\infty\,,
\]
it is easily seen that $R(tu) < +\infty$ for some
$t>0$, whence the assertion.
\par\noindent
$(c)$~Let $u\in W^{1,p}_c(\R^N)$ with
\[
\int |u|^p\,d\mu<+\infty\,,\,\,
\int |u|^p\,d\mum < \int |u|^p\,d\mup < +\infty\,,\,\,
\lim_{r\to 0}\,\int_{B_r(x)} |u|^p\,d\mup=0
\quad\text{for all $x\in\R^N$}
\]
and let us choose a representative for $u$, $\mup$ and 
$\mum$.
\par
By substituting $u$ with
\[
u_k = \left(|u|-\frac{1}{k}\right)^+
\]
with $k$ large enough, we may assume that
$u\geq 0$ a.e. in $\R^N$ and that
\[
\mup\left(\left\{x\in\R^N:\,\,u(x)>0\right\}\right) 
< +\infty\,,\qquad
\mum\left(\left\{x\in\R^N:\,\,u(x)>0\right\}\right)
< +\infty \,.
\]
If we set
\[
\hat\nu(B) = 
\mup\left(B\cap\left\{x\in\R^N:\,\,u(x)>0\right\}\right) +
\mum\left(B\cap\left\{x\in\R^N:\,\,u(x)>0\right\}\right)
\qquad\text{for all $B\in\bor(\R^N)$}\,,
\]
we have that $\hat\nu$ is a positive Radon measure
on $\R^N$ and there exist two Borel functions
$\eta_1,\eta_2:\R^N\rightarrow[0,1]$ such that
\[
\int_B \eta_j\,d\hat\nu = \nu_j\left(
B\cap\left\{x\in\R^N:\,\,u(x)>0\right\}\right)
\qquad\text{for all $B\in\bor(\R^N)$}\,,
\]
whence
\[
\int u^p(\eta_1 - \eta_2)\,d\hat\nu >0 \,.
\]
We have
\[
\lim_{r\to 0^+}\,
\int_{B_r(x)} u^p(\eta_1 - \eta_2)\,d\hat\nu \leq
\lim_{r\to 0^+}\,
\int_{B_r(x)} u^p\,d\mup = 0
\qquad\text{for all $x\in \R^N$}\,.
\]
Therefore, if $x\in\R^N$ is a Lebesgue point of 
$u^p\,(\eta_1 - \eta_2)$ with respect to $\hat\nu$ such that 
$u(x)^p(\eta_1(x) - \eta_2(x))>0$
(see~\cite[Corollary~2.23]{ambrosio_fusco_pallara2000}),
we have $\hat\nu(\{x\})=0$.
Then, for every $m\geq 1$, we can find~$m$ Lebesgue points
$x_1,\ldots,x_m$ of $u^p(\eta_1 - \eta_2)$ with respect
to $\hat\nu$ such that
\[
u^p(x_j)(\eta_1(x_j) - \eta_2(x_j)) > 0
\qquad\text{for all $j=1,\ldots,m$}\,.
\]
Let $r>0$ be such that $B_r(x_i)\cap B_r(x_j) = \emptyset$
whenever $i\neq j$ and such that
\[
\int_{B_r(x_j)} u^p(\eta_1 - \eta_2)d\hat\nu >0
\qquad\text{for all $j=1,\ldots,m$}\,.
\]
For every $j=1,\ldots,m$, let 
$\vartheta_j\in C^\infty_c(B_r(x_j))$ be such that
\[
\int (\vartheta_j u)^p(\eta_1 - \eta_2)d\hat\nu >0
\qquad\text{for all $j=1,\ldots,m$}\,.
\]
If we set
\begin{gather*}
\pi(\xi) = \sum_{j=1}^m \xi_j \vartheta_j u\,,\\
K = \left\{
\frac{\pi(\xi)}{\displaystyle{
\left(\int |\pi(\xi)|^p\,d\mup - \int |\pi(\xi)|^p\,d\mum
\right)^{1/p}}}:\,\,\xi\in S^{m-1}\right\}\,,
\end{gather*}
we have $K\in\mathcal{K}_m$ and 
$\sup\limits_{u\in K}R(u)<+\infty$, whence
$\lambda_m^p(\mu,\mup,\mum)<+\infty$.
\end{proof}
\begin{example} 
Let $p>N$, $\mu=\infty_{\R^N\setminus B_1(0)}$, $\mup=\delta_0$
and $\mum=0$.
Since
\[
\left\{u\in W^{1,p}_0(B_1):\,\,u(0)\neq 0 \right\} 
= \left\{u\in W^{1,p}_0(B_1):\,\,u(0) < 0 \right\} \cup
\left\{u\in W^{1,p}_0(B_1):\,\,u(0) > 0 \right\} \,,
\]
we have $i(K)=1$ for all nonempty, compact and symmetric subset 
$K$ of $L^p_{loc}(\R^N)\setminus\{0\}$ with
\[
\sup_{u\in K}\,R(u) < +\infty\,.
\]
Therefore, it follows that
\[
\lambda_m^p(\mu,\mup,\mum) = +\infty
\quad\text{for all $m\geq 2$}\,.
\]
By the way, a direct computation shows that
\[
\lambda_1^p(\mu,\mup,\mum) = \frac{1}{w(0)^{p-1}}\,,
\]
where $w\in W^{1,p}_0(B_1)$ satisfies $-\Delta_p w = \delta_0$.
\end{example}
\begin{proposition}
\label{minimaxfinLp}
Assume one of the following conditions:
\begin{enumerate}[label={\upshape\alph*)}, align=parleft, 
widest=iii, leftmargin=*]
\item[\mylabel{haKfin}{\ha}]
if $(u_n)$ is a sequence in $W^{1,p}_{c}(\R^N)$ 
satisfying
\[
\sup_n \left(
\int |\nabla u_n|^p\,d\leb^N + \int |u_n|^p\,d\mu
+ \int |u_n|^p\,d\mup
+ \int |u_n|^p\,d\mum\right) < +\infty
\]
and converging in $L^p_{loc}(\R^N)$ to some 
$u\in W^{1,p}_{loc}(\R^N)$, then
\[
\lim_{n\to\infty} \int |u_n|^p\,d\mup =
\int |u|^p\,d\mup\,;
\]
\item[\mylabel{hbKfin}{\hb}]
we have $\mum=0$.
\end{enumerate}
\par\indent
Then, for every integer $m\geq 1$, we have 
\[
\lambda_m^p(\mu,\mup,\mum) =
\inf_{K\in\mathcal{K}_m^{fin}}\,\sup_{u\in K}\,R(u) \,.
\]
\end{proposition}
\begin{proof}
We aim to apply Theorem~\ref{minimaxfin} and 
Remark~\ref{rem:minimaxfin}.
Actually, assumption~\ref{hbsuff} of Remark~\ref{rem:minimaxfin} 
follows from assumption~\ref{haKfin} and 
Proposition~\ref{strongappr}, while assumption~\ref{hcsuff} of 
Remark~\ref{rem:minimaxfin} follows from Proposition~\ref{mu-lsc} 
and assumption~\ref{hbKfin}.
\end{proof}
\begin{proposition}
If we set $\mu^{(n)}=\mu + \infty_{\R^N\setminus B_n(0)}$ and 
define $R^{(n)}$ accordingly, then we have
\[
\inf_{K\in\mathcal{K}_m^{fin}} 
\sup_{u\in K}\, R(u) =
\lim_{n\to\infty}\,\left(\inf_{K\in\mathcal{K}_m^{fin}} 
\sup_{u\in K}\, R^{(n)}(u)\right) 
\qquad\text{for all $m\geq 1$}\,.
\]
If either assumption~\ref{haKfin} or assumption~\ref{hbKfin} of
Proposition~\ref{minimaxfinLp} is satisfied, then we also have
\[
\lambda_m^p(\mu,\mup,\mum) =
\lim_{n\to\infty}\,\lambda_m^p(\mu^{(n)},\mup,\mum)
\qquad\text{for all $m\geq 1$}\,.
\]
\end{proposition}
\begin{proof}
Of course, we have
\[
\inf_{K\in\mathcal{K}_m^{fin}} 
\sup_{u\in K}\, R(u) \leq
\liminf_{n\to\infty}\,\left(\inf_{K\in\mathcal{K}_m^{fin}} 
\sup_{u\in K}\, R^{(n)}(u)\right) \,.
\]
To prove that
\[
\inf_{K\in\mathcal{K}_m^{fin}} 
\sup_{u\in K}\, R(u) \geq
\limsup_{n\to\infty}\,\left(\inf_{K\in\mathcal{K}_m^{fin}} 
\sup_{u\in K}\, R^{(n)}(u)\right) \,,
\]
we aim to apply Theorem~\ref{thm4.1dema}.
Assumption~\ref{ha41} is clearly satisfied, while
assumption~\ref{hb41} follows from Proposition~\ref{strongappr}
and assumption~\ref{hc41} follows from Proposition~\ref{mu-lsc}.
Therefore, the first claim is proved.
\par
Then we also have
\begin{gather*}
\lambda_m^p(\mu,\mup,\mum) \leq
\liminf_{n\to\infty}\,\lambda_m^p(\mu^{(n)},\mup,\mum)\,,\\
\limsup_{n\to\infty}\,\lambda_m^p(\mu^{(n)},\mup,\mum) \leq
\inf_{K\in\mathcal{K}_m^{fin}} \sup_{u\in K}\, R(u) \,.
\end{gather*}
By Proposition~\ref{minimaxfinLp} the second assertion follows.
\end{proof}
%

%--------------------------------------------------------------------

\section{Semicontinuity properties of inf-sup values of measures}
\label{sect:semicontinfsup}
Throughout this section, we consider three sequences
$(\mu^{(n)})$, $(\mup^{(n)})$, $(\mum^{(n)})$
in $\cmeas(\R^N)$, three measures 
$\mu, \mup, \mum \in \cmeas(\R^N)$,
an index $i$ as in Section~\ref{sect:convminmax} 
and the related inf-sup values 
$\lambda_m^p(\mu^{(n)},\mup^{(n)},\mum^{(n)})$ and
$\lambda_m^p(\mu,\mup,\mum)$ defined in Section~\ref{sect:preeig}
with respect to the metrizable and locally convex topological 
vector space $\mathcal{X}=L^p_{loc}(\R^N)$.
\par
We also consider the functionals $f_{\mu^{(n)},0}$,
$f_{\mu,0}$ defined in Section~\ref{sect:convergence}
and we define
\[
g_1^{(n)},g_2^{(n)}, R^{(n)}:L^p_{loc}(\R^N)\rightarrow[0,+\infty]
\]
by
\begin{alignat*}{3}
&g_j^{(n)}(u) &&=
\begin{cases}
\displaystyle{\frac{1}{p}\,\int |u|^p\,d\nu_j^{(n)}\,,}
&\qquad\text{if $u\in W^{1,p}_{loc}(\R^N)$}\,,\\
\noalign{\medskip}
+\infty
&\qquad\text{otherwise}\,,
\end{cases} \\
&R^{(n)}(u) &&=
\begin{cases}
f_{\mu^{(n)},0}(u)
&\qquad\text{if $1+g_2^{(n)}(u)\leq g_1^{(n)}(u)<+\infty$}\,,\\
\noalign{\medskip}
+\infty
&\qquad\text{otherwise}\,,
\end{cases}
\end{alignat*}
and
\[
g_1,g_2, R:L^p_{loc}(\R^N)\rightarrow[0,+\infty]
\]
in the analogous way with $\mu, \mup, \mum$ instead of
$\mu^{(n)}, \mup^{(n)}, \mum^{(n)}$.

\subsection{Lower semicontinuity of inf-sup values of measures}
Throughout this subsection we assume that:
\emph{
\begin{enumerate}[align=parleft]
\item[\mylabel{his}{\his}]
if $(n_k)$ is a strictly increasing sequence in $\N$ and
$(u^{(k)})$ is a sequence in $W^{1,p}_{c}(\R^N)$ satisfying
\[
\sup_{k} \biggl(
\int |\nabla u^{(k)}|^p\,d\leb^N
+ \int |u^{(k)}|^p\,d\mu^{(n_k)} 
+ \int |u^{(k)}|^p\,d\mup^{(n_k)}
+ \int |u^{(k)}|^p\,d\mum^{(n_k)}\biggr) < +\infty
\]
and converging in $L^p_{loc}(\R^N)$ to some 
$u\in W^{1,p}_{loc}(\R^N)$, then
\[
\limsup_{k\to\infty} \int |u^{(k)}|^p\,d\mup^{(n_k)} 
\leq \int |u|^p\,d\mup \,;
\]
\item[\mylabel{hiis}{\hiis}]
if $p\geq N$, we do not have 
$\mu(\R^N)=0$ and $\mum(\R^N) \leq \mup(\R^N) < +\infty$;
\item[\mylabel{hiiis}{\hiiis}]
we have
\begin{multline*}
~\qquad\qquad
f_{\mu,0}+\lambda g_2+I_{\{g_1<+\infty\}}
\leq\left(\Gamma-\liminf_{n\to\infty} 
\left(f_{\mu^{(n)},0}+\lambda g_2^{(n)}+I_{\{g_1^{(n)}<+\infty\}}
\right)\right)(u) \\
\qquad \text{for all $\lambda>0$ and $u\in L^p_{loc}(\R^N)$}\,.
\end{multline*}
\end{enumerate}
}
\begin{lemma}
\label{oldis}
If $(n_k)$ is a strictly increasing sequence in $\N$ and
$(u^{(k)})$ is a sequence in $W^{1,p}_{loc}(\R^N)$ satisfying
\[
\sup_k f_{\mu^{(n_k)},0}(u^{(k)}) <+\infty\,,\,\,
\sup_k\int |u^{(k)}|^p\,d\mum^{(n_k)} < +\infty\,,\,\,
\sup_k\int |u^{(k)}|^p\,d\mup^{(n_k)} < +\infty
\]
and converging in $L^p_{loc}(\R^N)$ to some 
$u\in W^{1,p}_{loc}(\R^N)$, then 
\begin{alignat*}{3}
&\liminf_{k\to\infty}\,f_{\mu^{(n_{k})},0}(u^{(k)}) 
&&\geq	f_{\mu,0}(u)\,,\\
&\liminf_{k\to\infty} \int |u^{(k)}|^p\,d\mum^{(n_{k})} 
&&\geq \int |u|^p\,d\mum\,,\\
&\limsup_{k\to\infty} \int |u^{(k)}|^p\,d\mup^{(n_{k})} 
&&\leq \int |u|^p\,d\mup < +\infty\,.
\end{alignat*}
\end{lemma}
\begin{proof}
Let $d$ be a compatible distance in $L^p_{loc}(\R^N)$.
By Proposition~\ref{strongappr}, for every $k\in\N$ there exists
$v^{(k)}\in W^{1,p}_c(\R^N)$ such that
\begin{multline*}
\biggl(d(v^{(k)},u^{(k)}) +
\int |\nabla v^{(k)}-\nabla u^{(k)}|^p\,d\leb^N 
+ \int |v^{(k)}-u^{(k)}|^p\,d\mu^{(n_k)} \biggr. \\
\biggl. 
+ \int |v^{(k)}-u^{(k)}|^p\,d\mup^{(n_k)}
+ \int |v^{(k)}-u^{(k)}|^p\,d\mum^{(n_k)}\biggr) < \frac{1}{k}\,.
\end{multline*}
From assumption~\ref{his} we infer that
\[
\limsup_{k\to\infty} \int |u^{(k)}|^p\,d\mup^{(n_{k})} 
= \limsup_{k\to\infty} \int |v^{(k)}|^p\,d\mup^{(n_{k})} 
\leq \int |u|^p\,d\mup\,.
\]
On the other hand, by assumption~\ref{hiiis} we have 
\[
\int |u|^p\,d\mup < +\infty
\]
and, for every $\lambda>0$,
\begin{alignat*}{5}
&f_{\mu,0}(u) 
&&\leq f_{\mu,0}(u)+\lambda \int |u|^p\,d\mum \\
&&&\leq \liminf_{k\to\infty} 
\left(f_{\mu^{(n_k)},0}(u^{(k)})
+ \lambda \int |u^{(k)}|^p\,d\mum^{(n_k)}\right) \\
&&&
\leq \liminf_{k\to\infty}\,f_{\mu^{(n_k)},0}(u^{(k)})
+\lambda \limsup_{k\to\infty}\,
\int |u^{(k)}|^p\,d\mum^{(n_k)}\,,\\
&\int |u|^p\,d\mum
&&\leq \frac{1}{\lambda}\,\left(f_{\mu,0}(u)
+ \lambda \int |u|^p\,d\mum\right) \\
&&&\leq \frac{1}{\lambda}\,\liminf_{k\to\infty} 
\left(f_{\mu^{(n_k)},0}(u^{(k)})
+ \lambda \int |u^{(k)}|^p\,d\mum^{(n_k)}\right) \\
&&&
\leq
\frac{1}{\lambda}\,\limsup_{k\to\infty}\,
f_{\mu^{(n_k)},0}(u^{(k)})
+ \liminf_{k\to\infty}\,\int |u^{(k)}|^p\,d\mum^{(n_k)}\,.
\end{alignat*}
By the arbitrariness of $\lambda$ the assertion follows.
\end{proof}
\begin{proposition}
\label{propcoercs}
If $(n_k)$ is a strictly increasing sequence in $\N$ and
$(u^{(k)})$ is a sequence in $W^{1,p}_{loc}(\R^N)$ satisfying
\[
\sup_k f_{\mu^{(n_k)},0}(u^{(k)}) <+\infty\,,\quad
\sup_k \int |u^{(k)}|^p\,d\mum^{(n_k)} < +\infty\,,
\quad \int |u^{(k)}|^p\,d\mup^{(n_k)} < +\infty
\quad\text{for all $k\in\N$}\,,
\]
then there exists a subsequence $(u^{(k_j)})$ converging 
in $L^p_{loc}(\R^N)$ to some $u\in W^{1,p}_{loc}(\R^N)$ with
\begin{alignat*}{3}
&\liminf_{j\to\infty}\,f_{\mu^{(n_{k_j})},0}(u^{(k_j)}) 
&&\geq	f_{\mu,0}(u)\,,\\
&\liminf_{j\to\infty} \int |u^{(k_j)}|^p\,d\mum^{(n_{k_j})} 
&&\geq \int |u|^p\,d\mum\,,\\
&\limsup_{j\to\infty} \int |u^{(k_j)}|^p\,d\mup^{(n_{k_j})} 
&&\leq \int |u|^p\,d\mup < +\infty\,.
\end{alignat*}
\end{proposition}
\begin{proof}
Consider first the particular case in which
\[
\sup_k f_{\mu^{(n_k)},0}(u^{(k)}) <+\infty\,,\,\,
\sup_k \int |u^{(k)}|^p\,d\mum^{(n_k)} < +\infty\,,\,\,
\sup_k \int |u^{(k)}|^p\,d\mup^{(n_k)} < +\infty\,.
\]
By Lemma~\ref{oldis} it is enough to prove that
\[
\liminf_{k\to\infty}\,
\int_{B_1(0)} |u^{(k)}|^p\,d\leb^N < +\infty\,.
\]
If $p<N$, this fact follows from the boundedness of 
$(\nabla u_k)$ in $L^p(\R^N;\R^N)$.
If $p\geq N$, assume for the sake of contradiction that
\[
\lim_{k\to\infty}\,\int_{B_1(0)} |u^{(k)}|^p\,d\leb^N = +\infty\,.
\]
Then a suitably rescaled sequence $(v^{(k)})$ satisfies
\begin{multline*}
\lim_{k\to\infty} f_{\mu^{(n_k)},0}(v^{(k)}) = 
\lim_{k\to\infty} \int |v^{(k)}|^p\,d\mum^{(n_k)} = 
\lim_{k\to\infty} \int |v^{(k)}|^p\,d\mup^{(n_k)} = 0\,,\\
\int_{B_1(0)} |v^{(k)}|^p\,d\leb^N = 1
\quad\text{for all $k\in\N$}\,.
\end{multline*}
Up to a subsequence, $(v^{(k)})$ is convergent in 
$W^{1,p}_{loc}(\R^N)$ to some $v$ satisfying,
by Lemma~\ref{oldis},
\[
\lim_{n\to\infty} f_{\mu,0}(v) = \int |v|^p\,d\mum = 0\,,\qquad
\int |v|^p\,d\mup < +\infty\,,\qquad
\int_{B_1(0)} |v|^p\,d\leb^N = 1\,.
\]
Therefore $v$ is a nonzero constant and
$\mu(\R^N)=\mum(\R^N)=0$, while $\mup(\R^N)<+\infty$.
This fact contradicts assumption~\ref{hiis}.
\par
Now let us treat the general case and suppose, for the sake 
of contradiction, that up to a subsequence 
\[
\lim_{k\to\infty} \int |u^{(k)}|^p\,d\mup^{(n_k)} = +\infty \,.
\]
Then a suitably rescaled sequence $(v^{(k)})$ satisfies
\[
\lim_{k\to\infty} f_{\mu^{(n_k)},0}(v^{(k)}) = 0\,,\,\,
\lim_{k\to\infty} \int |v^{(k)}|^p\,d\mum^{(n_k)} = 0\,,\,\,
\int |v^{(k)}|^p\,d\mup^{(n_k)} = 1
\quad\text{for all $k\in\N$}\,.
\]
By the previous step, up to a subsequence $(v^{(k)})$ is 
convergent in $L^p_{loc}(\R^N)$ to some $v\in W^{1,p}_{loc}(\R^N)$
such that
\[
f_{\mu,0}(v) = 0\,,\qquad
\int |v|^p\,d\mum = 0\,,\qquad
1\leq \int |v|^p\,d\mup <+\infty\,.
\]
It follows that $v$ is a nonzero constant and that 
$\mu(\R^N)=\mum(\R^N)=0$, while $\mup(\R^N)<+\infty$.
If $p<N$, this is a contradiction, as $v\in L^{p^*}(\R^N)$.
If $p\geq N$, a contradiction follows from assumption~\ref{hiis}.
\end{proof}
\begin{theorem}\label{thmlsc}
For every integer $m\geq 1$, we have 
\[
\lambda_m^p(\mu,\mup,\mum) \leq \liminf_{n\to\infty} \,
\lambda_m^p(\mu^{(n)},\mup^{(n)},\mum^{(n)})\,.
\]
\end{theorem}
\begin{proof}
We aim to apply Corollary~\ref{cor4.3dema} with 
$\mathcal X=L^p_{loc}(\R^N)$.
Assumption~\ref{ha43} of Corollary~\ref{cor4.3dema} 
is obviously satisfied, while
assumption~\ref{hb43} of Corollary~\ref{cor4.3dema}
is just assumption~\ref{hiiis} and
assumption~\ref{hc43} of Corollary~\ref{cor4.3dema}
follows from Proposition~\ref{propcoercs}. 
\par
Finally, if $u\in W^{1,p}_{loc}(\R^N)\setminus\{0\}$ and
$f_{\mu,0}(u)=0$, we infer that $u$ is constant, $p\geq N$, 
$\mu(\R^N)=0$ and we cannot have 
\[
\int |u|^p\,d\mum \leq \int |u|^p\,d\mup < +\infty
\]
by assumption~\ref{hiis}.
Therefore assumption~\ref{hd43} of Corollary~\ref{cor4.3dema}
is satisfied and the assertion follows.
\end{proof}

\subsection{Upper semicontinuity of inf-sup values of measures}
Throughout this subsection we assume that:
\emph{
\begin{enumerate}[align=parleft]
\item[\mylabel{hivs}{\hivs}]
if $(n_k)$ is a strictly increasing sequence in $\N$ and 
$(u^{(k)})$ a sequence converging to $u$ in 
$L^p_{loc}(\R^N)\setminus\{0\}$ with
\begin{multline*}
~\qquad\qquad\qquad
\sup_k f_{\mu^{(n_k)},0}(u^{(k)})<+\infty\,,\qquad
\sup_k \int |u^{(k)}|^p\,d\mup^{(n_k)}<+\infty\,,\\
\int |u^{(k)}|^p\,d\mum^{(n_k)}<\int |u^{(k)}|^p\,d\mup^{(n_k)}
\qquad\text{for all $k\in\N$}\,,
\end{multline*}
then
\[
\int |u|^p\,d\mup \leq 
\liminf_{k\to\infty}\,\int |u^{(k)}|\,d\mup^{(n_k)}\,;
\]
\item[\mylabel{hvs}{\hvs}]
we have
\begin{multline*}
~\qquad\qquad
f_{\mu,0}+\lambda g_2+I_{\{g_1<+\infty\}}
\geq\left(\Gamma-\limsup_{n\to\infty} 
\left(f_{\mu^{(n)},0}+\lambda g_2^{(n)}+I_{\{g_1^{(n)}<+\infty\}}
\right)\right)(u) \\
\qquad \text{for all $\lambda>0$ and $u\in L^p_{loc}(\R^N)$}\,.
\end{multline*}
\end{enumerate}
}
\begin{theorem}\label{thmusc}
Assume one of the following conditions:
\begin{enumerate}[label={\upshape\alph*)}, align=parleft, 
widest=iii, leftmargin=*]
\item[\mylabel{hausc}{\ha}]
if $(u_n)$ is a sequence in $W^{1,p}_{c}(\R^N)$ 
satisfying
\[
\sup_n \left(
\int |\nabla u_n|^p\,d\leb^N + \int |u_n|^p\,d\mu
+ \int |u_n|^p\,d\mup
+ \int |u_n|^p\,d\mum\right) < +\infty
\]
and converging in $L^p_{loc}(\R^N)$ to some 
$u\in W^{1,p}_{loc}(\R^N)$, then
\[
\lim_{n\to\infty} \int |u_n|^p\,d\mup =
\int |u|^p\,d\mup\,;
\]
\item[\mylabel{hbusc}{\hb}]
we have $\mum=0$.
\end{enumerate}
\par\indent
Then, for every integer $m\geq 1$, we have 
\[
\lambda_m^p(\mu,\mup,\mum) \geq \limsup_{n\to\infty} \,
\lambda_m^p(\mu^{(n)},\mup^{(n)},\mum^{(n)})\,.
\]
\end{theorem}
\begin{proof}
First of all we claim that, by Theorem~\ref{thm4.1dema},
we have
\[
\inf_{K\in\mathcal{K}_m^{fin}} \sup_{u\in K} \, R(u) 
\geq \limsup_{n\to\infty} \left(
\inf_{K\in\mathcal{K}_m^{fin}} 
\sup_{u\in K} \, R^{(n)}(u)\right)\,.
\]
Actually, assumption~\ref{ha41} of Theorem~\ref{thm4.1dema} 
is obviously satisfied, while assumption~\ref{hb41} is 
assumption~\ref{hvs} and assumption~\ref{hc41} is implied by
assumption~\ref{hivs}.
\par
\emph{A fortiori} we have
\[
\inf_{K\in\mathcal{K}_m^{fin}} \sup_{u\in K} \, R(u) 
\geq \limsup_{n\to\infty} \,
\lambda_m^p(\mu^{(n)},\mup^{(n)},\mum^{(n)})
\]
and the assertion follows from Proposition~\ref{minimaxfinLp}.
\end{proof}
%

%--------------------------------------------------------------------

\section{Existence of nonlinear eigenvectors for sign-changing 
capacitary measures}
\label{sect:eigmeas}
Let $\mu, \mup, \mum\in\cmeas(\R^N)$.
In this section we want to show that, under suitable
assumptions, the inf-sup values $\lambda_m^p(\mu,\mup,\mum)$
introduced in Section~\ref{sect:preeig} are true eigenvalues
of the problem
\[
\begin{cases}
-\Delta_pu+|u|^{p-2}u\,\mu = \la |u|^{p-2}u(\nu_1-\nu_2)
\qquad\text{in $\R^N$}\,,\\
\noalign{\medskip}
\displaystyle{
\int |u|^p\,d\nu_2 < \int |u|^p\,d\nu_1} \,,
\end{cases}
\]
with corresponding eigenvectors.
To this aim, we will relate the inf-sup values
$\lambda_m^p(\mu,\mup,\mum)$ with the inf-sup values 
$\hat{\lambda}_m^p(\mu,\mup,\mum)$ defined in a functional 
setting where standard variational methods apply.
\par
Throughout this section we assume that:
\emph{
\begin{enumerate}[align=parleft]
\item[\mylabel{hi}{\hi}]
if $(u_n)$ is a sequence in $W^{1,p}_{c}(\R^N)$ 
satisfying
\[
\sup_n \left(
\int |\nabla u_n|^p\,d\leb^N + \int |u_n|^p\,d\mu
+ \int |u_n|^p\,d\mup
+ \int |u_n|^p\,d\mum\right) < +\infty
\]
and converging in $L^p_{loc}(\R^N)$ to some 
$u\in W^{1,p}_{loc}(\R^N)$, then
\[
\lim_{n\to\infty} \int |u_n|^p\,d\mup =
\int |u|^p\,d\mup\,;
\]
\item[\mylabel{hii}{\hii}]
if $p\geq N$, we do not have 
$\mu(\R^N)=0$ and $\mum(\R^N) \leq \mup(\R^N) < +\infty$.
\end{enumerate}
}
Taking into account Proposition~\ref{mu-lsc}, these assumptions 
turn out to be hypotheses~\ref{his} and~\ref{hiis}
of Section~\ref{sect:semicontinfsup}, in the case in which
$\mu^{(n)}=\mu$, $\mup^{(n)}=\mup$ and $\mum^{(n)}=\mum$.
\begin{proposition}
\label{propcoerc}
If $(u_n)$ is a sequence in $W^{1,p}_{loc}(\R^N)$ satisfying
\[
\sup_n f_{\mu,0}(u_n) <+\infty\,,\,\,
\sup_n \int |u_n|^p\,d\mum < +\infty\,,\,\,
\int |u_n|^p\,d\mup < +\infty
\quad\text{for all $n\in\N$}\,,
\]
then there exists a subsequence $(u^{(n_j)})$ converging 
in $L^p_{loc}(\R^N)$ to some $u\in W^{1,p}_{loc}(\R^N)$ with
\begin{alignat*}{3}
&\liminf_{j\to\infty} f_{\mu,0}(u_{n_j}) 
&&\geq f_{\mu,0}(u)\,,\\
&\liminf_{j\to\infty} \int |u^{(n_j)}|^p\,d\mum 
&&\geq \int |u|^p\,d\mum\,,\\
&\lim_{j\to\infty} \int |u^{(n_j)}|^p\,d\mup 
&&= \int |u|^p\,d\mup < +\infty \,.
\end{alignat*}
\end{proposition}
\begin{proof}
Taking into account Proposition~\ref{mu-lsc},
it is a particular case of Proposition~\ref{propcoercs}.
\end{proof}
Now we set
\[
X = \left\{u\in W^{1,p}_{loc}(\R^N)\cap L^p(\R^N,\mup)
\cap L^p(\R^N,\mum):\,\,
f_{\mu,0}(u) < +\infty\right\}\,.
\]
\begin{proposition}
We have that $X$ is a vector subspace of $W^{1,p}_{loc}(\R^N)$
and
\[
\|u\| := \left(\int |\nabla u|^p\,d\leb^N +
\int |u|^p\,d\mu + \int |u|^p\,d\mup
+ \int |u|^p\,d\mum\right)^{1/p}
\]
is a norm on $X$ which makes $X$ a uniformly convex
Banach space.
\par
Moreover, $X\cap W^{1,p}_c(\R^N)$ is dense in $X$ and 
the linear maps
\[
\begin{array}{cccccc}
X & \longrightarrow & L^p_{loc}(\R^N)
&\qquad\qquad X & \longrightarrow & L^p(\R^N,\mup) \\
\noalign{\medskip}
u & \mapsto & u
&\qquad\qquad u & \mapsto & u
\end{array}
\]
are completely continuous.
\end{proposition}
\begin{proof}
It is easily seen that $X$ is a vector subspace of 
$W^{1,p}_{loc}(\R^N)$ and that $\|u\|$ is a norm in $X$.
In particular, assumption~\ref{hii} guarantees that $\|u\|=0$
only if $u=0$.
\par
Of course
\[
\begin{array}{ccc}
X & \longrightarrow &
L^p(\R^N;\R^N) \times L^p(\R^N,\mu)
\times L^p(\R^N,\mup) \times L^p(\R^N,\mum) \\
\noalign{\medskip}
u & \mapsto &
(\nabla u,u,u,u)
\end{array}
\]
is a linear isometry.
We claim that its image is closed.
Actually, if $(u_n)$ is a sequence in $X$ such that
$((\nabla u_n,u_n,u_n,u_n))$ is convergent to
$(U,v_1,v_2,v_3)$, from 
Propositions~\ref{propcoerc} and~\ref{mu-lsc}
we infer that, up to a subsequence, $(u_n)$ is convergent
in $L^p_{loc}(\R^N)$ to some $u\in X$ with 
$(\nabla u,u,u,u) = (U,v_1,v_2,v_3)$ and the claim follows.
\par
Therefore, $X$ is a uniformly convex Banach space.
By Proposition~\ref{strongappr} we have that 
$X\cap W^{1,p}_c(\R^N)$ is dense in~$X$, while the linear maps
\[
\begin{array}{cccccc}
X & \longrightarrow & L^p_{loc}(\R^N)
&\qquad\qquad X & \longrightarrow & L^p(\R^N,\mup) \\
\noalign{\medskip}
u & \mapsto & u
&\qquad\qquad u & \mapsto & u
\end{array}
\]
are completely continuous by Propositions~\ref{propcoerc}
and~\ref{mu-lsc}.
\end{proof}
\begin{remark}
\label{rem:X}
We will see that in $X$ standard variational methods apply.
On the other hand $X$ depends on $\mu$, $\mup$ and $\mum$,
while $L^p_{loc}(\R^N)$ is a fixed space, hence more suitable
for $\Gamma$-convergence.
\par
If $\mu=\infty_{\R^N\setminus A}$, where $A$ is $p$-quasi open,
$\mup = \leb^N$ and $\mum=0$, then $X=W^{1,p}_0(A)$ endowed
with the usual structure of uniformly convex Banach space.
\end{remark}
We also define $g_1,g_2:L^p_{loc}(\R^N)\rightarrow[0,+\infty]$ by
\[
g_j(u) =
\begin{cases}
\displaystyle{\frac{1}{p}\,\int |u|^p\,d\nu_j}
&\qquad\text{if $u\in W^{1,p}_{loc}(\R^N)$}\,,\\
\noalign{\medskip}
+\infty
&\qquad\text{otherwise}\,,
\end{cases}
\]
and set
\[
\varphi = f_{\mu,0}\bigl|_X\,,\,\,
\psi_1 = {g_1}\bigl|_X\,,\,\,
\psi_2 = {g_2}\bigl|_X\,,\,\,
\widehat M = 
\left\{u\in X:\,\,\psi_1(u) - \psi_2(u) = 1\right\}\,.
\]
Of course, $\varphi$, $\psi_1$ and $\psi_2$ are even, convex, 
positively homogeneous of degree $p$ and of class $C^1$.
According to Section~\ref{sect:nep}, we denote by 
$\widehat{\mathcal{K}}_m$ the family
of nonempty, compact and symmetric subsets $K$ of
$\widehat{M}$ (with respect to the topology of $X$)
such that $\idx{K}\geq m$ and we set
\[
\hat{\lambda}_m^p(\mu,\mup,\mum) =
\inf_{K\in\widehat{\mathcal{K}}_m}\,\max_{u\in K}\,\varphi(u) \,,
\]
where we agree that 
$\hat{\lambda}_m^p(\mu,\mup,\mum) = +\infty$ if there is no 
$K$ included in $\widehat{M}$ with $\idx{K}\geq m$.
\begin{theorem}
\label{thm:lambdaequal}
For every integer $m\geq 1$, we have
\[
\hat{\lambda}_m^p(\mu,\mup,\mum) =
\lambda_m^p(\mu,\mup,\mum)  \,.
\]
\end{theorem}
\begin{proof}
Let
\[
\widetilde{M} = 
\left\{u\in X:\,\,1+\psi_2(u) \leq \psi_1(u)\right\} 
\]
and denote by $\widetilde{\mathcal{K}}_m$ the family
of nonempty, compact and symmetric subsets $K$ of
$\widetilde{M}$ such that $\idx{K}\geq m$.
\par
Of course, the topologies of $X$ and of 
$L^p_{loc}(\R^N)$ agree on finite dimensional subspaces.
Moreover, assumption~\ref{hbsuff} of Remark~\ref{rem:minimaxfin}
is satisfied by $\varphi$, $\psi_1$ and $\psi_2$ in the space $X$,
while assumption~\ref{haKfin} of Proposition~\ref{minimaxfinLp}
is just assumption~\ref{hi}.
Combining Theorem~\ref{minimaxfin} with
Proposition~\ref{minimaxfinLp}, we infer that
\[
\inf_{K\in\widetilde{\mathcal{K}}_m}\,\max_{u\in K}\,\,\varphi(u) 
= \lambda_m^p(\mu,\mup,\mum)  \,.
\]
Of course, we have
\[
\inf_{K\in\widetilde{\mathcal{K}}_m}\,\max_{u\in K}\,\,\varphi(u)
\leq \inf_{K\in\widehat{\mathcal{K}}_m}\,
\max_{u\in K}\,\varphi(u)\,,
\]
as $\widehat{M}\subseteq \widetilde{M}$.
On the other hand, if $K\in \widetilde{\mathcal{K}}_m$,
we have that
\[
\widehat{K} = 
\left\{\frac{u}{(\psi_1(u)-\psi_2(u))^{1/p}}:\,\,u\in K\right\}
\]
satisfies $\widehat{K}\in \widehat{\mathcal{K}}_m$ and
$\max\limits_{u\in\widehat{K}} \varphi(u) 
\leq \max\limits_{u\in K} \varphi(u)$, 
whence 
\[
\inf_{K\in\widehat{\mathcal{K}}_m}\,
\max_{u\in K}\,\varphi(u) \leq
\inf_{K\in\widetilde{\mathcal{K}}_m}\,\max_{u\in K}\,\varphi(u) 
\]
and the assertion follows.
\end{proof}
\begin{corollary}
\label{cor:lambdafin}
If there exists $u\in W^{1,p}_c(\R^N)$ such that
\[
\int |u|^p\,d\mu<+\infty\,,\qquad
\int |u|^p\,d\mum < \int |u|^p\,d\mup < +\infty\,,\qquad
\lim_{r\to 0}\,\int_{B_r(x)} |u|^p\,d\mup=0
\quad\text{for all $x\in\R^N$}\,,
\]
then we have
$\hat{\lambda}_m^p(\mu,\mup,\mum)<+\infty$ for all $m\geq 1$.
\end{corollary}
\begin{proof}
It follows from Proposition~\ref{prop:lambdafin}
and Theorem~\ref{thm:lambdaequal}.
\end{proof}
\begin{theorem}
\label{thm:phipsi}
The functionals $\varphi$, $\psi_1$ and $\psi_2$ satisfy the 
assumptions~\ref{hae} and~\ref{hbe} of Section~\ref{sect:nep}.
In particular, the assertions of Theorem~\ref{generaleig} hold true.
\end{theorem}
\begin{proof}
Since $\psi_1'$ is the composition
\[
\begin{array}{ccccccc}
X & \longrightarrow & L^p(\R^N,\mup) 
& \longrightarrow & L^{p'}(\R^N,\mup)
& \longrightarrow & X' \\
\noalign{\medskip}
u & \mapsto & u & \mapsto 
& |u|^{p-2}u & \mapsto &\psi_1'(u)
\end{array}
\]
the complete continuity of $\psi_1'$ follows from the
complete continuity of the first map and the continuity
of the other maps.
\par
Given $\lambda>0$, it is standard (see e.g.\cite{browder1983})
that $(\varphi' + \lambda\psi_2' + \lambda\psi_1')$ is of 
class $(S)_+$.
Then also
\[
\varphi' + \lambda\psi_2' =
(\varphi' + \lambda\psi_2' + \lambda\psi_1') - \lambda\psi_1'
\]
is of class $(S)_+$.
\par
Finally, if $u\in X\setminus\{0\}$ satisfies $\varphi(u)=0$,
then $u$ is a nonzero constant, $p\geq N$, $\mu(\R^N)=0$ and
we cannot have 
\[
\int |u|^p\,d\mum \leq \int |u|^p\,d\mup < +\infty
\]
by assumption~\ref{hii}.
\end{proof}
\begin{example}
Let $N=1$, $p=2$, $\mu=0$ and
\[
\nu_1(B) = \leb^1(B\cap ]0,1[)\,,\qquad
\nu_2(B) = \leb^1(B\cap ]-1,0[)\,,\qquad
\text{for all $B\in\bor(\R)$}\,.
\]
Then we have
\[
\inf \left\{\varphi(u):\,\,\psi_1(u)-\psi_2(u)=1\right\}=0\,,
\qquad
\text{$\varphi(u)>0$ for all $u$ with $\psi_1(u)-\psi_2(u)>0$}\,.
\]
On the other hand, assumption~\ref{hii} is not satisfied.
\end{example}
%

%--------------------------------------------------------------------

\section{On the existence of optimal capacitary measures}
\label{sect:potentials}
Let $\underline{\mu}, \nu\in\cmeas(\R^N)$ and 
$W:\R^N \rightarrow [0,+\infty[$ be a $\leb^N$-measurable function.
Let also $V_{\underline{\mu}+\nu}$ be the $\leb^N$-measurable
function introduced in Definition~\ref{def:Vmu} and set
$\mup = W \,\leb^N$, $\mum=\nu$ and
\[
\lambda_m^p(\mu) = \lambda_m^p(\mu,\mup,\mum)
\qquad\text{for all $\mu\in\cmeas(\R^N)$ and $m\geq 1$}\,.
\]
\par
If $N\geq 2$, we define a convex function $\tau:\R\rightarrow\R$ by
\[
\tau(s) = \sum_{k=N}^\infty \frac{|s|^{\frac{k}{N-1}}}{k!}
\]
and denote by $\tau^*:\R\rightarrow\R$ its conjugate function, 
namely
\[
\tau^*(t) = \sup_{s\in\R} \left(ts - \tau(s)\right)\,.
\]
Since
\[
\tau(2^{N-1}\,s) \geq 2^N\,\tau(s) 
\qquad\text{for all $s\in\R$}\,,
\]
we have
\begin{equation}
\label{eq:tau*}
\tau^*\left(2\,t\right) \leq 2^N\,\tau^*(t)
\qquad\text{for all $t\in\R$}\,.
\end{equation}
Throughout this section, we assume that:
\emph{
\begin{enumerate}[align=parleft]
\item[\mylabel{muW}{\muW}]
\begin{itemize}
\item
if $p< N$, we have
\[
\int_{\{V_{\underline{\mu}+\nu} < R W\}}
W^{N/p}\,d\leb^N < +\infty
\qquad\text{for all $R>0$}\,;
\]
\item
if $p=N$, we have
\[
\inf_{\varepsilon > 0}\,
\leb^N\left(\left\{V_{\underline{\mu}+\nu} 
< \varepsilon\right\}\right) < +\infty
\]
and
\[
\int_{\{V_{\underline{\mu}+\nu} < R W\}}
\tau^*\left(W\right)\,d\leb^N < +\infty
\qquad\text{for all $R>0$}\,;
\]
\item
if $p>N$, we have
\[
\inf_{\varepsilon > 0}\,
\leb^N\left(\left\{V_{\underline{\mu}+\nu} 
< \varepsilon\right\}\right) < +\infty
\]
and there exists $q\in[1,\infty[$ such that
\[
\int_{\{V_{\underline{\mu}+\nu} < R W\}}
W^q\,d\leb^N < +\infty
\qquad\text{for all $R>0$}\,.
\]
\end{itemize}
\end{enumerate}
}
\begin{proposition}
\label{prop:muW}
The following facts hold:
\begin{enumerate}[label={\upshape\alph*)}, align=parleft, 
widest=iii, leftmargin=*]
\item[\mylabel{hamuW}{\ha}]
if $(\mu^{(n)})$ is locally $\gamma$-convergent to $\mu$
in $\cmeas(\R^N)$ and $\mu^{(n)} \geq \underline{\mu}$
for all $n\in\N$, then $\mu \geq \underline{\mu}$ and
\[
\lambda_m^p(\mu) \leq \liminf_{n\to\infty} \,
\lambda_m^p(\mu^{(n)})
\qquad\text{for all $m\geq 1$}\,;
\]
\item[\mylabel{hbmuW}{\hb}]
for every $\mu\in\cmeas(\R^N)$ with $\mu\geq\underline{\mu}$, 
the assumptions~\ref{hi} and~\ref{hii} of 
Section~\ref{sect:eigmeas} are satisfied by
$(\mu,W\,\leb^N,\nu)$, in particular \[
\hat{\lambda}_m^p(\mu,W\,\leb^N,\nu) =
\lambda_m^p(\mu,W\,\leb^N,\nu),
\qquad\text{ for all }m\geq 1\,;
\]
\item[\mylabel{hcmuW}{\hc}]
for every $\mu\in\cmeas(\R^N)$, we have
\[
\lambda_m^p(\mu) < +\infty \qquad\text{for all $m\geq 1$}
\]
if and only if there exists $u\in W^{1,p}_c(\R^N)$ such that
\[
\int |u|^p\,d\mu < +\infty\,,\qquad
\int |u|^p\,d\nu < \int |u|^p\,W\,d\leb^N < +\infty\,.
\]
\end{enumerate}
\end{proposition}
\begin{proof}
\par\noindent
\par\noindent
\ref{hamuW}~By Corollary~\ref{cor:leq} we have 
$\mu \geq \underline{\mu}$.
The second assertion follows from~Theorem~\ref{thmlsc} as 
soon as the assumptions~\ref{his}-\ref{hiis}-\ref{hiiis} are 
verified.
We deal first with assumption~\ref{his}. 
Actually, we prove a stronger statement, which will be useful
also in the verification of~\ref{hiiis}.
\par
Let us consider a strictly increasing sequence $(n_k)$ in $\N$ and 
a sequence $(u^{(k)})$ in $W^{1,p}_c(\R^N)$ converging in 
$L^p_{loc}(\R^N)$ to $u\in W^{1,p}_{loc}(\R^N)$ with
\[
\sup_k\left(\int |\nabla u^{(k)}|^p\,d\leb^N
+ \int |u^{(k)}|^p\,d\mu^{(n_k)}
+ \int |u^{(k)}|^p\,d\nu \right) < +\infty \,.
\]
We claim that
\[
\limsup_{k\to\infty} \int |u^{(k)}|^p\,W\,d\leb^N 
\leq \int |u|^p\,W\,d\leb^N < +\infty\,.
\]
Up to a subsequence, $(u^{(k)})$ is convergent to $u$ 
$\leb^N$-a.e. in $\R^N$ and we have
\[
\sup_k\left(\int |\nabla u^{(k)}|^p\,d\leb^N
+ \int |u^{(k)}|^p\,d\underline{\mu}
+ \int |u^{(k)}|^p\,d\nu \right) < +\infty\,.
\]
Since for every $R>0$ it is
\[
\int_{\{R W \leq V_{\underline{\mu}+\nu} \}} 
|u^{(k)}|^p\,W\, d\leb^N \leq
\frac{1}{R}\,\left(\int |u^{(k)}|^p\,d\underline{\mu} 
+ \int |u^{(k)}|^p\,d\nu \right) \,,
\]
it is enough to show that
\begin{equation}
\label{eq:muW}
\limsup_{k\to\infty}\,
\int_{\{V_{\underline{\mu}+\nu} < R W\}}
|u^{(k)}|^p\,W\, d\leb^N
\leq
\int_{\{V_{\underline{\mu}+\nu} < R W\}}
|u|^p\,W\, d\leb^N < +\infty
\qquad\text{for all $R>0$}\,.
\end{equation}
In the case $p<N$, the sequence $(u^{(k)})$ is bounded in 
$L^{p^*}(\R^N)$, so that~\eqref{eq:muW} follows from 
assumption~\ref{muW}.
\par
If $p\geq N$, first of all by assumption~\ref{muW} there exists 
$\varepsilon>0$ such that
\[
\leb^N\left(\left\{V_{\underline{\mu}+\nu}  
< \varepsilon\right\}\right) < +\infty \,.
\]
If we set
$C = \left\{V_{\underline{\mu}+\nu} \geq \varepsilon\right\}$,
we have
\[
\sup_k\left(\int_{\R^N} |\nabla u^{(k)}|^p\,d\leb^N
+ \varepsilon \int_C |u^{(k)}|^p\,d\leb^N 
\right) < +\infty
\]
and $\leb^N(\R^N\setminus C)<+\infty$.
Since
\[
\sup_k\,\leb^N\left(\left\{|u^{(k)}| \geq 1\right\}\right)
< +\infty \,,
\]
it follows that
\[
\sup_k\left(\int_{\R^N} |\nabla u^{(k)}|^p\,d\leb^N
+ \int_{\R^N} |u^{(k)}|^p\,d\leb^N 
\right) < +\infty\,.
\]
Now, in the case $p=N$, according 
to~\cite[Theorem~1.1]{li_ruf2008} there exist $d_N,\alpha_N>0$ 
such that
\[
\int \tau\left(\alpha_N |v|^N\right)\,d\leb^N \leq d_N 
\qquad\text{whenever $v\in W^{1,N}_c(\R^N)$ and
$\int |\nabla v|^N\,d\leb^N + \int |v|^N\,d\leb^N \leq 1$}\,.
\]
Therefore, for every $\sigma>0$ there exists $j\geq 1$ such that
\[
2^{-j}\,\,\int \tau\left(2^{-j}\,|u^{(k)}|^N\right)\,d\leb^N 
< \sigma \qquad\text{for all $k\in\N$}\,.
\]
On the other hand, we have
\[
|u^{(k)}|^N\,W \leq 
2^{-j}\,\tau\left(2^{-j}\,|u^{(k)}|^N\right)
+ 2^{-j}\tau^*\left(2^{2j}\, W\right)
\qquad\text{a.e. in $\R^N$}
\]
and
\[
\int_{\{V_{\underline{\mu}+\nu} < R W\}}
\tau^*\left(2^{2j}\, W\right)\,d\leb^N < +\infty
\qquad\text{for all $R, j$}
\]
by assumption~\ref{muW} and~\eqref{eq:tau*}.
Therefore~\eqref{eq:muW} follows.
\par
In the case $p>N$, we have that $(u^{(k)})$ is bounded in each
$L^r(\R^N)$ with $p\leq r\leq\infty$ and ~\eqref{eq:muW} follows 
again from assumption~\ref{muW}.
Therefore assumption~\ref{his} is satisfied.
\par
Assumption~\ref{hiiis} follows from the previous step,
Theorem~\ref{GammalimthmW1pc} and Proposition~\ref{mu-lsc}.
\par
Finally, in the case $p \geq N$ also assumption~\ref{hiis} is 
satisfied, as $\leb^N(C) = +\infty$ implies that
$(\underline{\mu}+\nu)(\R^N) = +\infty$.
\par\noindent
\ref{hbmuW}~We argue as in the previous step, noting that
assumption~\ref{hi} is a special case of~\ref{his}.
\par\noindent
\ref{hcmuW}~The assertion follows from 
Proposition~\ref{prop:lambdafin}.
\end{proof}
\begin{theorem}
\label{thm:existmuV}
Let $\Psi:[0,+\infty]\rightarrow[0,+\infty]$ be a function as in 
Corollary~\ref{cor:lsc} and let
\[
0 < c \leq\int \Psi(V_{\underline{\mu}})\,d\leb^N \,.
\]
Denote by $\mathcal{M}$ the set of 
$\mu$'s in $\cmeas(\R^N)$ such that 
\[
\mu\geq\underline{\mu}\,,\qquad
\int \Psi(V_\mu)\,d\leb^N \leq c
\]
and such that there exists $u\in W^{1,p}_c(\R^N)$ satisfying
\[
\int |u|^p\,d\mu < +\infty\,,\qquad
\int |u|^p\,d\nu < \int |u|^p\,W\,d\leb^N < +\infty\,.
\]
If $\mathcal{M}\neq\emptyset$ then, for every 
$F:\R^k\rightarrow\R$ nondecreasing in each variable and lower 
semicontinuous, there exists a minimum $\mu\in\mathcal{M}$ of
\[
\left\{\mu\mapsto 
F(\lambda_1^p(\mu),\ldots,\lambda_k^p(\mu))
\right\}
\]
satisfying
\[
\int \Psi(V_\mu)\,d\leb^N = c \,.
\]
\end{theorem}
\begin{proof}
If 
\[
\int \Psi(V_{\underline{\mu}})\,d\leb^N = c\,,
\]
then $\underline{\mu}\in\mathcal{M}$ is a minimum with the 
required property.
If
\[
\int \Psi(V_{\underline{\mu}})\,d\leb^N > c\,,
\]
let $\R_{lsc}$ be the set $\R$ endowed with the topology of the
lower semicontinuity: a subset~$U$ of $\R$ is said to be 
open if $U=]s,+\infty[$ for some $s\in\Rb$.
\par
For every $\mu\in\mathcal{M}$, we have
$\lambda_m^p(\mu) < +\infty$ for all $m\geq 1$
by~\ref{hcmuW} of Proposition~\ref{prop:muW}.
Then the map
\[
\left\{\mu\mapsto (\lambda_1^p(\mu),\ldots,\lambda_k^p(\mu))
\right\}
\]
is continuous from $\mathcal{M}$ into $\R_{lsc}^k$ 
by Proposition~\ref{prop:muW} and the function $F$ is lower
semicontinuous from $\R_{lsc}^k$ into $\R$.
Therefore the functional
\[
\left\{\mu\mapsto 
F(\lambda_1^p(\mu),\ldots,\lambda_k^p(\mu))
\right\}
\]
is lower semicontinuous from $\mathcal{M}$ into $\R$.
\par
Let $\overline{\mu}\in\mathcal{M}$.
Since $F$ is nondecreasing in each variable, it is 
enough to restrict the minimization to
\[
\mathcal{N} = 
\left\{\mu\in\mathcal{M}:\,\,
\lambda_1^p(\mu) \leq \lambda_k^p(\overline{\mu})\right\}\,.
\]
Observe that, if $(\mu^{(n)})$ is a sequence in $\mathcal{N}$ 
locally $\gamma$-converging to $\mu$ in $\cmeas(\R^N)$, 
then $\lambda_1^p(\mu) < +\infty$ by~\ref{hamuW} of 
Proposition~\ref{prop:muW}, which implies that there exists 
$u\in W^{1,p}_c(\R^N)$ satisfying
\[
\int |u|^p\,d\mu < +\infty\,,\qquad
\int |u|^p\,d\nu < \int |u|^p\,W\,d\leb^N < +\infty
\]
by~\ref{halambdafin} of Proposition~\ref{prop:lambdafin}.
Combining this fact with Corollary~\ref{cor:lsc} and~\ref{hamuW} 
of Proposition~\ref{prop:muW}, it follows that 
$\mu\in\mathcal{N}$, so that $\mathcal{N}$ is a nonempty and 
closed subset of the metrizable and compact space $\cmeas(\R^N)$.
\par
Therefore, there exists a minimum $\mu\in \mathcal{M}$ of
\[
\left\{\mu\mapsto 
F(\lambda_1^p(\mu),\ldots,\lambda_k^p(\mu))
\right\} \,.
\]
If 
\[
\int \Psi(V_\mu)\,d\leb^N < c \,,
\]
define for $t\geq 0$
\begin{alignat*}{3}
&V^{(t)} 
&&= \Psi^{-1}\left[\min\left\{\Psi(V_{\underline{\mu}}),
\Psi(V_{\mu})+t\,\exp(-|x|^2)\right\}\right]\,,\\
&\mu^{(t)} 
&&= \infty_{\R^N\setminus A_{\underline{\mu}}}
+ \underline{\mu}_s + V^{(t)}\,\leb^N\,.
\end{alignat*}
Then $V_{\underline{\mu}} \leq V^{(t)} \leq V_\mu$ $\leb^N$-a.e. 
in $\R^N$ and $\underline{\mu}\leq \mu^{(t)} \leq \mu$.
Moreover, from~\ref{hbAVmu} and~\ref{hdAVmu} of
Proposition~\ref{prop:AVmu} we infer that
\[
V_{\mu^{(t)}} \geq V^{(t)}
\qquad\text{$\leb^N$-a.e. in $\R^N$}\,.
\]
In the case $t=0$, we have $V^{(0)}=V_\mu$ and $V_\mu\,\leb^N$ 
is $\sigma$-finite on $A_\mu$, whence $V_{\mu^{(0)}} = V^{(0)}$ 
$\leb^N$-a.e. in $\R^N$.
If $t>0$, let 
\[
R(t,n) = \Psi^{-1}\left(\Psi(+\infty) + t\,\exp(-n^2)\right)\,.
\]
Then we have
\[
V^{(t)} \leq \max\left\{V_{\underline{\mu}},R(t,n)\right\}
\qquad\text{$\leb^N$-a.e. in $B_n(0)$}\,,
\]
whence $A_{\mu^{(t)}} = A_{\underline{\mu}}$.
It follows that 
\[
V_{\mu^{(t)}} = V^{(t)}
\qquad\text{$\leb^N$-a.e. in $\R^N$, for all $t\geq 0$}\,.
\]
If we choose $t > 0$ such that
\[
\int \Psi(V_{\mu^{(t)}})\,d\leb^N = c \,,
\]
then $\mu^{(t)}$ is a minimum with the required property.
\end{proof}
\begin{theorem}
\label{thm:existmuA}
Let 
\[
0 < c \leq \leb^N(A_{\underline{\mu}})
\]
and denote by $\mathcal{M}$ the set of $\mu$'s in $\cmeas(\R^N)$ 
such that 
\[
\mu\geq\underline{\mu}\,,\qquad \leb^N(A_\mu) \leq c
\]
and such that there exists $u\in W^{1,p}_c(\R^N)$ satisfying
\[
\int |u|^p\,d\mu < +\infty\,,\qquad
\int |u|^p\,d\nu < \int |u|^p\,W\,d\leb^N < +\infty\,.
\]
If $\mathcal{M}\neq\emptyset$ then, for every 
$F:\R^k\rightarrow\R$ nondecreasing in each variable and lower
semicontinuous, there exists a minimum $\mu\in\mathcal{M}$ of
\[
\left\{\mu\mapsto 
F(\lambda_1^p(\mu),\ldots,\lambda_k^p(\mu))
\right\}
\]
satisfying
\[
\leb^N(A_\mu) = c \,.
\]
\end{theorem}
\begin{proof}
If $\leb^N(A_{\underline{\mu}})=c$, then $\underline{\mu}$
is a minimum with the required property.
Otherwise, assume that $\leb^N(A_{\underline{\mu}})>c$.
Arguing as in the proof of Theorem~\ref{thm:existmuV},
we find a minimum $\mu\in\mathcal{M}$ of
\[
\left\{\mu\mapsto 
F(\lambda_1^p(\mu),\ldots,\lambda_k^p(\mu))
\right\} \,.
\]
If $\leb^N(A_\mu) < c$, consider $\mu^{(r)}\in\cmeas(\R^N)$ 
defined by
\[
\mu^{(r)}(B) = \underline{\mu}(B\cap B_r(0)) +
\tilde{\mu}(B\setminus B_r(0))
\qquad\text{for all $B\in\bor(\R^N)$}\,,
\]
where $\tilde{\mu}$ is the outer regular representative
of $\mu$ given by Proposition~\ref{prop:equivout}.
Then it is easily seen that 
$\underline{\mu} \leq \mu^{(r)} \leq \mu$ and that
\[
A_{\mu}\cup(A_{\underline{\mu}}\cap B_r(0))
\subseteq A_{\mu^{(r)}} \,.
\] 
If $\mu^{(r)}(W) < +\infty$ for some Borel and
$p$-finely open $W$, there exists a Borel and 
$p$-quasi open $A$ such that $W\setminus B_r(0)\subseteq A$
and $\mu(A)<+\infty$, so that
\[
W\cap B_r(0) \subseteq  A_{\underline{\mu}}\cap B_r(0)\,,
\qquad 
\cp_p\left[(W\setminus B_r(0)) \setminus A_\mu\right] \leq
\cp_p\left(A \setminus A_\mu\right) = 0
\]
by~\ref{haAmu} of Proposition~\ref{prop:Amu}.
From the quasi-Lindel\"of property we infer that
\[
\cp_p\left[A_{\mu^{(r)}} \setminus \left(A_{\mu}\cup
(A_{\underline{\mu}}\cap B_r(0))\right)\right] = 0\,,
\]
whence
\[
\leb^N\left(A_{\mu^{(r)}}\right) =
\leb^N\left(A_{\mu}\cup(A_{\underline{\mu}}\cap B_r(0))\right)\,.
\]
If we choose $r>0$ such that
\[
\leb^N\left(A_{\mu^{(r)}}\right) = c \,,
\]
then $\mu^{(r)}$ is a minimum with the required property.
\end{proof}
Now we first consider the particular case in which
\[
\underline{\mu}(B) = 
\begin{cases}
\displaystyle{
\int_{B\cap\underline{A}} \underline{V}\,d\leb^N}
&\qquad\text{if $\cp_p(B\setminus\underline{A})=0$}\,,\\
\noalign{\medskip}
+\infty
&\qquad\text{if $\cp_p(B\setminus\underline{A})>0$}\,,
\end{cases}
\qquad\text{for all $B\in\bor(\R^N)$}\,.
\]
for some $p$-quasi open subset~$\underline{A}$ of $\R^N$
and some $p$-quasi upper semicontinuous function 
$\underline{V}:\underline{A}\rightarrow[0,+\infty]$.
\begin{corollary}
\label{cor:existV}
Let $\Psi:[0,+\infty]\rightarrow[0,+\infty]$ be a function as in 
Corollary~\ref{cor:lsc} and let 
\[
0 < c \leq\int_{\underline{A}} \Psi(\underline{V})\,d\leb^N \,.
\]
Denote by $\mathcal{V}$ the set of 
$\leb^N$-measurable functions 
$V:\underline{A}\rightarrow[0,+\infty]$ such that 
\[
\text{$V\geq \underline{V}$ \quad $\leb^N$-a.e. in 
$\underline{A}$,\qquad 
$\int_{\underline{A}} \Psi(V)\,d\leb^N \leq c$}
\]
and such that there exists 
$u\in W^{1,p}_0(\underline{A})$ satisfying
\[
\int_{\underline{A}} |u|^p\,V\,d\leb^N < +\infty\,,\qquad
\int_{\underline{A}} |u|^p\,d\nu 
< \int_{\underline{A}} |u|^p\,W\,d\leb^N < +\infty\,.
\]
If $\mathcal{V}\neq\emptyset$ then, for every 
$F:\R^k\rightarrow\R$ nondecreasing in each variable and lower
semicontinuous, there exists a minimum $V\in\mathcal{V}$ of
\[
\left\{V\mapsto 
F(\lambda_1^p(V),\ldots,\lambda_k^p(V)) 
\right\}
\]
satisfying
\[
\int_{\underline{A}} \Psi(V)\,d\leb^N = c \,,
\]
where $\lambda_m^p(V) = \lambda_m^p(\mu)$ with
\begin{equation}
\label{eq:V->mu}
\mu(B) = 
\begin{cases}
\displaystyle{
\int_{B\cap\underline{A}} V\,d\leb^N}
&\qquad\text{if $\cp_p(B\setminus\underline{A})=0$}\,,\\
\noalign{\medskip}
+\infty
&\qquad\text{if $\cp_p(B\setminus\underline{A})>0$}\,,
\end{cases}
\qquad\text{for all $B\in\bor(\R^N)$}\,.
\end{equation}
\end{corollary}
\begin{proof}
We aim to apply Theorem~\ref{thm:existmuV}.
Without loss of generality, we may assume that $\Psi(+\infty)=0$.
Consider $\underline{V}$ and each $V$ defined on all $\R^N$
with value $+\infty$ outside $\underline{A}$.
Then the definition of $\underline{\mu}$ and~\eqref{eq:V->mu}
can be reformulated as
\[
\underline{\mu}=\infty_{\R^N\setminus \underline{A}} +
\underline{V}\,\leb^N\,,\qquad
\mu=\infty_{\R^N\setminus \underline{A}} + V\,\leb^N\,.
\]
For every $u\in W^{1,p}_0(\underline{A})$ there exists a sequence
$(u_n)$ in $W^{1,p}_c(\R^N)$ converging to $u$ in $W^{1,p}(\R^N)$ 
with $|u_n| \leq |u|$ q.e. in $\R^N$.
Combining this fact with Proposition~\ref{prop:AVmu}, we see that,
if $V\in\mathcal{V}$ and $\mu$ is defined according 
to~\eqref{eq:V->mu}, then we have $\mu\in\mathcal{M}$.
On the other hand, if $\mu\in\mathcal{M}$ we infer again 
from Proposition~\ref{prop:AVmu} that $V_\mu\in\mathcal{V}$.
Moreover, by~\ref{heAVmu} of Proposition~\ref{prop:AVmu} we have
\[
c \leq\int \Psi(V_{\underline{\mu}})\,d\leb^N \,.
\]
Let $\mu\in\mathcal{M}$ be a minimum of
\[
\left\{\mu\mapsto 
F(\lambda_1^p(\mu),\ldots,\lambda_k^p(\mu))
\right\}
\]
according to Theorem~\ref{thm:existmuV}.
By Proposition~\ref{prop:AVmu}, since $F$ is nondecreasing in 
each variable, we have
\[
F(\lambda_1^p(V_\mu),\ldots,\lambda_k^p(V_\mu)) \leq
F(\lambda_1^p(\mu),\ldots,\lambda_k^p(\mu)) \leq
F(\lambda_1^p(V),\ldots,\lambda_k^p(V)) 
\quad\text{for all $V\in\mathcal{V}$}\,.
\]
Since $V_\mu\in\mathcal{V}$, the assertion follows.
\end{proof}
Then let us consider the particular case in which
$\underline{\mu} = \infty_{\R^N\setminus \underline{A}}$
for some $p$-quasi open subset~$\underline{A}$ of $\R^N$.
For every $p$-quasi open subset $A$ of $\R^N$ and $m\geq 1$, 
we set $\lambda_m^p(A) = \lambda_m^p(\infty_{\R^N\setminus A})$.
\par
In this case assumption~\ref{muW} reads:
\emph{
\begin{enumerate}[align=parleft]
\item[\mylabel{AW}{\AW}]
\begin{itemize}
\item
if $p< N$, we have
\[
\int_{\underline{A}\cap\{V_\nu < R W\}}
W^{N/p}\,d\leb^N < +\infty
\qquad\text{for all $R>0$}\,;
\]
\item
if $p=N$, we have
\[
\inf_{\varepsilon > 0}\,
\leb^N\left(\underline{A}\cap
\left\{V_\nu < \varepsilon\right\}\right) < +\infty
\]
and
\[
\int_{\underline{A}\cap\{V_\nu < R W\}}
\tau^*\left(W\right)\,d\leb^N < +\infty
\qquad\text{for all $R>0$}\,;
\]
\item
if $p>N$, we have
\[
\inf_{\varepsilon > 0}\,
\leb^N\left(\underline{A}\cap
\left\{V_\nu < \varepsilon\right\}\right) < +\infty
\]
and there exists $q\in[1,\infty[$ such that
\[
\int_{\underline{A}\cap\{V_\nu < R W\}}
W^q\,d\leb^N < +\infty
\qquad\text{for all $R>0$}\,.
\]
\end{itemize}
\end{enumerate}
}
In particular, assumption~\ref{AW} is satisfied
if $\leb^N\left(\underline{A}\right) < +\infty$
and $W=1$.
\begin{corollary}
\label{cor:AW}
The following facts hold:
\begin{enumerate}[label={\upshape\alph*)}, align=parleft, 
widest=iii, leftmargin=*]
\item[\mylabel{haAW}{\ha}]
if $(A^{(n)})$ is a sequence of $p$-quasi open subsets of 
$\underline{A}$ locally $\gamma$-converging to a $p$-quasi 
open subset $A$, then 
$\cp_p\left(A \setminus \underline{A}\right)=0$ and
\[
\lambda_m^p(A) \leq \liminf_{n\to\infty} \,
\lambda_m^p(A^{(n)})
\qquad\text{for all $m\geq 1$}\,;
\]
\item[\mylabel{hbAW}{\hb}]
for every $p$-quasi open subset $A$ of $\underline{A}$, 
the assumptions~\ref{hi} and~\ref{hii} of 
Section~\ref{sect:eigmeas} are satisfied by
$(\infty_{\R^N\setminus A},W\,\leb^N,\nu)$, in particular 
\[
\lambda_m^p(A) =
\lambda_m^p(\infty_{\R^N\setminus A},W\,\leb^N,\nu) =
\hat{\lambda}_m^p(\infty_{\R^N\setminus A},W\,\leb^N,\nu),
\qquad\text{ for all }m\geq 1\,;
\]
\item[\mylabel{hcAW}{\hc}]
for every $p$-quasi open subset $A$ of $\R^N$, we have
\[
\lambda_m^p(A) < +\infty \qquad\text{for all $m\geq 1$}
\]
if and only if there exists $u\in W^{1,p}_0(A)$ satisfying
\[
\int |u|^p\,d\nu < \int |u|^p\,W\,d\leb^N < +\infty\,;
\]
in particular, if $W=1$, $\nu=0$ and 
$0 < \leb^N(A) < +\infty$, then 
the eigenvalues $\lambda_m^p(A)$
agree with those defined by~\eqref{eq:vareig}.
\end{enumerate}
\end{corollary}
\begin{proof}
Taking into account Definition~\ref{defn:locconv}
and Proposition~\ref{prop:AVmu}, it is a 
particular case of Proposition~\ref{prop:muW}.
\end{proof}
\begin{corollary}
\label{cor:existA}
Let $c\in]0,\leb^N(\underline{A})]$ and denote by $\mathcal{A}$ 
the family of $p$-quasi open subsets $A$ of $\R^N$ such that 
\[
A\subseteq \underline{A}\,,\qquad
\leb^N(A) \leq c
\] 
and such that there exists $u\in W^{1,p}_0(A)$ satisfying
\[
\int |u|^p\,d\nu < \int |u|^p\,W\,d\leb^N < +\infty\,.
\]
If $\mathcal{A}\neq\emptyset$ then, for every 
$F:\R^k\rightarrow\R$ nondecreasing in each variable and
lower semicontinuous, there exists a minimum in $\mathcal{A}$ of
\[
\left\{A\mapsto F(\lambda_1^p(A),\ldots,\lambda_k^p(A))\right\}
\]
satisfying
\[
\leb^N(A) = c \,.
\]
\end{corollary}
\begin{proof}
We aim to apply Theorem~\ref{thm:existmuA}.
By Proposition~\ref{prop:AVmu}, if $A\in\mathcal{A}$ we have 
$\infty_{\R^N\setminus A}\in\mathcal{M}$.
On the other hand, if $\mu\in\mathcal{M}$ we infer again 
from Proposition~\ref{prop:AVmu} that 
$\cp_p(A_\mu\setminus\underline{A})=0$, whence
$A_\mu\cap\underline{A}\in\mathcal{A}$.
\par
Let $\mu\in\mathcal{M}$ be a minimum of
\[
\left\{\mu\mapsto 
F(\lambda_1^p(\mu),\ldots,\lambda_k^p(\mu))
\right\}
\]
according to Theorem~\ref{thm:existmuA}.
By Proposition~\ref{prop:AVmu}, since $F$ is nondecreasing in 
each variable, we have
\[
F(\lambda_1^p(A_\mu\cap\underline{A}),\ldots,
\lambda_k^p(A_\mu\cap\underline{A})) \leq
F(\lambda_1^p(\mu),\ldots,\lambda_k^p(\mu)) \leq
F(\lambda_1^p(A),\ldots,\lambda_k^p(A)) 
\quad\text{for all $A\in\mathcal{A}$}\,.
\]
Since $A_\mu\cap\underline{A}\in\mathcal{A}$, the assertion follows.
\end{proof}
\begin{proof}[Proof of Theorem~\ref{thm:main1}]
Let $\underline{A}=\Omega$, $W=1$ and $\nu=0$, so that
assumption~\ref{AW} is satisfied.
According to Corollary~\ref{cor:AW}, for every $p$-quasi open 
subset $A$ of $\Omega$ with $\leb^N(A)>0$, the eigenvalues 
$\lambda_m^p(A)$ agree with those defined by~\eqref{eq:vareig}.
Then the assertion follows from Corollary~\ref{cor:existA}.
\end{proof}
\begin{proof}[Proof of Theorem~\ref{thm:main2}]
We aim to apply Corollary~\ref{cor:existV}.
Without loss of generality, we may assume that
$\nu\in\cmeas(\R^N)$.
Let $\underline{A}=\Omega$, $\underline{V}=0$ and $W=1$,
so that $\underline{\mu} = \infty_{\R^N\setminus \underline{A}}$.
\emph{A fortiori} we have
$V_{\underline{\mu}+\nu}=+\infty$ on $\R^N\setminus\Omega$.
Since $\Omega$ has finite measure, it is easily seen that
assumption~\ref{muW} is satisfied.
Then the assertion follows.
\end{proof}

%--------------------------------------------------------------------

%\bibliography{nonlinear}

\begin{thebibliography}{10}

\bibitem{ambrosio_fusco_pallara2000}
{\sc L.~Ambrosio, N.~Fusco, and D.~Pallara}, {\em Functions of bounded
  variation and free discontinuity problems}, Oxford Mathematical Monographs,
  The Clarendon Press, Oxford University Press, New York, 2000.

\bibitem{ambrosio_tilli2004}
{\sc L.~Ambrosio and P.~Tilli}, {\em Topics on analysis in metric spaces},
  vol.~25 of Oxford Lecture Series in Mathematics and its Applications, Oxford
  University Press, Oxford, 2004.

\bibitem{antunes}
{\sc P.~R.~S. Antunes}, {\em Extremal {$p$}-{L}aplacian eigenvalues},
  Nonlinearity, 32 (2019), pp.~5087--5109.

\bibitem{bartsch1993}
{\sc T.~Bartsch}, {\em Topological methods for variational problems with
  symmetries}, vol.~1560 of Lecture Notes in Mathematics, Springer-Verlag,
  Berlin, 1993.
  
 \bibitem{brelot}
{\sc M.~Brelot}, {\em On topologies and boundaries in potential theory},
  Enlarged edition of a course of lectures delivered in 1966. Lecture Notes in
  Mathematics, Vol. 175, Springer-Verlag, Berlin-New York, 1971.

\bibitem{browder1983}
{\sc F.~E. Browder}, {\em Fixed point theory and nonlinear problems}, Bull.
  Amer. Math. Soc. (N.S.), 9 (1983), pp.~1--39.

\bibitem{bucc}
{\sc D.~Bucur}, {\em Uniform concentration-compactness for {S}obolev spaces on
  variable domains}, J. Differential Equations, 162 (2000), pp.~427--450.

\bibitem{bucur}
\leavevmode\vrule height 2pt depth -1.6pt width 23pt, {\em Minimization of the
  {$k$}-th eigenvalue of the {D}irichlet {L}aplacian}, Arch. Ration. Mech.
  Anal., 206 (2012), pp.~1073--1083.

\bibitem{bucur_buttazzo2005}
{\sc D.~Bucur and G.~Buttazzo}, {\em Variational methods in shape optimization
  problems}, Progress in Nonlinear Differential Equations and their
  Applications, 65, Birkh\"auser Boston, Inc., Boston, MA, 2005.

\bibitem{bm}
{\sc D.~Bucur and D.~Mazzoleni}, {\em A surgery result for the spectrum of the
  {D}irichlet {L}aplacian}, SIAM J. Math. Anal., 47 (2015), pp.~4451--4466.

\bibitem{bdm}
{\sc G.~Buttazzo and G.~Dal~Maso}, {\em An existence result for a class of
  shape optimization problems}, Arch. Rational Mech. Anal., 122 (1993),
  pp.~183--195.

\bibitem{bugeruve}
{\sc G.~Buttazzo, A.~Gerolin, B.~Ruffini, and B.~Velichkov}, {\em Optimal
  potentials for {S}chr\"{o}dinger operators}, J. \'{E}c. polytech. Math., 1
  (2014), pp.~71--100.

\bibitem{cantrellcosner}
{\sc R.~S. Cantrell and C.~Cosner}, {\em Spatial ecology via reaction-diffusion
  equations}, Wiley Series in Mathematical and Computational Biology, John
  Wiley \& Sons, Ltd., Chichester, 2003.

\bibitem{champion_depascale2007}
{\sc T.~Champion and L.~De~Pascale}, {\em Asymptotic behaviour of nonlinear
  eigenvalue problems involving {$p$}-{L}aplacian-type operators}, Proc. Roy.
  Soc. Edinburgh Sect. A, 137 (2007), pp.~1179--1195.

\bibitem{cingolani_degiovanni_vannella2018}
{\sc S.~Cingolani, M.~Degiovanni, and G.~Vannella}, {\em Amann-{Z}ehnder type
  results for {$p$}-{L}aplace problems}, Ann. Mat. Pura Appl. (4), 197 (2018),
  pp.~605--640.

\bibitem{cdm}
{\sc J.-N. Corvellec, M.~Degiovanni, and M.~Marzocchi}, {\em Deformation
  properties for continuous functionals and critical point theory}, Topol.
  Methods Nonlinear Anal., 1 (1993), pp.~151--171.

\bibitem{dalmaso1987}
{\sc G.~Dal~Maso}, {\em {$\Gamma$}-convergence and {$\mu$}-capacities}, Ann.
  Scuola Norm. Sup. Pisa Cl. Sci. (4), 14 (1987), pp.~423--464.

\bibitem{dmgammaconv}
\leavevmode\vrule height 2pt depth -1.6pt width 23pt, {\em An introduction to
  {$\Gamma$}-convergence}, vol.~8 of Progress in Nonlinear Differential
  Equations and their Applications, Birkh\"{a}user Boston, Inc., Boston, MA,
  1993.

\bibitem{dmmo}
{\sc G.~Dal~Maso and U.~Mosco}, {\em Wiener's criterion and
  {$\Gamma$}-convergence}, Appl. Math. Optim., 15 (1987), pp.~15--63.

\bibitem{dmmu}
{\sc G.~Dal~Maso and F.~Murat}, {\em Asymptotic behaviour and correctors for
  {D}irichlet problems in perforated domains with homogeneous monotone
  operators}, Ann. Scuola Norm. Sup. Pisa Cl. Sci. (4), 24 (1997),
  pp.~239--290.

\bibitem{degiovanni_lancelotti2007}
{\sc M.~Degiovanni and S.~Lancelotti}, {\em Linking over cones and nontrivial
  solutions for {$p$}-{L}aplace equations with {$p$}-superlinear nonlinearity},
  Ann. Inst. H. Poincar\'{e} Anal. Non Lin\'{e}aire, 24 (2007), pp.~907--919.

\bibitem{dema}
{\sc M.~Degiovanni and M.~Marzocchi}, {\em Limit of minimax values under
  {$\Gamma$}-convergence}, Electron. J. Differential Equations, No. 266 (2014),
  p.~19.

\bibitem{deimling1985}
{\sc K.~Deimling}, {\em Nonlinear functional analysis}, Springer-Verlag,
  Berlin, 1985.

\bibitem{ekeland_temam1999}
{\sc I.~Ekeland and R.~T\'{e}mam}, {\em Convex analysis and variational
  problems}, vol.~28 of Classics in Applied Mathematics, Society for Industrial
  and Applied Mathematics (SIAM), Philadelphia, PA, english~ed., 1999.
\newblock Translated from the French.

\bibitem{fadell_rabinowitz1977}
{\sc E.~R. Fadell and P.~H. Rabinowitz}, {\em Bifurcation for odd potential
  operators and an alternative topological index}, J. Functional Analysis, 26
  (1977), pp.~48--67.

\bibitem{fadell_rabinowitz1978}
\leavevmode\vrule height 2pt depth -1.6pt width 23pt, {\em Generalized
  cohomological index theories for {L}ie group actions with an application to
  bifurcation questions for {H}amiltonian systems}, Invent. Math., 45 (1978),
  pp.~139--174.

\bibitem{fuscop}
{\sc N.~Fusco, S.~Mukherjee, and Y.~R.-Y. Zhang}, {\em A variational
  characterisation of the second eigenvalue of the {$p$}-{L}aplacian on quasi
  open sets}, Proc. Lond. Math. Soc. (3), 119 (2019), pp.~579--612.

\bibitem{garcia_peral1987}
{\sc J.~P. Garc\'{\i}a~Azorero and I.~Peral~Alonso}, {\em Existence and
  nonuniqueness for the {$p$}-{L}aplacian: nonlinear eigenvalues}, Comm.
  Partial Differential Equations, 12 (1987), pp.~1389--1430.

\bibitem{hekima}
{\sc J.~Heinonen, T.~Kilpel\"{a}inen, and J.~Mal\'{y}}, {\em Connectedness in
  fine topologies}, Ann. Acad. Sci. Fenn. Ser. A I Math., 15 (1990),
  pp.~107--123.

\bibitem{hekima_book}
{\sc J.~Heinonen, T.~Kilpel\"{a}inen, and O.~Martio}, {\em Nonlinear potential
  theory of degenerate elliptic equations}, Dover Publications, Inc., Mineola,
  NY, 2006.
\newblock Unabridged republication of the 1993 original.

\bibitem{henrotbook}
{\sc A.~Henrot}, ed., {\em Shape optimization and spectral theory}, De Gruyter
  Open, Warsaw, 2017.
  
\bibitem{ioffelsc}
{\sc A.~D. Ioffe}, {\em On lower semicontinuity of integral functionals. {I}},
  SIAM J. Control Optimization, 15 (1977), pp.~521--538.

\bibitem{km}
{\sc T.~Kilpel\"{a}inen and J.~Mal\'{y}}, {\em Supersolutions to degenerate
  elliptic equation on quasi open sets}, Comm. Partial Differential Equations,
  17 (1992), pp.~371--405.

\bibitem{krasnoselskii1964}
{\sc M.~A. Krasnosel'skii}, {\em Topological methods in the theory of nonlinear
  integral equations}, Translated by A. H. Armstrong; translation edited by J.
  Burlak. A Pergamon Press Book, The Macmillan Co., New York, 1964.

\bibitem{li_ruf2008}
{\sc Y.~Li and B.~Ruf}, {\em A sharp {T}rudinger-{M}oser type inequality for
  unbounded domains in {$\mathbb{R}^n$}}, Indiana Univ. Math. J., 57 (2008),
  pp.~451--480.

\bibitem{mp}
{\sc D.~Mazzoleni and A.~Pratelli}, {\em Existence of minimizers for spectral
  problems}, J. Math. Pures Appl. (9), 100 (2013), pp.~433--453.

\bibitem{rabinowitz1986}
{\sc P.~H. Rabinowitz}, {\em Minimax methods in critical point theory with
  applications to differential equations}, vol.~65 of CBMS Regional Conference
  Series in Mathematics, Published for the Conference Board of the Mathematical
  Sciences, Washington, DC; by the American Mathematical Society, Providence,
  RI, 1986.

\bibitem{szulkin1988}
{\sc A.~Szulkin}, {\em Ljusternik-{S}chnirelmann theory on {${\it
  C}^1$}-manifolds}, Ann. Inst. H. Poincar\'{e} Anal. Non Lin\'{e}aire, 5
  (1988), pp.~119--139.

\bibitem{sw}
{\sc A.~Szulkin and M.~Willem}, {\em Eigenvalue problems with indefinite
  weight}, Studia Math., 135 (1999), pp.~191--201.

\end{thebibliography}
%%\bibliographystyle{abbrv}
%\bibliographystyle{siam}
%%
\end{document}